\documentclass[9pt,reqno]{amsart}
\usepackage[a4paper, margin=2.3cm]{geometry}

\usepackage{eulervm}

\usepackage{newcent}

\usepackage{ulem}
\usepackage{amsthm}
\usepackage{mathrsfs}

\usepackage{tabu}

\usepackage{lineno}

\usepackage{amsaddr}

\usepackage{float}
\usepackage[all]{xy}\usepackage{xypic}
\usepackage{braket}
\usepackage[colorlinks = true,citecolor=black]{hyperref}
\hypersetup{
  pdftitle   = {},
  pdfauthor  = {},
  pdfcreator = {\LaTeX\ with package \flqq hyperref\frqq}
}

\DeclareMathAlphabet{\mathpzc}{OT1}{pzc}{m}{it}
\usepackage{tikz-cd}
\usepackage{tkz-graph}
\usetikzlibrary{shapes.geometric,arrows, shapes}

\newtheorem{theorem}{Theorem}[section]
\newtheorem*{theorem*}{Theorem}
\newtheorem{theorem-non}{Theorem}
\newtheorem{proposition}[theorem]{Proposition}

\newtheorem*{lemma*}{Lemma}
\newtheorem{corollary}[theorem]{Corollary}
\newtheorem{conjecture}[theorem]{Conjecture}
\newtheorem*{conjecture*}{Conjecture}

\theoremstyle{definition}
\newtheorem{definition}[theorem]{Definition}
\newtheorem{example}[theorem]{Example}

\theoremstyle{remark}
\newtheorem{remark}[theorem]{Remark}

\DeclareMathOperator{\rank}{rank}

\numberwithin{equation}{section}

\begin{document}
\title[]{Hermitian non-K\"{a}hler structures on products of principal $S^{1}$-bundles over complex flag manifolds and applications in Hermitian geometry with torsion}

\author{Eder M. Correa}
\address{\resizebox{12cm}{.2cm}{ \textit{IMPA \ - \ Instituto de Matem\'{a}tica Pura e Aplicada,  Estr. Dona Castorina, 110, Rio de Janeiro, 22460-320, Brasil}} }
\thanks{Eder M. Correa was supported by CNPq grant 150899/2017-3}
\thanks{E-mail: \rm edermoraes@impa.br} 
\begin{abstract} 
In this paper we provide an explicit description of normal almost contact structures obtained from Cartan-Ehresmann connections (gauge fields) on principal $S^{1}$-bundles over complex flag manifolds. The main feature of our approach is to employ elements of representation theory of complex simple Lie algebras in order to describe and classify these structures. By following \cite{MORIMOTO}, we use these normal almost contact structures to explicitly describe a huge class of compact Hermitian non-K\"{a}hler manifolds obtained from products of principal $S^{1}$-bundles over complex flag manifolds. Moreover, by following \cite{Manjarin}, we obtain from our description several concrete examples of 1-parametric families of complex structures on products of principal $S^{1}$-bundles over flag manifolds, these concrete examples generalize the Calabi-Eckmann manifolds \cite{CALABIECKMANN}. Further, by following \cite{GRANTCHAROV}, as an application of our main results in the setting of KT structures on toric bundles over flag manifolds, we classify a huge class of explicit examples of Calabi-Yau structures with torsion (CYT) on certain Vaisman manifolds (generalized Hopf manifolds \cite{Vaisman}). Also as an application of our main results, we provide several new concrete examples of astheno-K\"{a}hler structures on products of compact homogeneous Sasaki manifolds. This last construction generalizes, in the homogeneous setting, the results introduced in \cite{Matsuo1} for Calabi-Eckmann manifolds. 
\end{abstract}

\maketitle

\hypersetup{linkcolor=black}

\hypersetup{linkcolor=black}


\section{Introduction}

\subsection{An overview for the reader}
In 1948, H. Hopf \cite{HOPF} gave the first examples of compact complex manifolds which are non-K\"{a}hler by showing that $S^{1} \times S^{2m+1}$ admits a complex structure for any positive integer $m$. These structures are obtained from the quotient of $\mathbb{C}^{m+1} \backslash \{0\}$ by a holomorphic and totally discontinuous action of $\mathbb{Z}$.

In 1953, Calabi and Eckmann \cite{CALABIECKMANN} showed that any product of spheres of odd dimension $S^{2n+1} \times S^{2m+1}$ (for $n,m \geq 0$) can be endowed with a structure of complex manifold. In order to achieve that, they considered the fibration 
\begin{equation}
S^{1} \times S^{1} \hookrightarrow S^{2n+1} \times S^{2m+1} \to \mathbb{C}P^{n} \times \mathbb{C}P^{m},
\end{equation}
where $\mathbb{C}P^{n}$  denotes the complex projective space of dimension $n$, and equipped the torus fiber of this bundle with a structure of elliptic curve. 

In 1963, A. Morimoto \cite{MORIMOTO} made a study of almost complex structures on the product space of almost contact manifolds \cite{GRAY}, \cite{SASAKIALMOST}. He showed that any product of almost contact manifolds can be endowed with an almost complex structure, and proved that the induced almost complex structure on the product is integrable if and only if both underlying almost contact structures are normal \cite{SASAKIHARAKEYMAALMOST}. By using the normal contact metric structure on $S^{2n+1}$ introduced in \cite{SASAKIHARAKEYMAONDIFF}, he generalized Calabi-Eckmann construction of complex structures on products of odd-dimensional spheres. Later, in 1980, by means of Morimoto's ideas, K. Tsukada showed in \cite{Tsukada} how to endow $S^{2n+1} \times S^{2m+1}$ with a 1-parametric family of complex structures which include Morimoto's complex structure as a particular case.

In \cite{LOEB}, J.-J. Loeb and M. Nicolau generalized Calabi-Eckmann and Hopf complex structures through the construction of a class of complex structures on the product $S^{2n+1} \times S^{2m+1}$ that contains the precedents. Similar techniques were used by S. L\'{o}pez de Medrano and A. Verjovsky in \cite{MEDRANO} to construct another family of non-K\"{a}hlerian compact manifolds, and later it was generalized by L. Meersseman in \cite{MEERSSEMAN}. In 2008, M. Manjar\'{i}n constructed 1-parameter families of complex structures by means of normal almost contact structures and CR-structures. These families include as particular cases the Hopf manifolds, the Calabi-Eckmann manifolds, and the complex structures on the product of two normal almost contact manifolds constructed by Morimoto and Tsukada.

Recently, in \cite{SANKARAN}, inspired by Loeb-Nicolau's construction \cite{LOEB}, Sankaran and Thakur obtained a family of complex structures on $S(L_{1})\times S(L_{2})$, where $S(L_{i}) \to X_{i}$, $i = 1, 2$, is the smooth principal $S^{1}$-bundle associated to a holomorphic principal $\mathbb{C}^{\times}$-bundle $L_{i} \to X_{i}$ over a complex manifold $X_{i}$, $i = 1, 2$.

In this paper, we study Morimoto's construction \cite{MORIMOTO}, and Manjar\'{i}n's construction \cite{Manjarin}, of complex structures, and 1-parameter families of complex structures, respectively, on products of normal almost contact manifolds provided by principal $S^{1}$-bundles over flag manifolds. Our main purpose is classify these complex structures by using elements of Lie theory which underlie the geometry of complex flag manifolds, such as representation theory of simple Lie algebras and {\textit{painted Dynkin diagrams}} \cite{Alekseevsky}. In order to do so, we develop some techniques introduced in \cite{CONTACTCORREA} to explicitly compute the Cartan-Ehresmann connections on principal $S^{1}$-bundles over complex flag manifolds. Then, we combine this with the ideas developed in \cite{MORIMOTO}, and \cite[Proposition 2.9]{Manjarin}, obtaining a concrete treatment in terms of differential forms, irreducible representations, and painted Dynkin diagrams for such constructions of complex structures. 

Also, by looking at the aspects of Hermitian geometry on manifolds provided by total spaces of principal torus bundles over flag manifolds \cite[Proposition 4.26]{Algmodels}, \cite{GRANTCHAROV}, \cite{Poddar}, as applications of our main results, we give a concrete description of Calabi-Yau connections with torsion (CYT structures) on certain Vaisman manifolds (generalized Hopf manifolds \cite{Vaisman}). Recently, CYT structures on non-K\"{a}hler manifolds have attracted attention as models for string compactifications, see for instance \cite{IVANOV} and references therein. By following \cite{GRANTCHAROV}, we provide a constructive method to describe a huge class of concrete examples of CYT structures which generalizes, in the homogeneous setting, some well-known constructions from Hopf manifolds $S^{2n+1}\times S^{1}$ to more general spaces.

Still in the setting of Hermitian geometry, we give also a concrete description of astheno-K\"{a}hler structures on products of compact homogeneous Sasaki manifolds. As pointed out in \cite{FINO}, until recently astheno-K\"{a}hler metrics were not receiving a
lot of attention due to the lack of examples. Actually, there are not many examples of astheno-K\"{a}hler manifolds, some of them are given by Calabi-Eckmann manifolds \cite{Matsuo1} and by nilmanifolds \cite{FINO2}. By following \cite{Tsukada}, \cite{Matsuo1}, and \cite{Matsuo2}, we obtain as an application of our main results several concrete new examples of astheno-K\"{a}hler manifolds which generalize the construction on Calabi-Eckmann manifolds $S^{2n+1} \times S^{2m+1}$ introduced in \cite{Matsuo1}. In this last case, as in the previous constructions described, our treatment take into account elements of representation theory of complex simple Lie algebras which control the projective algebraic geometry of complex flag manifolds. 

\subsection{Main results} Our main results can be organized as follows: 

\begin{enumerate}

\item Classification of normal almost contact structures obtained from Cartan-Ehresmann connections on principal circle bundles over flag manifolds;

\item Classification of 1-parametric families of complex structures on products of circle bundles over flag manifolds by using elements of Lie theory;

\item Explicit description by using elements of Lie theory of Calabi-Yau structures with torsion on certain compact Vaisman manifolds;

\item Explicit description, by using elements of Lie theory, of astheno-K\"{a}hler structures on products of compact homogeneous Sasaki manifolds.
\end{enumerate}

In what follows, we provide a brief description for all the results listed above. Our first result provides an improvement for Morimoto's result \cite[p. 432]{MORIMOTO}, which asserts that any simply connected homogeneous contact manifold admits a normal almost contact structure. Actually, we show that the normality condition holds for any principal $S^{1}$-bundle over a complex flag manifold by describing explicitly these normal almost contact structures in terms of the associated Cartan-Ehresmann connection. We recall that a complex flag manifold is a compact simply connected homogeneous Hodge manifold $X$ defined by
\begin{equation}
X = G^{\mathbb{C}}/P = G/G \cap P,
\end{equation}
where $G^{\mathbb{C}}$ is a complex simple Lie group with compact real form given by $G$, and $P \subset G^{\mathbb{C}}$ is a parabolic Lie subgroup. By considering ${\text{Lie}}(G^{\mathbb{C}}) = \mathfrak{g}^{\mathbb{C}}$, if we choose a Cartan subalgebra $\mathfrak{h} \subset \mathfrak{g}^{\mathbb{C}}$, and a simple root system $\Sigma \subset \mathfrak{h}^{\ast}$, up to conjugation, we have that $P = P_{\Theta}$, for some $\Theta \subset \Sigma$, where $P_{\Theta}$ is a parabolic Lie subgroup completely determined by $\Theta$, see for instance \cite{Alekseevsky}. In this last setting, our first result is the following.

\begin{theorem-non}
\label{Theo1}
Let $X_{P} = G^{\mathbb{C}}/P$ be a complex flag manifold associated to some parabolic Lie subgroup $P = P_{\Theta}$ of a complex simple Lie group $G^{\mathbb{C}}$. Then, given a principal $S^{1}$-bundle $Q \in \mathscr{P}(X_{P},{\rm{U}}(1))$, we have that

\begin{enumerate}

\item $Q = \displaystyle \sum_{\alpha \in \Sigma \backslash \Theta} Q(\ell_{\alpha}\omega_{\alpha})$, such that $\ell_{\alpha} \in \mathbb{Z}$, $\forall \alpha \in \Sigma \backslash \Theta$.

\item The manifold defined by the total space $Q$ admits a normal almost contact structure $(\phi,\xi = \frac{\partial}{\partial \theta},\eta)$, such that
\begin{equation}
\label{localalmostcontactform}
\eta = \displaystyle \frac{\sqrt{-1}}{2}\big ( \partial - \overline{\partial} \big ) \log \Big ( \prod_{\alpha \in \Sigma \backslash \Theta} ||s_{U}v_{\omega_{\alpha}}^{+}||^{2  \ell_{\alpha}}\Big) + d\theta_{U},
\end{equation}
where $s_{U} \colon U \subset X_{P} \to G^{\mathbb{C}}$ is a local section, and $v_{\omega_{\alpha}}^{+}$ is the highest weight vector with weight $\omega_{\alpha}$ associated to the fundamental irreducible $\mathfrak{g}^{\mathbb{C}}$-module $V(\omega_{\alpha})$, $\forall \alpha \in \Sigma \backslash \Theta$. Moreover, $\phi \in {\text{End}}(TQ)$ is completely determined by the horizontal lift of $\sqrt{-1}\eta$ and the canonical invariant complex structure $J_{0}$ of $X_{P}$. Furthermore, it satisfies $\pi_{\ast} \circ \phi = J_{0} \circ \pi_{\ast}$.

\item We have a Riemannian metric $g_{Q}$ on $Q$ such that 
\begin{equation}
g_{Q} = \pi^{\ast} \big (\omega_{X_{P}}({\rm{id}} \otimes J_{0}) \big ) + \eta \otimes \eta \ \ {\text{and}} \ \ \mathscr{L}_{\xi}g_{Q} = 0, 
\end{equation}
where $\omega_{X_{P}}$ is an invariant K\"{a}hler form on $X_{P}$.

\end{enumerate}

\end{theorem-non}

The result above combines \cite[Theorem 1]{HATAKEYMA}, and \cite[Theorem 6]{MORIMOTO}, with the recent results provided by the author in \cite{CONTACTCORREA}, and \cite{CALABICORREA}. The key point in our construction is to provide the precise description of
\begin{equation}
\label{isocircleline}
{\text{Pic}}(X_{P}) \cong \mathscr{P}(X_{P},{\rm{U}}(1)), \ \ L \mapsto Q(L),
\end{equation}
for any complex flag manifold $X_{P} = G^{\mathbb{C}}/P$, where $\mathscr{P}(X_{P},{\rm{U}}(1))$ is the set of isomorphism classes of principal $S^{1}$-bundles over $X_{P}$. 
Although the isomorphism \ref{isocircleline} is well-known, e.g. \cite{TOROIDAL}, the precise correspondence in terms of connections, curvature and characteristic classes has been recently described in \cite{CONTACTCORREA}. 

\begin{remark}
Let us recall that, given an almost contact manifold $M$ with structure tensors $(\phi, \xi,\eta)$, we can define an almost ${\rm{CR}}$-structure on $M$ by using its structure tensors as follows: consider $\mathscr{D} = \ker(\eta)$, and define 
\begin{equation}
{\mathcal{J}}_{\phi} \colon \mathscr{D}^{\mathbb{C}} \to \mathscr{D}^{\mathbb{C}},
\end{equation}
where $\mathscr{D}^{\mathbb{C}} = \mathscr{D}\otimes\mathbb{C}$, and ${\mathcal{J}}_{\phi}$ is the $\mathbb{C}$-linear extension of $\phi|_{\mathscr{D}}$. From this, it can be shown that the normality condition for $(\phi, \xi,\eta)$ implies the integrability of ${\mathcal{J}}_{\phi}$, so a normal almost contact manifold has a canonical ${\rm{CR}}$-structure, see for instance \cite[Theorem 6.6]{BLAIR}.
\end{remark}
By following the ideas developed in \cite{Manjarin}, and the result of Theorem \ref{Theo1}, we provide the following result which classifies 1-parametric families of complex structures on products of principal circle bundles over complex flag manifolds.

\begin{theorem-non}
\label{Theo2}
Let $X_{P_{i}}$ be a complex flag manifold associated to some parabolic Lie subgroup $P_{i} \subset G_{i}^{\mathbb{C}}$, such that $i = 1,2$. Then, given principal $S^{1}$-bundles $Q_{1} \in \mathscr{P}(X_{P_{1}},{\rm{U}}(1))$ and $Q_{2} \in \mathscr{P}(X_{P_{2}},{\rm{U}}(1))$, we have the following:

\begin{enumerate}

\item There exists a $1$-parametric family of complex structures $J_{\tau} \in {\text{End}}(T(Q_{1} \times Q_{2}))$ determined by $\mathcal{J}_{1} \oplus \mathcal{J}_{2}$, and by a complex valued $1$-form $\Psi_{\tau} \in \Omega^{1}(Q_{1} \times Q_{2}) \otimes \mathbb{C}$, which satisfies $J_{\tau}(\Psi_{\tau}) = \sqrt{-1} \Psi_{\tau}$, defined by
$$\Psi_{\tau} = \displaystyle{\frac{\sqrt{-1}}{2{\text{Im}}(\tau)} \Bigg \{\overline{\tau} \Bigg [d^{c}\log \Bigg ( \frac{1}{\displaystyle{\prod_{\alpha \in \Sigma_{1} \backslash \Theta_{1}} ||s_{U_{1}}v_{\omega_{\alpha}}^{+}||^{\ell_{\alpha}}}}\Bigg ) + d\theta_{U_{1}} \Bigg ] + d^{c}\log \Bigg ( \frac{1}{\displaystyle{\prod_{\beta \in \Sigma_{2} \backslash \Theta_{2}} ||s_{U_{2}}w_{\omega_{\beta}}^{+}||^{\ell_{\beta}}}}\Bigg ) + d\theta_{U_{2}} \Bigg \}},$$
 for some local section $s_{U_{i}} \colon U_{i} \subset X_{P_{i}} \to G_{i}^{\mathbb{C}}$, $i = 1,2$, such that $\tau \in \mathbb{C} \backslash \mathbb{R}$. 
\item Particularly, if $\tau = \sqrt{-1}$, we have Morimoto's complex structure  $J_{\sqrt{-1}} \in {\text{End}}(T(Q_{1} \times Q_{2}))$, i.e.,
\begin{equation}
\label{cplxstructure}
J_{\sqrt{-1}}(X,Y) = \big (\phi_{1}(X) - \eta_{2}(Y)\xi_{1}, \phi_{2}(Y) + \eta_{1}(X)\xi_{2} \big ),
\end{equation}
for all $(X,Y) \in T(Q_{1} \times Q_{2})$, where $(\phi_{i},\xi_{i},\eta_{i})$ is a normal almost contact structure on $Q_{i}$, $i =1,2$, as in Theorem \ref{Theo1}. Moreover, by considering the Riemannian metric $g_{Q_{1}} \times g_{Q_{2}}$, with $g_{Q_{i}}$ as in Theorem \ref{Theo1}, we obtain a Hermitian non-K\"{a}hler structure $(J_{\sqrt{-1}},g_{Q_{1}} \times g_{Q_{2}})$ on $Q_{1} \times Q_{2}$, with fundamental form $\Omega$ given by 
\begin{equation}
\label{khlstructure}
\Omega = \pi_{1}^{\ast} \omega_{X_{P_{1}}} + \pi_{2}^{\ast} \omega_{X_{P_{2}}} + \eta_{1} \wedge \eta_{2},
\end{equation}
where $\pi_{i} \colon Q_{i} \to X_{P_{i}}$, and $\omega_{X_{P_{i}}}$ is an invariant K\"{a}hler metric on $X_{P_{i}}$, $i = 1,2$. Furthermore, regarding the complex structure $J_{\sqrt{-1}} \in {\text{End}}(T(Q_{1} \times Q_{2}))$ described above, we have that the natural projection map
\begin{equation}
\label{morimototorus}
\pi_{1} \times \pi_{2} \colon \big (Q_{1}\times Q_{2},J_{\sqrt{-1}} \big ) \to  \big (X_{P_{1}}\times X_{P_{2}}, J_{1} \times J_{2} \big ),
\end{equation}
is holomorphic, where $J_{i}$ is an invariant complex structure on $X_{P_{i}}$, $i = 1,2$.
\end{enumerate}

\end{theorem-non}

The result of Theorem \ref{Theo2} together with the result of Theorem \ref{Theo1}, and the description of invariant K\"{a}hler metrics given in \cite{AZAD}, allow us to describe explicitly a huge class of complex manifolds and Hermitian non-K\"{a}hler manifolds. Furthermore, the last theorem also generalizes Morimoto's result \cite[p. 432]{MORIMOTO}, which in turn implies the construction of complex structures on products of odd-dimensional spheres.

As we have seen, the second part of Theorem \ref{Theo2} gives us a concrete description of Hermitian non-K\"{a}hler structures on principal $T^{2}$-bundles over Cartesian products of complex flag manifolds. In a more general setting, given a principal $G$-bundle $G \hookrightarrow M \to B$, over an almost complex manifold $B$, if we suppose that $G$ is a Lie group of even real dimension, then we can endow the manifold underlying the total space $M$ with an almost complex structure $\mathscr{J} \in {\text{End}}(TM)$, see for instance \cite[Proposition 4.25]{Algmodels}. A particular setting in which the last ideas becomes even more interesting is when $G = T^{2n}$. Actually, when the fiber of the principal bundle is an even dimensional torus, under the assumption of integrability for the complex structure on the base manifold, and that the characteristic classes of $M$ to be of $(1,1)$-type, one can show that the almost complex structure, as just mentioned, is in fact integrable, see for instance \cite[Proposition 4.26]{Algmodels}. This last description of complex structure on torus bundles turns out to be essentially the same complex structure obtained in Theorem \ref{Theo1} for the particular case when $G = T^{2}$.

Recently, many interesting results related to K\"{a}hler structures with torsion (a.k.a. KT structures) on torus bundle were introduced in \cite{GRANTCHAROV}, see also \cite{Poddar}. In order to contextualize the first application of our main results, let us recall some generalities about Hermitian geometry with torsion. 

We recall that a KT structure on a Hermitian manifold $(M,J,g)$ can be defined from the unique Hermitian connection $\nabla^{B}$, called in the literature as {\textit{Bismut connection}} or KT {\textit{connection}}, satisfying the conditions
\begin{equation}
\nabla^{B}g = 0 \ \ {\text{and}} \ \ \nabla^{B} J = 0.
\end{equation}
In this last setting, we call the triple $(g, J, \nabla^{B})$ a KT structure defined on $M$. By considering the fundamental $2$-form $\Omega = g(J \otimes {\rm{id}})$, associated to a KT structure $(g, J, \nabla^{B})$, we have
\begin{equation}
g(T_{\nabla^{B}}(X,Y),Z) = d\Omega(JX,JY,JZ) = (Jd\Omega)(X,Y,Z),
\end{equation}
$\forall X,Y,Z \in TM$, where $T_{\nabla^{B}}$ is the torsion of $\nabla^{B}$, e.g. \cite{Gauduchon1}. Thus, we can associate to a KT structure $(g, J, \nabla^{B})$ a torsion 3-form $T_{B} = Jd\Omega \in \Omega^{3}(M)$.

Now, we notice that, since $\nabla^{B}$ is a Hermitian connection, its (restricted) holonomy group ${\text{Hol}}^{0}(\nabla^{B})$ is in general contained in the unitary group ${\rm{U}}(n)$. If the restricted holonomy group of a KT connection $\nabla^{B}$ can be reduced to ${\rm{SU}}(n)$, the Hermitian structure is said to be Calabi-Yau with torsion (a.k.a. CYT), and, in this latter case, we call $(g, J, \nabla^{B})$ a CYT structure.

In the context of complex structures on torus bundle, by following the results of \cite{GRANTCHAROV}, as an application of Theorem \ref{Theo1} we obtain the following theorem:

\begin{theorem-non}
\label{CYThomo1}
Let $X_{P}$ be a complex flag manifold, with real dimension $2m$, associated to some parabolic Lie subgroup $P \subset G^{\mathbb{C}}$, and let $I(X_{P})$ be its Fano index. Then the manifold $M = Q(L) \times {\rm{U}}(1)$, such that $L = K_{X_{P}}^{\otimes \frac{\ell}{I(X_{P})}}$, $\ell > 0$, admits a CYT structure $(g_{M},\mathscr{J},\nabla^{B})$ whose fundamental form $\Omega_{M} = g_{M}(\mathscr{J} \otimes {\rm{id}})$ is given by 
\begin{equation}
\label{CYTstructure1}
\Omega_{M} = \frac{m\ell}{I(X_{P})} d\eta + \eta \wedge d\sigma,
\end{equation}
such that $\sqrt{-1} d\sigma \in \Omega^{1}({\rm{U}}(1);\sqrt{-1}\mathbb{R})$ is the Maurer-Cartan form, and (locally)
\begin{equation}
\eta =  \frac{1}{2I(X_{P})}d^{c}\log \Big (\big | \big |s_{U}v_{\delta_{P}}^{+} \big| \big |^{2} \Big ) + d\theta_{U},
\end{equation}
for some local section $s_{U} \colon U \subset X_{P} \to G^{\mathbb{C}}$, where $v_{\delta_{P}}^{+}$ denotes the highest weight vector of weight $\delta_{P}$ associated to the irreducible $\mathfrak{g}^{\mathbb{C}}$-module $V(\delta_{P})$.
\end{theorem-non}
\begin{remark}
\label{HopfVaisman}
Notice that manifolds defined by Cartesian products of the form $Q \times S^{1}$, where $Q$ is a compact Sasakian manifold, are particular examples of {\textit{Vaisman manifolds}}, or {\textit{generalized Hopf manifolds}}, see \cite{Vaisman}, \cite{Structurevaisman}. Notice also that in this last setting we have a locally conformal K\"{a}hler structure on $Q \times S^{1}$ obtained from the globally conformal structure associated to the K\"{a}hler covering given by the metric cone
\begin{equation}
\label{conesymplectic}
\big (\mathscr{C}(Q) = Q \times \mathbb{R}^{+}, \omega_{\mathscr{C}} = \frac{1}{2}d(r^{2}\eta) \big),
\end{equation}
where $\eta \in \Omega^{1}(Q)$ is the underlying contact structure, for more details see \cite{Structurevaisman}, \cite{Dragomir}. It is straightforward to show that the Hermitian structure $(g_{M},\mathscr{J})$ which underlies the CYT structure $(g_{M},\mathscr{J},\nabla^{B})$ defined on $M = Q(L) \times {\rm{U}}(1)$ in Theorem \ref{CYTstructure1} is in fact Vaisman, i.e., locally conformally K\"{a}hler with parallel\footnote{Here the parallelism required is with respect to the Levi-Civita connection $\nabla$ associated to $g_{M}$, see for instance \cite{Vaisman}.} {\textit{Lee form}}, see Corollary \ref{CYTVaisman}.
\end{remark}

As we have seen, from Theorem \ref{CYThomo1} we obtain a quite constructive method to produce KT structures $(g_{M},\mathscr{J},\nabla^{B})$ which are explicit solutions of the equation
\begin{equation}
{\text{Ric}}^{\nabla^{B}}(\Omega_{M}) = 0,
\end{equation}
i.e., the last theorem provides a constructive method to produce explicit examples of KT structures which satisfies ${\text{Hol}}^{0}(\nabla^{B}) \subset {\rm{SU}}(m+1)$. 

In the theoretical physics context, CYT structures was considered first by A. Strominger \cite{Strominger} and C. Hull \cite{Hull}. According to \cite{IVANOV}, for compactifications of string theory with non-vanishing torsion, it is required that the connection with torsion has holonomy contained in ${\rm{SU}}(n)$, $n = 2,3,4$, ${\rm{G_{2}}}$, ${\rm{Spin}}(7)$.

\begin{table}[H]
\caption{ Table with examples of complex Hermitian non-K\"{a}hler manifolds with CYT structure provided by Theorem \ref{CYThomo1}.}
\begin{tabular}{|c|c|c|c|c|}
\hline
               &         &             &            &     \\  
${\text{Hol}}^{0}(\nabla^{B}) \subset {\rm{SU(2)}}$ &    $ {\text{Hol}}^{0}(\nabla^{B}) \subset {\rm{SU(3)}}$     &   $ {\text{Hol}}^{0}(\nabla^{B}) \subset {\rm{SU(4)}}$  &   $ {\text{Hol}}^{0}(\nabla^{B}) \subset {\rm{SU(5)}}$       &   $ {\text{Hol}}^{0}(\nabla^{B}) \subset {\rm{SU}} \big (|\Pi^{+}| + 1 \big )$  \\
               &         &             &            &             \\         
\hline
               &         &              &           &           \\  
    $S^{3} \times S^{1}$ &     $S^{5} \times S^{1}$  &     $X_{1,1} \times S^{1}$  & $\mathscr{V}_{2}(\mathbb{R}^{6}) \times S^{1}$      &  $Q(K_{G/T})\times S^{1}$  \\
               &         &              &             &        \\  
\hline
    
\end{tabular}

\end{table}
\begin{remark}
In \cite{GRANTCHAROV1} it was shown that any compact complex homogeneous space with vanishing first Chern class, after a suitable deformation in the complex structure, admits a homogeneous Calabi-Yau structure with torsion (CYT), provided that it also has an invariant volume form. The proof of this result uses Guan's result \cite{Guan}, the torus bundle construction of the CYT structures provided in \cite{GRANTCHAROV} and the result \cite[Theorem 3]{GRANTCHAROV1}, related to the existence of CYT structures on certain C-spaces \cite{WANG} with vanishing first Chern class. 
\end{remark}
Our next result is also an application of Theorem \ref{Theo1} and Theorem \ref{Theo2}. We recall that a Hermitian manifold $(M,J,g)$ is said to be {\textit{astheno-K\"{a}hler}} if its fundamental $2$-form $\Omega = g(J \otimes {\rm{id}})$ satisfies 
\begin{equation}
\label{astheno}
dd^{c}\Omega^{n-2} = 0,
\end{equation}
where $\dim_{\mathbb{C}}(M) = n$. In \cite{Jost}, Jost and Yau used the condition \ref{astheno} to study Hermitian harmonic maps and to extend Siu's Rigidity Theorem \cite{Siu} to non-K\"{a}hler manifolds. Since then, many other results related to astheno-K\"{a}hler manifolds have been established, see for instance \cite{Liu}, \cite{LiYau}, \cite{Tosatti}. However, as pointed out in \cite{FINO}, there are not many examples of astheno-K\"{a}hler manifolds, some of them are given by Calabi-Eckmann manifolds \cite{Matsuo1} and by nilmanifolds \cite{FINO2}. Thus, by following \cite{Tsukada}, \cite{Matsuo1}, \cite{Matsuo2}, our next result combines Theorem \ref{Theo1}, Theorem \ref{Theo2}, and the ideas introduced in \cite{CONTACTCORREA}, in order to provide a systematic method to obtain explicit examples of astheno-K\"{a}hler manifolds. These manifolds are given by Cartesian products of compact homogeneous Sasaki manifolds, with Hermitian structure completely determined by elements of Lie theory. The result is the following:

\begin{theorem-non}
\label{asthenoflag}
Let $Q_{i} = Q(L_{i})$ be a compact homogeneous Sasaki manifold with structure tensors $(\phi_{i},\xi_{i},\eta_{i},g_{i})$, such that $L_{i}^{-1} \in {\text{Pic}}(X_{P_{i}})$ is an ample line bundle, for some $P_{i} = P_{\Theta_{i}} \subset G_{i}^{\mathbb{C}}$, $i = 1,2$. Then we have that:
\begin{enumerate}

\item The manifold $M = Q(L_{1}) \times Q(L_{2})$ admits a $1$-parametric family of Hermitian structures $(\Omega_{a,b},J_{a,b})$, $a + \sqrt{-1}b \in \mathbb{C} \backslash \mathbb{R}$, completely determined by principal $S^{1}$-connections $\sqrt{-1}\eta_{i} \in \Omega^{1}(Q(L_{i});\sqrt{-1}\mathbb{R})$, $i = 1,2$, such that 
\begin{equation}
\displaystyle \eta_{i} = \frac{1}{2}d^{c} \log \Big ( ||s_{U_{i}}v_{\lambda(L_{i})}^{+}||^{2}\Big) + d\theta_{U_{i}},
\end{equation}
for some local section $s_{U_{i}} \colon U_{i} \subset X_{P_{i}} \to G_{i}^{\mathbb{C}}$, where $v_{\lambda(L_{i})}^{+}$ is the highest weight vector of weight $\lambda(L_{i})$ associated to the irreducible $\mathfrak{g}^{\mathbb{C}}$-module $V(\lambda(L_{i}))$, $i = 1,2$; 

\item Moreover, the Hermitian structure $(\Omega_{a,b},J_{a,b})$ is astheno-K\"{a}hler if and only if the real constants $a,b \in \mathbb{R}$ satisfy 
\begin{equation}
\label{asthenocond}
m_{\Theta_{1}}\big(m_{\Theta_{1}} - 1\big) + 2am_{\Theta_{1}}m_{\Theta_{2}} + m_{\Theta_{2}}\big(m_{\Theta_{2}} - 1\big)\big(a^{2} + b^{2}\big) = 0,
\end{equation}
where $m_{\Theta_{i}} = | \Pi_{i}^{+} \backslash \langle \Theta_{i} \rangle^{+} |$, $i = 1,2$, and $m_{\Theta_{1}} + m_{\Theta_{2}} + 1 > 3$.
\end{enumerate}

\end{theorem-non}

It is worth mentioning that, under the hypotheses of theorem above, the complex structure $J_{a,b} \in {\text{End}}(T(Q(L_{1}) \times Q(L_{2})))$, $a + \sqrt{-1}b \in \mathbb{C} \backslash \mathbb{R}$, is given by Tsukada's complex structure \cite{Tsukada}. It is defined by
\begin{equation}
\displaystyle J_{a,b} = \phi_{1} - \bigg ( \frac{a}{b}\eta_{1} + \frac{a^{2} + b^{2}}{b} \eta_{2}\bigg) \otimes \xi_{1} + \phi_{2} + \bigg ( \frac{1}{b}\eta_{1} + \frac{a}{b} \eta_{2}\bigg) \otimes \xi_{2}.
\end{equation}
After a suitable change, the complex structure above is the same provided in Theorem \ref{Theo2}. The complex structure above has a compatible Riemannian metric $g_{a,b} = \Omega_{a,b}({\rm{id}} \otimes J_{a,b})$, such that $\Omega_{a,b} \in \Omega^{2}(Q(L_{1}) \times Q(L_{2}))$ is given by
\begin{equation}
\Omega_{a,b} = \frac{1}{2}\Big (d\eta_{1} + d\eta_{2} \Big ) + b\eta_{1} \wedge \eta_{2}.
\end{equation}
From \cite{Matsuo1}, we have that $dd^{c}\Omega_{a,b}^{n-2} = 0$, $n = m_{\Theta_{1}} + m_{\Theta_{2}} + 1 > 3$, if and only if Equation \ref{asthenocond} holds. We also note that, from Theorem \ref{asthenoflag}, we obtain a huge class of examples that naturally generalize the construction provided in \cite{Matsuo1} for Calabi-Eckmann manifolds.

In 1984, Gauduchon made the following conjecture.
\begin{conjecture}[{\cite[IV.5]{Gauduchon}}]
\label{Gauduchonconjecture}
Let $M$ be a compact complex manifold. Let $\psi$ be a closed real $(1,1)$-form on $M$ satisfying $[\psi] = c_{1}^{BC}(M) \in H_{BC}^{1,1}(M,\mathbb{R})$. Then there exists a Gauduchon metric $\omega$ on $M$, i.e., satisfying $\partial \overline{\partial}(\omega^{n-1}) = 0$, such that 
\begin{equation}
{\text{Ric}}(\omega) = \psi,
\end{equation}
where ${\text{Ric}}(\omega)$ is the Chern-Ricci curvature, given locally by ${\text{Ric}}(\omega) = -\sqrt{-1}\partial \overline{\partial}\log \omega^{n}$.
\end{conjecture}
Here we consider the Bott-Chern cohomology group 
\begin{center}
$H_{BC}^{1,1}(M,\mathbb{R}) = \displaystyle \frac{\big \{d{\text{-closed real}} \ (1,1){\text{-forms}}\big\}}{\big\{\sqrt{-1}\partial \overline{\partial}f, \ f\in C^{\infty}(M)\big\}},$
\end{center}
and the first Bott-Chern class $c_{1}^{BC} \colon {\text{Pic}}(M) \to H_{BC}^{1,1}(M,\mathbb{R})$ (cf. \cite{BottChern}). Gauduchon's conjecture \ref{Gauduchonconjecture} is a natural extension of the celebrated Calabi conjecture \cite{Yau} to compact complex manifolds. Recently, Tosatti and Weinkove showed in \cite[Corollary 1.4]{Tosatti} that Conjecture \ref{Gauduchonconjecture} holds for the class of compact Hermitian non-K\"{a}hler manifolds satisfying the astheno-K\"{a}hler condition \ref{astheno}. The general result has been proved by Sz\'{e}kelyhidi, Tosatti, and Weinkove in \cite{STW}. From Theorem \ref{asthenoflag}, we have a constructive method to describe a huge class of explicit examples of compact complex manifolds which can be used to illustrate the solution of Gauduchon's Conjecture \ref{Gauduchonconjecture} (cf. Remark \ref{gauduchonatheno}).

It is worth pointing out that the manifolds described in Theorem \ref{asthenoflag} do not admit any K\"{a}hler structure. Actually, according to \cite[Theorem 2.13]{SANKARAN}, the manifolds considered in the last theorem above do not admit any symplectic structure. Hence, complex manifolds obtained from the Cartesian product of compact homogeneous Sasaki manifolds can not be algebraic.

\subsection{Outline of the paper} The content and main ideas of this paper are organized as follows: In Section \ref{sec2}, we shall cover some generalities about almost contact manifolds and contact manifolds. In this section we also describe Morimoto's construction \cite{MORIMOTO} of almost complex structures on products of almost contact manifolds and Manjar\'{i}n's construction of 1-parametric families of complex structures on products of normal almost contact manifolds. In Section \ref{sec3}, we provide a description of connections and curvatures on holomorphic line bundles and principal $S^{1}$-bundles over complex flag manifolds, our approach is mainly based on \cite{TOROIDAL}, \cite{CONTACTCORREA}, and \cite{BLAIR}. In Section \ref{sec4}, we shall prove our main results, namely, Theorem \ref{Theo1} and Theorem \ref{Theo2}, and provide examples. In Section \ref{section5}, we explore some applications of our main results in the study of Hermitian geometry with torsion on principal torus bundles over flag manifolds. The main purpose of this last section is to prove Theorem \ref{CYThomo1} and Theorem \ref{asthenoflag}.

\section{Almost complex structures on products of almost contact manifolds}
\label{sec2} 

\subsection{Almost contact manifolds} Let us recall some basic facts and generalities on almost contact geometry. Our approach is based on \cite{SASAKIALMOST}, \cite{BLAIR}.

\begin{definition}
An almost contact manifold is a (2n+1)-dimensional smooth manifold $M$ endowed with structure tensors $(\phi, \xi,\eta)$, such that $\phi \in {\text{End}}(TM)$, $\xi \in \Gamma(TM)$, and $\eta \in \Omega^{1}(M)$, satisfying
\begin{equation}
\label{almostcontact}
 \phi \circ \phi = - {\rm{id}} + \eta \otimes \xi, \ \ \eta(\xi) = 1.    
\end{equation}
\end{definition}

\begin{remark}
Given an almost contact manifold $M$ with structure tensors $(\phi, \xi,\eta)$, we can show from \ref{almostcontact} the following additional properties:
\begin{center}
$\phi(\xi) = 0$, \ \ $\eta \circ \phi = 0$ \ \ and \ \ $\rank(\phi) = 2n$.
\end{center}
Even though these additional properties can be derived from \ref{almostcontact}, see for instance \cite[Theorem 4.1]{BLAIR}, many authors include the properties above in the definition of almost contact manifolds.
\end{remark}
In the setting of almost contact manifolds we have the concept of normality which is characterized by the equation
 \begin{equation}
\label{normality}
\big [ \phi,\phi \big ] + d\eta \otimes \xi = 0,
\end{equation}
where $[\phi,\phi]$ is the Nijenhuis torsion of $\phi$.
\begin{definition}[Sasaki and Hatakeyama, \cite{SASAKIHARAKEYMAALMOST}]
An almost contact manifold $M$ with structure tensors $(\phi,\xi,\eta)$ which satisfy \ref{normality} is called normal almost contact manifold.
\end{definition}
\begin{remark}
Given an almost contact manifold $M$ with structure tensors $(\phi,\xi,\eta)$, we can consider the manifold defined by $M \times \mathbb{R}$. We denote a vector field on $M \times \mathbb{R}$ by $(X,f\frac{d}{dt})$, where $X$ is tangent to $M$, $t$ is the coordinate on $\mathbb{R}$, and $f \in C^{\infty}(M \times \mathbb{R})$. From this, we can define an almost complex structure on $M \times \mathbb{R}$ by setting
\begin{equation}
\label{complexcone}
J \Big (X,f\frac{d}{dt} \Big ) = \Big (\phi(X) - f \xi, \eta(X) \frac{d}{d t} \Big ).
\end{equation}
By following \cite{SASAKIHARAKEYMAALMOST}, we can show that 

\begin{center}
$\big [ \phi,\phi \big ] + d\eta \otimes \xi = 0 \Longleftrightarrow \big [ J,J \big ] = 0.$
\end{center}
Thus, we have the normality condition for $(\phi,\xi,\eta)$ equivalent to the integrability condition for the almost complex structure $J$ defined in \ref{complexcone}. 
\end{remark}

A special context on which we have a natural normal almost contact structure is provide by the following result.

\begin{theorem}[\cite{MORIMOTO}, \cite{HATAKEYMA}]
\label{almostcircle}
Let $M$ be the total space of a principal ${\rm{U}}(1)$-bundle over a complex manifold $(N,J)$. Suppose we have a connection $1$-form $\sqrt{-1}\eta$ on $M$ such that $d\eta = \pi^{\ast}\omega$, here $\pi$ denotes the projection of $M$ onto $N$, and $\omega$ is a $2$-form on $N$ satisfying 

\begin{center}

$\omega(JX,JY) = \omega(X,Y),$

\end{center}
for $X,Y \in \Gamma(TN)$. Then, we can define a $(1,1)$-tensor field $\phi$ on $M$ and a vector field $\xi$ on $M$ such that $(\phi,\xi,\eta)$ is a normal almost contact structure on $M$.
\end{theorem}

\begin{proof}
The complete proof for this result can be found in \cite{MORIMOTO}, \cite{HATAKEYMA}. Let us briefly outline the main ideas involved. Consider $\xi = \frac{\partial}{\partial \theta} \in \Gamma(TM)$ as being the vector field defined by the infinitesimal action of $\mathfrak{u}(1)$ on $M$ and let $\sqrt{-1}\eta \in \Omega^{1}(M;\mathfrak{u}(1))$ be the connection $1$-form such that $d\eta = \pi^{\ast}\omega$. Without loss of generality, we can suppose that $\eta(\xi) = 1$.

Now, we define $\phi \in {\text{End}}(TM)$ by setting
\begin{center}
$ \phi(X) := \begin{cases}
    (J\pi_{\ast}X)^{H}, \ \ \ {\text{if}}  \ \ X \bot \xi.\\
    \ \ \ \ \  0 \  \  \ \ \ \ \ ,  \ \  \ {\text{if}} \ \ X \parallel \xi .                   \\
  \end{cases}$
    
\end{center}
Here we denote by $(J\pi_{\ast}X)^{H}$ the horizontal lift of $J\pi_{\ast}X$ relative to the connection $\sqrt{-1}\eta \in \Omega^{1}(M;\mathfrak{u}(1))$. A straightforward computation shows that $(\phi, \xi, \eta)$ defines an almost contact structure. For the normality condition, we just need to check that

\begin{center}

$\big [ J,J \big ] \equiv 0$ and $\omega \in \Omega^{1,1}(N) \implies \big [ \phi,\phi \big ] + d\eta \otimes \xi = 0,$

\end{center}
the details of the implication above can be found in \cite[Theorem 6]{MORIMOTO}.
\end{proof}

An important result which will be useful for us is the following.

\begin{theorem}[\cite{HATAKEYMA}]
\label{metrichatakeyma}
Let $\pi \colon M \to N$ be a principal circle bundle over an almost complex manifold $(N,J_{N})$. Then we have an almost contact structure $(\phi,\xi,\eta)$ on $M$ which satisfies 

\begin{enumerate}

\item $\sqrt{-1}\eta$ is a connection $1$-form on $M$.

\item $\phi$ is invariant under the transformations generated by the infinitesimal transformation $\xi = \frac{\partial}{\partial \theta}$.

\item $\pi_{\ast} \circ \phi = J_{N} \circ \pi_{\ast}$.
\end{enumerate}
Moreover, if $N$ is an almost Hermitian manifold with metric $g_{N}$, then we can also define an associated Riemannian metric $g_{M}$ on $M$ such that
\begin{equation}
\label{hatakeymaeq}
g_{M}  = \pi^{\ast} g_{N} + \eta \otimes \eta \ \ and \ \ \mathscr{L}_{\xi}g_{M} = 0,
\end{equation}
i.e. $\xi$ is a Killing vector field on $(M,g_{M})$.

\end{theorem}

\subsection{Contact geometry and almost contact geometry} An important class of almost contact manifolds is provided by contact manifolds. Let us recall some basic generalities on contact geometry.
\begin{definition}
Let $M$ be a smooth connected manifold of dimension $2n + 1$. A contact structure on $M$ is a $1$-form $\eta \in \Omega^{1}(M)$ which satisfies $\eta \wedge (d\eta)^{n} \neq 0$.
\end{definition}
When a smooth connected $(2n+1)$-dimensional manifold $M$ admits a contact structure $\eta \in \Omega^{1}(M)$ the pair $(M,\eta)$ is called contact manifold. Given a contact manifold $(M,\eta)$, at each point $p \in M$ we have from the condition $\eta \wedge (d\eta)^{n} \neq 0$ that $(d\eta)_{p}$ is a quadratic form of rank $2n$ in the Grassman algebra $\bigwedge T_{p}^{\ast}M$, thus we obtain 
\begin{equation}
T_{p}M = \mathscr{D}_{p} \oplus \mathcal{F}_{\eta_{p}},
\end{equation}
such that $\mathscr{D} = \ker(\eta)$ and
\begin{center}
$p \in M \mapsto \mathcal{F}_{\eta_{p}} = \big\{ X \in T_{p}M \ \big | \ (d\eta)_{p}(X,T_{p}M) = 0 \big \} \subset T_{p}M,$ 
\end{center}
defines the characteristic foliation.

Let $(M,\eta)$ be a contact manifold. From the condition $\eta \wedge (d\eta)^{n} \neq 0$, we have that there exists $\xi \colon C^{\infty}(M) \to C^{\infty}(M)$, such that 
\begin{equation}
\label{derivation}
df \wedge (d\eta)^{n} = \xi(f)\eta \wedge (d\eta)^{n},
\end{equation}
$\forall f \in  C^{\infty}(M)$. From this, a straightforward computation shows that $\xi$ is a $\mathbb{R}$-linear derivation on $C^{\infty}(M)$, hence $\xi \in \Gamma(TM)$. Now, from Equation \ref{derivation} we can show that 
\begin{center}
$\beta(\xi)\eta \wedge (d\eta)^{n} = \beta \wedge (d\eta)^{n}$, 
\end{center}
$\forall \beta \in \Omega^{1}(M)$. By using this last fact we have
\begin{center}
$\eta(\xi)\eta \wedge (d\eta)^{n} = \eta \wedge (d\eta)^{n}$ \ \ and \ \ $d\eta(X,\xi)\eta \wedge (d\eta)^{n} = \displaystyle \frac{1}{n+1}\iota_{X}(d\eta)^{n+1}=0,$ 
\end{center}
for all $X \in \Gamma(TM)$. Therefore, we obtain $\xi \in \Gamma(TM)$ which satisfies
\begin{equation}
\eta(\xi) = 1 \ \ {\text{and}} \ \ d\eta(\xi,\cdot) = 0,
\end{equation}
see for instance \cite{TAKIZAWA} for more details about the description above. The vector field $\xi$ is called the characteristic vector field, or Reeb vector field, of the contact structure $\eta$.

A contact structure $\eta \in \Omega^{1}(M)$ is called regular if the associated characteristic vector field $\xi \in \Gamma(TM)$ is regular, namely, if every point of the manifold has a neighborhood such that any integral curve of the vector field passing through the neighborhood passes through only once \cite{PALAIS}. In this case $(M,\eta)$ is called regular contact manifold. 

In the setting of compact regular contact manifolds we have the following important well-known result.

\begin{theorem}[Boothby-Wang, \cite{BW}]
\label{BWT}
Let $\eta$ be a regular contact structure on a compact smooth manifold $M$, then:

\begin{enumerate}

    \item $M$ is a principal ${\rm{U}}(1)$-bundle over $N = M/{\rm{U}}(1)$,\\
    
    \item $\eta' = \sqrt{-1}\eta$ defines a connection on this bundle, and\\
    
    \item the manifold $N$ is a symplectic manifold whose the symplectic form $\omega$ determines an integral cocycle on $N$ which satisfies $d\eta = \pi^{\ast}\omega$, where $\pi \colon M \to N$.
\end{enumerate}

\end{theorem}

The next result states that, in fact, the converse of Theorem \ref{BWT} is also true. 

\begin{theorem}[Kobayashi, \cite{TOROIDAL}]
\label{CONVBW}
Let $(N,\omega_{N})$ be a symplectic manifold such that $[\omega_{N}] \in H^{2}(N,\mathbb{Z})$, then there exists a principal ${\rm{U}}(1)$-bundle $\pi \colon M \to N$ with a connection $1$-form $\eta' \in \Omega^{1}(M,\mathfrak{u}(1))$ which determines a regular contact structure $\eta = -\sqrt{-1}\eta'$ on $M$ which satisfies $d\eta = \pi^{\ast}\omega_{N}$.

\end{theorem}
We are particularly interested in the following setting. 

\begin{definition} 
\label{contactdef}
A contact manifold $(M,\eta)$ is said to be homogeneous if there is a connected Lie group $G$ acting transitively and
 effectively as a group of diffeomorphisms on $M$ which leave $\eta$ invariant, i.e. $g^{\ast}\eta = \eta$, $\forall g \in G$.

\end{definition}

We denote a homogeneous contact manifold by $(M,\eta,G)$. From this, we have the following important result.

\begin{theorem}[Boothby-Wang, \cite{BW}]
Let $(M,\eta,G)$ be a homogeneous contact manifold. Then the contact form $\eta$ is regular. Moreover, $M = G/K$ is a fiber bundle over $G/H_{0}K$ with fiber $H_{0}K/K$, where $H_{0}$ is the connected component of a $1$-dimensional Lie group $H$, and $H_{0}$ is either diffeomorphic to ${\rm{U}}(1)$ or $\mathbb{R}$.

\end{theorem}

If we suppose that $(M,\eta,G)$ is compact and simply connected, then according to \cite{MONTGOMERY}, without loss of generality, we can suppose that $G$ is compact. Furthermore, according to \cite{WANG} we can in fact suppose that $G$ is a semisimple Lie group. Hence, we have the following theorem.

\begin{theorem}[Boothby-Wang, \cite{BW}]
\label{BWhomo}
Let $(M,\eta,G)$ be a compact simply connected homogeneous contact manifold. Then $M$ is a circle bundle over a complex flag manifold $(N,\omega_{N})$ such that $\omega_{N}$ defines a $G$-invariant Hodge metric which satisfies $d\eta = \pi^{\ast}\omega_{N}$, where $\pi \colon M \to (N,\omega_{N})$.
\end{theorem}

The next result shows that in the setting of Theorem \ref{BWhomo} we have a complete description of compact simply connected homogeneous contact manifolds in terms of negative line bundles over flag manifolds. Given a complex manifold $N$, for every $L \in {\text{Pic}}(N)$ we can take a Hermitian structure $H$ on $L$ and consider the circle bundle defined by 
\begin{equation}
Q(L) = \Big \{ u \in L \ \Big | \ \sqrt{H(u,u)} = 1 \Big\}.
\end{equation}
Now, we have the following characterization for compact simply connected homogeneous contact manifolds, see for instance \cite{CONTACTCORREA}. 
\begin{proposition}
\label{CONTACTSIMPLY}
Let $(M,\eta,G)$ be a compact simply connected homogeneous contact manifold, then $M$ is a principal ${\rm{U}}(1)$-bundle $\pi \colon Q(L) \to N$ over a complex flag manifold, for some ample line bundle $L^{-1} \in {\text{Pic}}(N)$. Moreover, if $d\eta = \pi^{\ast}\omega_{N}$, where $[\omega_{N}] \in H^{2}(N,\mathbb{Z})$ is a $G$-invariant K\"{a}hler-Einstein metric, it follows that
\begin{equation}
L = K_{N}^{\otimes \frac{1}{I(N)}},
\end{equation}
where $I(N)$ is the Fano index of $N$. 
\end{proposition}

The next result together with the last proposition allow us to describe all compact homogeneous contact manifolds, the proof for the result below can be found in \cite{BOYERGALICKI}.

\begin{theorem}
\label{contacthomo}
Let $(M,\eta,G)$ be a compact homogeneous contact manifold. Then
\begin{enumerate}
    
    \item $M$ is a non-trivial circle bundle over a complex flag manifold,\\
    
    \item $M$ has finite fundamental group, and the universal cover $\widetilde{M}$ of $M$ is a compact homogeneous contact manifold. 
\end{enumerate}

\end{theorem}
The result above provides a complete description for any compact homogeneous contact manifold $(M,\eta,G)$ as being a quotient space 

\begin{center}
    
$M = \widetilde{M}/\Gamma$,     
    
\end{center}
where $\widetilde{M} = Q(L)$ is given by Proposition \ref{CONTACTSIMPLY} and $\Gamma = \mathbb{Z}_{\ell} \subset {\rm{U}}(1) \hookrightarrow \widetilde{M}$ is a cyclic group given by the deck transformations of the universal cover $\widetilde{M}$, see for instance \cite{BOYERGALICKI}. Hence, we have 

\begin{center}

$M = Q(L)/\mathbb{Z}_{\ell} = \ell \cdot Q(L) = \underbrace{Q(L) + \cdots + Q(L)}_{\ell-{\text{times}}}$,    

\end{center}
see for instance \cite{TOROIDAL}, \cite[Chapter 2]{BLAIR}. In this paper we also shall use the notation $M = Q(L^{\otimes \ell})$.

Therefore, under the assumption of the Einstein condition in the associated Boothby-Wang fibration $\pi \colon (M,\eta) \to (N,\omega_{N})$, i.e. ${\text{Ric}}(\omega_{N}) = k\omega_{N}$, $k \in \mathbb{Z}_{>0}$, we have

\begin{center}

$N = G^{\mathbb{C}}/P = G/G \cap P$ \ \ and \ \ $M = Q(K_{N}^{\otimes \frac{\ell}{I(N)}})$,

\end{center}
where $G^{\mathbb{C}}$ is a complexification of $G$, $P \subset G^{\mathbb{C}}$ is a parabolic Lie subgroup, and $I(N)$ is the Fano index of $N$.
\begin{remark}
 As we can see, it is suitable to denote $N = X_{P}$ in order to emphasize the parabolic Lie subgroup $P \subset G^{\mathbb{C}}$.
\end{remark}

 \subsection{Morimoto's construction of almost complex structures} 
 \label{subsec2.3}
 In what follows we shall cover some basic results concerned to the construction of almost complex structures on products of almost contact manifolds, our approach is according to \cite{MORIMOTO}.
 
 Let $M_{1}$ and $M_{2}$ be almost contact manifolds with structure tensors $(\phi_{1},\xi_{1},\eta_{1})$ and $(\phi_{2},\xi_{2},\eta_{2})$, respectively. For any $X \in TM_{1}$ and $Y \in TM_{2}$, we can define
 \begin{equation}
 \label{almostcomplexstruct}
 J(X,Y) = \big (\phi_{1}(X) - \eta_{2}(Y)\xi_{1}, \phi_{2}(Y) + \eta_{1}(X)\xi_{2} \big ),
 \end{equation}
it is straightforward to check that $J\circ J = -{\rm{id}}$. From this, we have the following proposition:

\begin{proposition}
Let $M_{1}$ and $M_{2}$ be almost contact manifolds. Then $M_{1} \times M_{2}$ admits an almost complex structure induced by the almost contact structure of $M_{1}$ and $M_{2}$.
\end{proposition}
 
We have the following characterization of the integrability condition for the almost complex structure \ref{almostcomplexstruct}

\begin{theorem}[Morimoto, \cite{MORIMOTO}]
\label{MORIMOTOTHEOREM}
Under the hypotheses of the last proposition, the almost complex structure $J$ is integrable if and only if both $(\phi_{1},\xi_{1},\eta_{1})$ and $(\phi_{2},\xi_{2},\eta_{2})$ are normal. 
\end{theorem}

An interesting corollary of Theorem \ref{MORIMOTOTHEOREM} is the following result of Calabi and Eckmann \cite{CALABIECKMANN}.

\begin{corollary}
\label{CALABIECKMANNCORO}
The manifold $S^{2n+1}\times S^{2m+1}$ admits a complex structure.
\end{corollary}

\begin{remark} 
\label{fundamentalremark}
Notice that if we have principal circle bundles $M_{1}$ and $M_{2}$ over almost Hermitian manifolds $(N_{i},J_{N_{i}},g_{N_{i}})$, $i=1,2$, respectively, then we have an almost Hermitian structure $(g,J)$ on $M_{1} \times M_{2}$, such that $J$ is defined as in \ref{almostcomplexstruct} and
\begin{equation}
g((X,Y),(Z,W)) = g_{M_{1}}(X,Z) + g_{M_{2}}(Y,W),
\end{equation}
where $g_{M_{i}}$ is the Riemannian metric on $M_{i}$, $i = 1,2$, obtained from Theorem \ref{metrichatakeyma}. Moreover, by considering the almost contact structures $(\phi_{1},\xi_{1},\eta_{1})$ and $(\phi_{2},\xi_{2},\eta_{2})$ of $M_{1}$ and $M_{2}$, respectively, a straightforward computation shows that the fundamental 2-form $\Omega = g(J \otimes {\rm{id}})$ is given by
\begin{equation}
\label{fundformproduct}
\Omega = \pi_{1}^{\ast} \omega_{N_{i}} + \pi_{2}^{\ast} \omega_{N_{2}} + \eta_{1} \wedge \eta_{2},
\end{equation}
where $\omega_{N_{i}} = g_{N_{i}}(J_{N_{i}} \otimes {\rm{id}})$, $i = 1,2$, $\eta_{1}$ and $\eta_{2}$ are taken as 1-forms extended to the product $M_{1}\times M_{2}$.

\end{remark}

The result provided by Corollary \ref{CALABIECKMANNCORO} can be understood in terms of Lie theory as follows. Associated to each odd-dimensional sphere we have a complex Hopf fibration, i.e.,

\begin{center}

${\rm{U}}(1) \hookrightarrow S^{2n+1} \to \mathbb{C}P^{n}$ \ \ and \ \ ${\rm{U}}(1) \hookrightarrow S^{2m+1} \to \mathbb{C}P^{m}.$

\end{center}
Now, notice that both principal ${\rm{U}}(1)$-bundles above can be endowed with a normal almost contact structure. Actually, it follows from Theorem \ref{almostcircle}, Theorem \ref{BWhomo}, and Proposition \ref{CONTACTSIMPLY} that any circle bundle
\begin{center}
${\rm{U}}(1) \hookrightarrow Q(K_{N}^{\otimes \frac{\ell}{I(N)}}) \to N = X_{P},$
\end{center}
can be endowed with a normal almost contact structure. Therefore, since for every complex Hopf fibration ${\rm{U}}(1) \hookrightarrow S^{2n+1} \to \mathbb{C}P^{n}$ we have
\begin{center}
$S^{2n+1} = Q(K_{\mathbb{C}P^{n}}^{\otimes \frac{1}{n+1}})$ \ \ and \ \ $\mathbb{C}P^{n} = {\rm{SU}}(n+1)/{\rm{SU}}(n)\times {\rm{U}}(1),$
\end{center}
it follows that the complex manifold obtained by the product of two odd-dimensional spheres is a particular example of Morimoto's construction \ref{MORIMOTOTHEOREM}.

The comment above leads to the following generalization for the Calabi and Eckmann construction. Consider the following principal ${\rm{U}}(1)$-bundles

\begin{center}

${\rm{U}}(1) \hookrightarrow Q(K_{N_{1}}^{\otimes \frac{\ell_{1}}{I(N_{1})}}) \to N_{1} = X_{P_{1}}$ \ \ and \ \ ${\rm{U}}(1) \hookrightarrow Q(K_{N_{2}}^{\otimes \frac{\ell_{2}}{I(N_{2})}}) \to N_{2} = X_{P_{2}}.$

\end{center}
From Theorem \ref{MORIMOTOTHEOREM} we have a compact complex manifold defined by the product
\begin{center}
$Q(K_{N_{1}}^{\otimes \frac{\ell_{1}}{I(N_{1})}}) \times  Q(K_{N_{2}}^{\otimes \frac{\ell_{2}}{I(N_{2})}}).$
\end{center}
In the setting above, if we take invariant K\"{a}hler structures $\omega_{N_{i}}$ on $N_{i}$, $i = 1,2$, then we have from Remark \ref{fundamentalremark} that
\begin{equation}
d\Omega = d\eta_{1} \wedge \eta_{2} - \eta_{1} \wedge d\eta_{2},
\end{equation}
where $\Omega$ is defined as in \ref{fundformproduct}. Thus, it follows that
\begin{equation}
\label{notkahler}
\big (Q(K_{N_{1}}^{\otimes \frac{\ell_{1}}{I(N_{1})}}) \times  Q(K_{N_{2}}^{\otimes \frac{\ell_{2}}{I(N_{2})}}), \Omega, J \big ),
\end{equation}
defines a compact Hermitian manifold which is not K\"{a}hler.

\subsection{Manjar\'{i}n's construction of one-parameter family of complex structures} In this subsection we shall give a brief description of the construction of one-parameter family of complex structures on the product of normal almost contact manifolds. The method which we will present is essentially the content of \cite[Proposition 2.9]{Manjarin}. For the sake of compatibility to \cite{Manjarin}, let us introduce some basic facts related to ${\rm{CR}}$-structures.

\begin{definition}
\label{defCR}
A ${\rm{CR}}$-manifold is a differentiable manifold $M$ endowed with a complex subbundle $T^{(1,0)}M$ of the complexified tangent bundle $TM\otimes \mathbb{C}$, which satisfies $T^{(1,0)}M \cap \overline{T^{(1,0)}M} = \{0\}$ and the Frobenius (formal) integrability property
\begin{equation}
\label{integrability}
\big [\Gamma(T^{(1,0)}M), \Gamma(T^{(1,0)}M)\big ] \subseteq \Gamma(T^{(1,0)}M).
\end{equation}
The subbundle $T^{(1,0)}M$ satisfying the properties above is called a ${\rm{CR}}$-structure on $M$. When $T^{(1,0)}M$ does not satisfy the integrability condition \ref{integrability}, we call $T^{(1,0)}M$ an almost ${\rm{CR}}$-structure on $M$.
\end{definition}

Let $M$ be an almost contact manifold with structure tensors $(\phi, \xi,\eta)$. We can define an almost ${\rm{CR}}$-structure on $M$ by using its structure tensors as follows: consider $\mathscr{D} = \ker(\eta)$ and define 
\begin{equation}
\label{complexCR}
{\mathcal{J}}_{\phi} \colon \mathscr{D}^{\mathbb{C}} \to \mathscr{D}^{\mathbb{C}},
\end{equation}
where $\mathscr{D}^{\mathbb{C}} = \mathscr{D}\otimes\mathbb{C}$, and ${\mathcal{J}}_{\phi}$ is the $\mathbb{C}$-linear extension of $\phi|_{\mathscr{D}}$. From this, we set
\begin{equation}
\label{AlmostCR}
T^{(1,0)}M := \big \{ X \in \mathscr{D}^{\mathbb{C}} \ \ \big | \ \ {\mathcal{J}}_{\phi}(X) = \sqrt{-1}X \big \}.
\end{equation}
It is straightforward to check that $T^{(1,0)}M \cap \overline{T^{(1,0)}M} = \{0\}$, thus we obtain an almost ${\rm{CR}}$-structure on $M$. 

In general, $T^{(1,0)}M$ defined as in \ref{AlmostCR} may fail to be integrable. By a result of S. Ianus \cite{Ianus}, see also \cite[Theorem 6.6]{BLAIR}, we have that 
\begin{equation}
\big [ \phi,\phi \big ] + d\eta \otimes \xi = 0 \Longrightarrow \big [\Gamma(T^{(1,0)}M), \Gamma(T^{(1,0)}M)\big ] \subseteq \Gamma(T^{(1,0)}M),
\end{equation}
where $T^{(1,0)}M$ is given by \ref{AlmostCR}. Thus, a normal almost contact manifold is always a ${\rm{CR}}$-manifold. If we have a Riemannian metric $g$ on $M$ compatible with $(\phi, \xi,\eta)$, in the sense that
\begin{equation}
\label{contactmetric}
g(\phi(X),\phi(Y)) = g(X,Y) - \eta(X)\eta(Y),
\end{equation}
$\forall X,Y \in TM$, we call $(M,\phi, \xi,\eta,g)$ an almost contact metric manifold. The necessary and sufficient condition for a contact metric manifold $(M,\phi, \xi,\eta,g)$ to be a ${\rm{CR}}$-manifold were provided by S. Tanno \cite{TANNO}, see also \cite[Theorem 6.7]{BLAIR}. Actually, a contact metric manifold might be ${\rm{CR}}$ without the structure $(\phi, \xi,\eta)$ being normal.

Given a ${\rm{CR}}$-manifold $(M,T^{(1,0)}M)$, we can define a real subbundle $\mathscr{D}$ of $TM$ by setting
\begin{equation}
\label{realCR}
 \mathscr{D}:= TM \cap \big (T^{(1,0)}M \oplus \overline{T^{(1,0)}M}\big).
\end{equation}
The subbundle $\mathscr{D}$ defines the \textit{Levi distribution} associated to $T^{(1,0)}M$, see for instance \cite{SORIN}. Further, we can define $\mathcal{J} \colon \mathscr{D} \to \mathscr{D}$, by imposing that 
\begin{equation}
\label{complexcontact}
X - \sqrt{-1}\mathcal{J}(X) \in T^{(1,0)}M, \ \ \forall X \in \mathscr{D}.
\end{equation}
By taking $T^{(0,1)}M = \overline{T^{(1,0)}M}$, we obtain
\begin{center}
$\mathscr{D}\otimes \mathbb{C} = T^{(1,0)}M \oplus T^{(0,1)}M$.
\end{center}
If we consider the $\mathbb{C}$-linear extension of $\mathcal{J}$, we have  
\begin{center}
$T^{(1,0)}M = \big \{ X \in \mathscr{D}^{\mathbb{C}} \ \ \big | \ \ {\mathcal{J}}(X) = \sqrt{-1}X \big \}$ \ \ and \ \ $T^{(0,1)}M = \big \{ X \in \mathscr{D}^{\mathbb{C}} \ \ \big | \ \ {\mathcal{J}}(X) = -\sqrt{-1}X \big \},$
\end{center}
see for instance \cite{BOGGESS}. As we can see from the description above, the ${\rm{CR}}$-structure $T^{(1,0)}M$ is completely determined by $(\mathscr{D},\mathcal{J})$. We denote by ${\text{Aut}}_{{\rm{CR}}}(M)$ the subset of ${\text{Diff}}(M)$ of maps such that $f_{\ast} \colon TM \to TM$ preserves $\mathscr{D}$ and commutes with $\mathcal{J}$. Let $\{\varphi_{t} \ | \ t \in \mathbb{R}\}$ be the flow induced by a smooth $\mathbb{R}$-action on $M$. We say that $\{\varphi_{t}\}$ defines a ${\rm{CR}}$-action if $\varphi_{t} \in {\text{Aut}}_{{\rm{CR}}}(M)$, $\forall t \in \mathbb{R}$. When $\dim_{\mathbb{R}}(M) = 2n + 1$, we call the action {\textit{transverse}} to the ${\rm{CR}}$-structure if the smooth vector field $T \in \Gamma(TM)$, defined by
\begin{equation}
\label{characteristicCR}
T(p) = \displaystyle \frac{d}{dt} \Big |_{t=0}\varphi_{t}(p),
\end{equation}
$\forall p \in M$, is everywhere transverse to $\mathscr{D}$, i.e., $TM = \mathscr{D} \oplus \langle T \rangle$.

From the ideas above we have the following (equivalent) alternative definition of normal almost contact structure.

\begin{definition}
\label{normalCR}
A normal almost contact structure on a manifold $M$ of odd-dimension is a pair $(T^{(1,0)}M,\varphi_{t})$ where $T^{(1,0)}M$ is a ${\text{CR}}$-structure and $\{\varphi_{t}\}$ is a flow induced by a smooth $\mathbb{R}$-action which is transverse to the ${\text{CR}}$-structure $T^{(1,0)}M$. Given a normal almost contact structure $(T^{(1,0)}M,\varphi_{t})$ we define its characteristic $1$-form $\eta$ by the conditions 
\begin{equation}
\label{conditionalmostcontact}
\eta(T) = 1 \ \ and \ \ \ker(\eta) = \mathscr{D},
\end{equation}
where $T \in \Gamma(TM)$ is defined by the flow $\{\varphi_{t}\}$ as in \ref{characteristicCR}, and $\mathscr{D}= TM \cap (T^{(1,0)}M \oplus \overline{T^{(1,0)}M})$. Therefore, we also denote a normal almost contact structure by $(T,\mathscr{D},\eta)$.
\end{definition}

\begin{remark}
\label{nacsequivalence}
In the setting above, given a normal almost contact structure $(T^{(1,0)}M,\varphi_{t})$ on a manifold $M$ of odd-dimension, we can recover the structure tensors as in \ref{almostcontact} as follows: consider the associated Levi distribution $\mathscr{D}= T^{(1,0)}M \cap TM$, and the tensor $\mathcal{J} \colon \mathscr{D} \to \mathscr{D}$ as in \ref{complexcontact}. We define a $(1,1)$-tensor $\phi_{\mathcal{J}} \in {\text{End}}(TM)$ from an extension of $\mathcal{J}$ by requiring that 
\begin{center}
$\phi_{\mathcal{J}}(T) = 0.$
\end{center}
It follows from \ref{conditionalmostcontact} that $\phi_{\mathcal{J}} \circ \phi_{\mathcal{J}} = - {\rm{id}} + \eta \otimes T$. Thus, by considering the structure tensors $(\phi_{\mathcal{J}},T,\eta)$ on $M$, we obtain an almost contact structure in the sense of  \ref{almostcontact}. The normality condition \ref{normality} follows from the integrability condition \ref{integrability}.
\end{remark}
\begin{proposition}[M. Manjar\'{i}n, \cite{Manjarin}] 
\label{Manjarincomplex}
Let $M_{1}$ and $M_{2}$ be two manifolds endowed with normal almost contact structures $(T_{1},\mathscr{D}_{1},\eta_{1})$ and $(T_{2},\mathscr{D}_{2},\eta_{2})$, respectively. Then, there exists a $1$-parametric family of complex structures $J_{\tau}$  on the product $M_{1}\times M_{2}$, for $\tau \in \mathbb{C}\backslash \mathbb{R}$, so that the complex manifold $M_{1} \times M_{2}$ admits a nowhere vanishing holomorphic vector field $v_{\tau}$.
\end{proposition}

The $1$-parametric family of complex structures $J_{\tau}$ on $M_{1} \times M_{2}$ obtained from the proposition above is defined by imposing that 
\begin{equation}
T^{(1,0)}(M_{1}\times M_{2}) = T^{(1,0)}M_{1} \oplus T^{(1,0)}M_{2} \oplus \langle v_{\tau} := T_{1} - \tau T_{2} \rangle_{\mathbb{C}}.
\end{equation}
The integrability of $J_{\tau} \in {\text{End}}(T(M_{1}\times M_{2}))$ follows from the fact that:
\begin{center}
$\big [ T^{(1,0)}M_{i},T^{(1,0)}M_{i}\big ] \subseteq T^{(1,0)}M_{i}$ \ \ and \ \ $\big [T_{i},T^{(1,0)}M_{i} \big ] \subseteq T^{(1,0)}M_{i}$, \ \ $i = 1,2.$ 
\end{center}

\begin{remark}
\label{complexform}
Another way to characterize the complex structure $J_{\tau}$ above is the following. Recall that a differential form $\omega \in \Omega^{\bullet}(M)$ is called basic with respect to a foliation $\mathcal{F}$ if $\mathscr{L}_{X}\omega = \iota_{X}d\omega = 0$, for every vector field $X$ tangent to the leaves of $\mathcal{F}$. By considering the complex structures $\mathcal{J}_{i} \colon \mathscr{D}_{i} \to \mathscr{D}_{i}$, $i = 1,2$, we have 
\begin{center}
$\mathscr{D}_{i}\otimes \mathbb{C} = T^{(1,0)}M_{i} \oplus T^{(0,1)}M_{i}$, $i = 1,2.$
\end{center}
Thus, since by hypothesis both ${\text{CR}}$-structures are normal, it follows that 
\begin{equation}
T(M_{1} \times M_{2}) = \mathscr{D}_{1} \oplus \mathscr{D}_{2} \oplus \langle T_{1},T_{2} \rangle_{\mathbb{R}},
\end{equation}
and $\mathcal{J}_{1} \oplus \mathcal{J}_{2} \in {\text{End}}(\mathscr{D}_{1} \oplus \mathscr{D}_{2})$ is integrable. Therefore, we have that the foliation generated by $T_{1}$ and $T_{2}$ is transversely holomorphic, notice that $[T_{1},T_{2}] = 0$. Now, by taking the complex valued $1$-form
\begin{equation}
\label{Characterization}
\Psi_{\tau} = \frac{\sqrt{-1}}{2{\text{Im}}(\tau)} \big (\overline{\tau}\eta_{1} + \eta_{2} \big ),
\end{equation}
it is not difficult to see that $\Psi(v_{\tau}) = 1$ and $\Psi_{\tau}(\overline{v_{\tau}}) = 0$. Moreover, we have that $d\Psi_{\tau}$ is basic with respect to the foliation generated by $T_{1}$ and $T_{2}$, and  $d\Psi_{\tau}$ is of $(1,1)$-type with respect to the transverse holomorphic structure induced by $\mathcal{J}_{1} \oplus \mathcal{J}_{2}$. Thus, we have a complex structure by taking an extension $J_{\tau} \in {\text{End}}(T(M_{1}\times M_{2}))$ of $\mathcal{J}_{1} \oplus \mathcal{J}_{2}$ whose the induced map $J_{\tau} \colon T(M_{1}\times M_{2}) \otimes \mathbb{C} \to T(M_{1}\times M_{2}) \otimes \mathbb{C}$ satisfies 
\begin{equation}
J_{\tau} (\Psi_{\tau}) = \sqrt{-1} \Psi_{\tau},
\end{equation}
i.e., we take $J_{\tau}$ in such a way that $\Psi_{\tau}$ is of the $(1,0)$-type. The integrability of $J_{\tau} \in {\text{End}}(T(M_{1}\times M_{2}))$ follows from the fact that we can take local basis formed by $\Psi_{\tau}$, and forms of $(1,0)$-type with respect to the transverse holomorphic structure induced by $\mathcal{J}_{1} \oplus \mathcal{J}_{2}$, such that 
\begin{center}
$d\big (\Omega^{1,0}(M_{1} \times M_{2}) \big ) \subseteq \Omega^{2,0}(M_{1} \times M_{2}) \oplus \Omega^{1,1}(M_{1} \times M_{2})$,
\end{center}
for more details see \cite{Manjarin}. Hence, the complex valued $1$-form defined in \ref{Characterization} allows us to recover completely the $1$-parametric family of complex structure provided by \ref{Manjarincomplex}.
\end{remark}

\begin{remark} 
\label{matrixcomplex}
Under the hypothesis of Proposition \ref{Manjarincomplex}, and considering the last comments, if we denote $\tau = a + \sqrt{-1}b \in \mathbb{C} \backslash \mathbb{R}$, the complex structure $J_{\tau} \in {\text{End}}(T(M_{1}\times M_{2}))$ also can be thought as an extension of $\mathcal{J}_{1} \oplus \mathcal{J}_{2} \in {\text{End}}(\mathscr{D}_{1} \oplus \mathscr{D}_{2})$ which satisfies 
\begin{equation}
J_{\tau}|_{\langle T_{1},T_{2} \rangle_{\mathbb{R}}} = \begin{pmatrix}
 - \frac{a}{b} & -\frac{1}{b} \\
 \frac{a^{2}+b^{2}}{b} &  \ \ \ \frac{a}{b}
\end{pmatrix}, 
\end{equation}
where $J_{\tau}|_{\langle T_{1},T_{2} \rangle_{\mathbb{R}}}$ stands for the matrix of the restriction of $J_{\tau}$ on $\langle T_{1},T_{2} \rangle_{\mathbb{R}}$ with respect to the natural basis $\{T_{1},T_{2}\}$. Therefore, a straightforward computation shows us that, if $\tau = \sqrt{-1}$, it follows that 
\begin{center}
$J_{\sqrt{-1}}(X,Y) = \big (\phi_{\mathcal{J}_{1}}(X) - \eta_{2}(Y)T_{1}, \phi_{\mathcal{J}_{2}}(Y) + \eta_{1}(X)T_{2} \big )$,
\end{center}
i.e., for $a = 0$, and $b = 1$, the complex structure $J_{\tau}$ provided by Proposition \ref{Manjarincomplex} coincides with Morimoto's complex structure \ref{almostcomplexstruct}. Thus, the result of Proposition \ref{Manjarincomplex} provides a generalization for Calabi-Eckmann manifolds and for complex structures on the product of two normal almost contact manifolds (i.e. Morimoto's construction).
\end{remark}

\begin{remark}[Tsukada's complex structures] \label{Tsukadacomplex}In \cite{Tsukada}, K. Tsukada introduced a 1-parametric family of complex structures on products of normal almost complex manifolds defined as follows: let $M_{1}$ and $M_{2}$ be normal almost contact manifolds with structure tensors $(\phi_{1},\xi_{1},\eta_{1})$ and $(\phi_{2},\xi_{2},\eta_{2})$, respectively. From this we can define an almost complex structure $J_{a,b} \in  {\text{End}}(T(M_{1}\times M_{2}))$, $a + \sqrt{-1}b \in \mathbb{C} \backslash \mathbb{R}$, by setting
\begin{equation}
J_{a,b} = \phi_{1} - \bigg ( \frac{a}{b}\eta_{1} + \frac{a^{2} + b^{2}}{b} \eta_{2}\bigg) \otimes \xi_{1} + \phi_{2} + \bigg ( \frac{1}{b}\eta_{1} + \frac{a}{b} \eta_{2}\bigg) \otimes \xi_{2}.
\end{equation}
As in \cite{MORIMOTO}, the integrability of the almost complex structure defined above follows from the normality condition of both almost contact structures involved in the construction. It is straightforward to check that 
\begin{equation}
J_{a,b}|_{\langle \xi_{1},\xi_{2} \rangle_{\mathbb{R}}} = \begin{pmatrix}
 - \frac{a}{b} & - \frac{a^{2} + b^{2}}{b} \\
 \ \frac{1}{b} &  \ \ \ \frac{a}{b}
\end{pmatrix}. 
\end{equation}
Thus, after a suitable change in the extension of $\mathcal{J}_{1} \oplus \mathcal{J}_{2} \in {\text{End}}(\mathscr{D}_{1} \oplus \mathscr{D}_{2})$, we see that the complex structures obtained from Proposition \ref{Manjarincomplex} coincides with Tsukada's complex structures.
\end{remark}

\section{Line bundles and principal $S^{1}$-bundles over complex flag manifolds} 
\label{sec3}

This section is devoted to provide some basic results about holomorphic line bundles and principal $S^{1}$-bundles over flag manifolds. The main references for the results which we shall cover in the next subsections are \cite{TOROIDAL}, \cite{CONTACTCORREA}, and \cite{BLAIR}. 

\subsection{Line bundles over flag manifolds}
\label{subsec3.1}
We start by collecting some basic facts about simple Lie algebras and simple Lie groups. Let $\mathfrak{g}^{\mathbb{C}}$ be a complex simple Lie algebra, by fixing a  Cartan subalgebra $\mathfrak{h}$ and a simple root system $\Sigma \subset \mathfrak{h}^{\ast}$, we have a decomposition of $\mathfrak{g}^{\mathbb{C}}$ given by
\begin{center}
$\mathfrak{g}^{\mathbb{C}} = \mathfrak{n}^{-} \oplus \mathfrak{h} \oplus \mathfrak{n}^{+}$, 
\end{center}
where $\mathfrak{n}^{-} = \sum_{\alpha \in \Pi^{-}}\mathfrak{g}_{\alpha}$ and $\mathfrak{n}^{+} = \sum_{\alpha \in \Pi^{+}}\mathfrak{g}_{\alpha}$, here we denote by $\Pi = \Pi^{+} \cup \Pi^{-}$ the root system associated to the simple root system $\Sigma = \{\alpha_{1},\ldots,\alpha_{l}\} \subset \mathfrak{h}^{\ast}$. We also denote by $\kappa$ the Cartan-Killing form of $\mathfrak{g}^{\mathbb{C}}$.

Now, given $\alpha \in \Pi^{+}$, we have $h_{\alpha} \in \mathfrak{h}$ such  that $\alpha = \kappa(\cdot,h_{\alpha})$, we can choose $x_{\alpha} \in \mathfrak{g}_{\alpha}$ and $y_{\alpha} \in \mathfrak{g}_{-\alpha}$ such that $[x_{\alpha},y_{\alpha}] = h_{\alpha}$. For every $\alpha \in \Sigma$, we can set 
\begin{equation}
h_{\alpha}^{\vee} = \frac{2}{\kappa(h_{\alpha},h_{\alpha})}h_{\alpha},
\end{equation}
from this we have the fundamental weights $\{\omega_{\alpha} \ | \ \alpha \in \Sigma\} \subset \mathfrak{h}^{\ast}$, where $\omega_{\alpha}(h_{\beta}^{\vee}) = \delta_{\alpha \beta}$, $\forall \alpha, \beta \in \Sigma$. We denote by
\begin{equation}
\Lambda_{\mathbb{Z}_{\geq 0}}^{\ast} = \bigoplus_{\alpha \in \Sigma}\mathbb{Z}_{\geq 0}\omega_{\alpha}, 
\end{equation}
the set of integral dominant weights of $\mathfrak{g}^{\mathbb{C}}$. From the Lie algebra representation theory, given $\mu \in \Lambda_{\mathbb{Z}_{\geq 0}}^{\ast}$ we have an irreducible $\mathfrak{g}^{\mathbb{C}}$-module $V(\mu)$ with highest weight $\mu$, we denote by $v_{\mu}^{+} \in V(\mu)$ the highest weight vector associated to $\mu \in  \Lambda_{\mathbb{Z}_{\geq 0}}^{\ast}$.

Let $G^{\mathbb{C}}$ be a connected, simply connected and complex Lie group with simple Lie algebra $\mathfrak{g}^{\mathbb{C}}$, and consider $G \subset G^{\mathbb{C}}$ as being a compact real form of $G^{\mathbb{C}}$. Given a parabolic Lie subgroup $P \subset G^{\mathbb{C}}$, without loss of generality, we can suppose

\begin{center}
$P  = P_{\Theta}$, \ for some \ $\Theta \subset \Sigma$.
\end{center}

By definition, we have $P_{\Theta} = N_{G^{\mathbb{C}}}(\mathfrak{p}_{\Theta})$, where ${\text{Lie}}(P_{\Theta}) = \mathfrak{p}_{\Theta} \subset \mathfrak{g}^{\mathbb{C}}$ is given by

\begin{center}

$\mathfrak{p}_{\Theta} = \mathfrak{n}^{+} \oplus \mathfrak{h} \oplus \mathfrak{n}(\Theta)^{-},$ \ with \ $\mathfrak{n}(\Theta)^{-} = \displaystyle \sum_{\alpha \in \langle \Theta \rangle^{-}} \mathfrak{g}_{\alpha}$.

\end{center}
It will be useful for us to consider the following basic subgroups

\begin{center}

$T^{\mathbb{C}} \subset B \subset P \subset G^{\mathbb{C}}$.

\end{center}
For each element in the chain of subgroups above we have the following characterization: 

\begin{enumerate}

\item $T^{\mathbb{C}} = \exp(\mathfrak{h})$,  \ \ (complex torus)

\item $B = N^{+}T^{\mathbb{C}}$, where $N^{+} = \exp(\mathfrak{n}^{+})$, \ \ (Borel subgroup)

\item $P = P_{\Theta} = N_{G^{\mathbb{C}}}(\mathfrak{p}_{\Theta})$, for some $\Theta \subset \Sigma \subset \mathfrak{h}^{\ast}$. \ \ (parabolic Lie subgroup)

\end{enumerate}
Associated to the data above we will be concerned to study the {\it generalized complex flag manifold} defined by  
$$
X_{P} = G^{\mathbb{C}} / P = G /G \cap P.
$$
The following theorem allows us to describe all $G$-invariant K\"{a}hler structures on $X_{P}$ in terms of local K\"{a}hler potentials.
\begin{theorem}[Azad-Biswas, \cite{AZAD}]
\label{AZADBISWAS}
Let $\omega \in \Omega^{1,1}(X_{P})^{G}$ be a closed invariant real $(1,1)$-form, then we have

\begin{center}

$\pi^{\ast}\omega = \sqrt{-1}\partial \overline{\partial}\varphi$,

\end{center}
where $\pi \colon G^{\mathbb{C}} \to X_{P}$ is the projection map, and $\varphi \colon G^{\mathbb{C}} \to \mathbb{R}$ is given by 
\begin{center}
$\varphi(g) = \displaystyle \sum_{\alpha \in \Sigma \backslash \Theta}c_{\alpha}\log||gv_{\omega_{\alpha}}^{+}||$, 
\end{center}
with $c_{\alpha} \in \mathbb{R}_{\geq 0}$, $\forall \alpha \in \Sigma \backslash \Theta$. Conversely, every function $\varphi$ as above defines a closed invariant real $(1,1)$-form $\omega_{\varphi} \in \Omega^{1,1}(X_{P})^{G}$. Moreover, if $c_{\alpha} > 0$,  $\forall \alpha \in \Sigma \backslash \Theta$, then $\omega_{\varphi}$ defines a K\"{a}hler form on $X_{P}$.

\end{theorem}

\begin{remark}
\label{innerproduct}
It is worth pointing out that the norm $|| \cdot ||$ in the last theorem is a norm induced by a fixed $G$-invariant inner product $\langle \cdot, \cdot \rangle_{\alpha}$ on $V(\omega_{\alpha})$, $\forall \alpha \in \Sigma \backslash \Theta$. 
\end{remark}

Let $X_{P}$ be a flag manifold associated to some parabolic Lie subgroup $P = P_{\Theta} \subset G^{\mathbb{C}}$. According to Theorem \ref{AZADBISWAS}, by taking a fundamental weight $\omega_{\alpha} \in \Lambda_{\mathbb{Z}_{\geq0}}^{\ast}$, such that $\alpha \in \Sigma \backslash \Theta$, we can associate to this weight a closed real $G$-invariant $(1,1)$-form $\Omega_{\alpha} \in \Omega^{1,1}(X_{P})^{G}$ which satisfies 
\begin{equation}
\label{fundclasses}
\pi^{\ast}\Omega_{\alpha} = \sqrt{-1}\partial \overline{\partial} \varphi_{\omega_{\alpha}},
\end{equation}
where $\pi \colon G^{\mathbb{C}} \to G^{\mathbb{C}} / P = X_{P}$ and $\varphi_{\omega_{\alpha}}(g) = \displaystyle \frac{1}{2\pi}\log||gv_{\omega_{\alpha}}^{+}||^{2}$. 

The characterization for $G$-invariant real $(1,1)$-forms on $X_{P}$ provided by Theorem \ref{AZADBISWAS} can be used to compute the Chern class for holomorphic line bundles over $X_{P}$. Let us briefly describe how it can be done. Since each $\omega_{\alpha} \in \Lambda_{\mathbb{Z}_{\geq 0}}^{\ast}$ is an integral dominant weight, we can associate to it a holomorphic character $\chi_{\omega_{\alpha}} \colon T^{\mathbb{C}} \to \mathbb{C}^{\times}$, such that $(d\chi_{\omega_{\alpha}})_{e} = \omega_{\alpha}$, see for instance \cite[p. 466]{TAYLOR}. Given a parabolic Lie subgroup $P \subset G^{\mathbb{C}}$, we can take the extension $\chi_{\omega_{\alpha}} \colon P \to \mathbb{C}^{\times}$ and define a holomorphic line bundle as a vector bundle associated to the principal $P$-bundle $P \hookrightarrow G^{\mathbb{C}} \to G^{\mathbb{C}}/P$ by the twisted product
\begin{equation}
\label{C8S8.2Sub8.2.1Eq8.2.4}
L_{\chi_{\omega_{\alpha}}} =  G^{\mathbb{C}} \times_{\chi_{\omega_{\alpha}}} \mathbb{C}_{-\omega_{\alpha}}.
\end{equation}

\begin{remark}
\label{remarkcocycle}
In the description above we consider $\mathbb{C}_{-\omega_{\alpha}}$ as a $P$-space with the action $pz = \chi_{\omega_{\alpha}}(p)^{-1}z$, $\forall p \in P$ and $\forall z \in \mathbb{C}$ (cf. \cite{Bott}). Therefore, in terms of $\check{C}$ech cocycles, if we consider an open cover $X_{P} = \bigcup_{i \in I}U_{i}$ and
$G^{\mathbb{C}} = \{(U_{i})_{i \in I}, \psi_{ij} \colon U_{i} \cap U_{j} \to P\}$, then we have 
\begin{equation}
L_{\chi_{\omega_{\alpha}}} = \Big \{(U_{i})_{i \in I},\chi_{\omega_{\alpha}}^{-1} \circ \psi_{i j} \colon U_{i} \cap U_{j} \to \mathbb{C}^{\times} \Big \}.
\end{equation}
Thus, $L_{\chi_{\omega_{\alpha}}} = \{g_{ij}\} \in \check{H}^{1}(X_{P},\mathscr{O}_{X_{P}}^{\ast})$, with $g_{ij} = \chi_{\omega_{\alpha}}^{-1} \circ \psi_{i j}$, where $i,j \in I$.
\end{remark}
For us it will be important to consider the following results, see for instance \cite{AZAD} and \cite{TOROIDAL}.
\begin{proposition}
\label{C8S8.2Sub8.2.3P8.2.7}
Let $X_{P}$ be a flag manifold associated to some parabolic Lie subgroup $P = P_{\Theta}\subset G^{\mathbb{C}}$. Then, for every fundamental weight $\omega_{\alpha} \in \Lambda_{\mathbb{Z}_{\geq 0}}^{\ast}$, such that $\alpha \in \Sigma \backslash \Theta$, we have
\begin{equation}
\label{C8S8.2Sub8.2.3Eq8.2.28}
c_{1}(L_{\chi_{\omega_{\alpha}}}) = [\Omega_{\alpha}].
\end{equation}

\end{proposition}

\begin{proof}
Consider an open cover $X_{P} = \bigcup_{i \in I} U_{i}$ which trivializes both $P \hookrightarrow G^{\mathbb{C}} \to X_{P}$ and $L_{\chi_{\omega_{\alpha}}} \to X_{P}$, such that $\alpha \in \Sigma \backslash \Theta$, and take a collection of local sections $(s_{i})_{i \in I}$, such that $s_{i} \colon U_{i} \to G^{\mathbb{C}}$. From this, we define $q_{i} \colon U_{i} \to \mathbb{R}_{+}$ by
\begin{equation}
\label{functionshermitian}
q_{i} =  {\mathrm{e}}^{-2\pi \varphi_{\omega_{\alpha}} \circ s_{i}} = \frac{1}{||s_{i}v_{\omega_{\alpha}}^{+}||^{2}},
\end{equation}
for every $i \in I$. These functions $(q_{i})_{i \in I}$ satisfy $q_{j} = |\chi_{\omega_{\alpha}}^{-1} \circ \psi_{ij}|^{2}q_{i}$ on $U_{i} \cap U_{j} \neq \emptyset$, here we have used that $s_{j} = s_{i}\psi_{ij}$ on $U_{i} \cap U_{j} \neq \emptyset$, and $pv_{\omega_{\alpha}}^{+} = \chi_{\omega_{\alpha}}(p)v_{\omega_{\alpha}}^{+}$ for every $p \in P$ and $\alpha \in \Sigma \backslash \Theta$. Hence, we have a collection of functions $(q_{i})_{i \in I}$ which satisfies on $U_{i} \cap U_{j} \neq \emptyset$
\begin{equation}
\label{collectionofequ}
q_{j} = |g_{ij}|^{2}q_{i},
\end{equation}
such that $g_{ij} = \chi_{\omega_{\alpha}}^{-1} \circ \psi_{i j}$, where $i,j \in I$.

From the collection of smooth functions described above we can define a Hermitian structure $H$ on $L_{\chi_{\omega_{\alpha}}}$ by taking on each trivialization $f_{i} \colon L_{\chi_{\omega_{\alpha}}} \to U_{i} \times \mathbb{C}$ a metric defined by
\begin{equation}
\label{hermitian}
H((x,v),(x,w)) = q_{i}(x) v\overline{w},
\end{equation}
for $(x,v),(x,w) \in L_{\chi_{\omega_{\alpha}}}|_{U_{i}} \cong U_{i} \times \mathbb{C}$. The Hermitian metric above induces a Chern connection $\nabla = d + \partial \log H$ with curvature $F_{\nabla}$ satisfying 
\begin{equation}
\displaystyle \frac{\sqrt{-1}}{2\pi}F_{\nabla} = \Omega_{\alpha}.
\end{equation}
Hence, it follows that $c_{1}(L_{\chi_{\omega_{\alpha}}}) = [\Omega_{\alpha}]$. From the ideas described above we have the desired result.
\end{proof}
\begin{proposition}
\label{C8S8.2Sub8.2.3P8.2.6}
Let $X_{P}$ be a flag manifold associated to some parabolic Lie subgroup $P = P_{\Theta}\subset G^{\mathbb{C}}$. Then, we have
\begin{equation}
\label{picardeq}
{\text{Pic}}(X_{P}) = H^{1,1}(X_{P},\mathbb{Z}) = H^{2}(X_{P},\mathbb{Z}) = \displaystyle \bigoplus_{\alpha \in \Sigma \backslash \Theta}\mathbb{Z}[\Omega_{\alpha}].
\end{equation}
\end{proposition}

\begin{remark}
\label{C8S8.2Sub8.2.3R8.2.10}
In the previous results and comments we have restricted our attention just to fundamental weights $\omega_{\alpha} \in \Lambda_{\mathbb{Z}_{\geq0}}^{\ast}$ for which $\alpha \in \Sigma \backslash \Theta$. Actually, if we have a parabolic Lie subgroup $P \subset G^{\mathbb{C}}$, such that $P = P_{\Theta}$, the decomposition 
\begin{equation}
P_{\Theta} = [P_{\Theta},P_{\Theta}]T(\Sigma \backslash \Theta)^{\mathbb{C}},
\end{equation}
see for instance \cite[Proposition 8]{Akhiezer}, such that
\begin{equation}
T(\Sigma \backslash \Theta)^{\mathbb{C}} = \exp \Big \{ \displaystyle \sum_{\alpha \in  \Sigma \backslash \Theta}a_{\alpha}h_{\alpha} \ \Big | \ a_{\alpha} \in \mathbb{C} \Big \},
\end{equation}
shows us that ${\text{Hom}}(P,\mathbb{C}^{\times}) = {\text{Hom}}(T(\Sigma \backslash \Theta)^{\mathbb{C}},\mathbb{C}^{\times})$. Therefore, if we take $\omega_{\alpha} \in \Lambda_{\mathbb{Z}_{\geq 0}}^{\ast}$, such that $\alpha \in \Theta$, we obtain $L_{\chi_{\omega_{\alpha}}} = X_{P} \times \mathbb{C}$, i.e., the associated holomorphic line bundle $L_{\chi_{\omega_{\alpha}}}$ is trivial.

\end{remark}

In order to study the Boothby-Wang fibration as in Theorem \ref{BWhomo} it will be important for us to compute $c_{1}(X_{P})$. In order to do so, let us introduce $\delta_{P} \in \mathfrak{h}^{\ast}$ by setting
\begin{equation}
\label{fundweight}
\delta_{P} = \displaystyle \sum_{\alpha \in \Pi^{+} \backslash \langle \Theta \rangle^{+}} \alpha.
\end{equation}
From this, we have the following result. 

\begin{proposition}
\label{C8S8.2Sub8.2.3Eq8.2.35}
Let $X_{P}$ be a flag manifold associated to some parabolic Lie subgroup $P = P_{\Theta}\subset G^{\mathbb{C}}$, then we have

$$K_{X_{P}}^{-1} = \det \big(T^{(1,0)}X_{P} \big) = L_{\chi_{\delta_{P}}}.$$

\end{proposition}
From Remark \ref{remarkcocycle}, the result above allows us to write 
$$K_{X_{P}}^{-1} = \bigg \{(U_{i})_{i \in I},\chi_{\delta_{P}}^{-1} \circ \psi_{i j} \colon U_{i} \cap U_{j} \to \mathbb{C}^{\times} \bigg \}.$$
Moreover, since the holomorphic character associated to $\delta_{P}$ can be written as
$$\chi_{\delta_{P}} = \displaystyle \prod_{\alpha \in \Sigma \backslash \Theta} \chi_{\omega_{\alpha}}^{\langle \delta_{P},h_{\alpha}^{\vee} \rangle},$$
we have the following characterization

$$K_{X_{P}}^{-1} =  L_{\chi_{\delta_{P}}} = \displaystyle \bigotimes_{\alpha \in \Sigma \backslash \Theta} L_{\chi_{\omega_{\alpha}}}^{\otimes \langle \delta_{P},h_{\alpha}^{\vee} \rangle}.$$
Therefore, we obtain the following description for $c_{1}(X_{P})$
\begin{equation}
\label{Cherncanonical}
c_{1}(X_{P}) = \displaystyle \sum_{\alpha \in \Sigma \backslash \Theta} \langle \delta_{P},h_{\alpha}^{\vee} \rangle \big [ \Omega_{\alpha} \big].
\end{equation}
Thus, from Theorem \ref{AZADBISWAS} we have a K\"{a}hler-Einstein structure $\omega_{X_{P}}$ on $X_{P}$ defined by

\begin{equation}
\label{canonicalmetric}
\omega_{X_{P}} =  \displaystyle \sum_{\alpha \in \Sigma \backslash \Theta} \langle \delta_{P},h_{\alpha}^{\vee} \rangle \Omega_{\alpha},
\end{equation}
notice that ${\text{Ric}}(\omega_{X_{P}}) = 2\pi \omega_{X_{P}}$. It is worth pointing out that, also from Theorem \ref{AZADBISWAS}, we have $\omega_{X_{P}}$ determined by the quasi-potential $\varphi \colon G^{\mathbb{C}} \to \mathbb{R}$ defined by
\begin{equation}
\label{quasipotential}
\varphi(g) = \displaystyle \frac{1}{2\pi} \log \Big ( \prod_{\alpha \in \Sigma \backslash \Theta} ||gv_{\omega_{\alpha}}^{+}||^{2  \langle \delta_{P},h_{\alpha}^{\vee} \rangle}\Big), 
\end{equation}
for every $g \in G^{\mathbb{C}}$. Hence, given a local section $s_{U} \colon U \subset X_{P} \to G^{\mathbb{C}}$ we have the following local expression for $\omega_{X_{P}}$
\begin{equation}
\label{localform}
\omega_{X_{P}} = \displaystyle \frac{\sqrt{-1}}{2 \pi} \partial \overline{\partial}\log \Big ( \prod_{\alpha \in \Sigma \backslash \Theta} ||s_{U}v_{\omega_{\alpha}}^{+}||^{2  \langle \delta_{P},h_{\alpha}^{\vee} \rangle}\Big).    
\end{equation}

\begin{remark}
In order to do some local computations it will be convenient for us to consider the open set defined by the ``opposite" big cell in $X_{P}$. This open set is a distinguished coordinate neighbourhood $U \subset X_{P}$ of $x_{0} = eP \in X_{P}$ defined by the maximal Schubert cell. A brief description for the opposite big cell can be done as follows: let $\Pi = \Pi^{+} \cup \Pi^{-}$ be the root system associated to the simple root system $\Sigma \subset \mathfrak{h}^{\ast}$. From this, we can define the opposite big cell $U \subset X_{P}$ by

\begin{center}

 $U =  B^{-}x_{0} = R_{u}(P_{\Theta})^{-}x_{0} \subset X_{P}$,  

\end{center}
 where $B^{-} = \exp(\mathfrak{h} \oplus \mathfrak{n}^{-})$ and
 
 \begin{center}
 
 $R_{u}(P_{\Theta})^{-} = \displaystyle \prod_{\alpha \in \Pi^{-} \backslash \langle \Theta \rangle^{-}}N_{\alpha}^{-}$, \ \ (opposite unipotent radical)
 
 \end{center}
with $N_{\alpha}^{-} = \exp(\mathfrak{g}_{\alpha})$, $\forall \alpha \in \Pi^{-} \backslash \langle \Theta \rangle^{-}$. The opposite big cell defines a contractible open dense subset of $X_{P}$, thus the restriction of any vector bundle over this open set is trivial. For further results about Schubert cells and Schubert varieties we suggest \cite{MONOMIAL}.
\end{remark}

\begin{remark} Unless otherwise stated, in the examples which we will describe throughout this work we will use the conventions of \cite{SMA} and \cite{Humphreys} for the realization of classical simple Lie algebras as matrix Lie algebras. \end{remark}

Let us illustrate the ideas described so far by means of basic examples.

\begin{example}
\label{exampleP1}
Consider $G^{\mathbb{C}} = {\rm{SL}}(2,\mathbb{C})$, we fix the triangular decomposition for $\mathfrak{sl}(2,\mathbb{C})$ given by

\begin{center}

$\mathfrak{sl}(2,\mathbb{C}) = \Big \langle x = \begin{pmatrix}
 0 & 1 \\
 0 & 0
\end{pmatrix} \Big \rangle_{\mathbb{C}} \oplus  \Big \langle h = \begin{pmatrix}
 1 & 0 \\
 0 & -1
\end{pmatrix} \Big \rangle_{\mathbb{C}} \oplus \Big \langle y = \begin{pmatrix}
 0 & 0 \\
 1 & 0
\end{pmatrix} \Big \rangle_{\mathbb{C}}$.

\end{center}
Notice that all the information about the decomposition above is codified in $\Sigma = \{\alpha\}$ and $\Pi = \{\alpha,-\alpha\}$. Also, our set of integral dominant weights in this case is given by 
\begin{center}
$\Lambda_{\mathbb{Z}_{\geq 0}}^{\ast} = \mathbb{Z}_{\geq 0}\omega_{\alpha} $.
\end{center}
We take $P = B$ (Borel subgroup) and from this we obtain $X_{B} = {\rm{SL}}(2,\mathbb{C})/B = \mathbb{C}{\rm{P}}^{1}$. Moreover, from the cellular decomposition
\begin{center}
$X_{B} = \mathbb{C}{\rm{P}}^{1} = N^{-}x_{0} \cup \pi \Big ( \begin{pmatrix}
 0 & 1 \\
 -1 & 0
\end{pmatrix} \Big ),$

\end{center}
we can take the open set defined by the opposite big cell $U =  N^{-}x_{0} \subset X_{B}$ and the local section $s_{U} \colon U \subset \mathbb{C}{\rm{P}}^{1} \to {\rm{SL}}(2,\mathbb{C})$ defined by $s_{U}(nx_{0}) = n$, \ \ $\forall n \in N^{-}$. It is worthwhile to observe that in this case we have the open set  $U =  N^{-}x_{0} \subset \mathbb{C}{\rm{P}}^{1}$ parameterized by 

\begin{center}
    
$z \in \mathbb{C} \mapsto \begin{pmatrix}
 1 & 0 \\
 z & 1
\end{pmatrix} x_{0} \subset \mathbb{C}{\rm{P}}^{1}.$
    
\end{center}
Since $V(\omega_{\alpha}) = \mathbb{C}^{2}$, $v_{\omega_{\alpha}}^{+} = e_{1}$ and $ \langle \delta_{B},h_{\alpha}^{\vee} \rangle = 2$, it follows from Equation \ref{localform} that, over the opposite big cell $U = N^{-}x_{0} \subset X_{B}$, we have

\begin{center}

$\omega_{\mathbb{C}{\rm{P}}^{1}} = \displaystyle  \frac{\sqrt{-1}}{2\pi} \langle \delta_{B},h_{\alpha}^{\vee} \rangle \partial \overline{\partial} \log \Bigg ( \Big |\Big |\begin{pmatrix}
 1 & 0 \\
 z & 1
\end{pmatrix} e_{1} \Big| \Big|^{2} \Bigg ) = \frac{\sqrt{-1}}{\pi} \partial \overline{\partial} \log (1+|z|^{2}).$

\end{center}
Notice that in this case we have $K_{\mathbb{C}{\rm{P}}^{1}}^{-1} =  T^{(1,0)}\mathbb{C}{\rm{P}}^{1} = T\mathbb{C}{\rm{P}}^{1}$, and 
$K_{\mathbb{C}{\rm{P}}^{1}} = T^{\ast}\mathbb{C}{\rm{P}}^{1}$, besides we have
\begin{center}

${\text{Pic}}(\mathbb{C}{\rm{P}}^{1}) = \mathbb{Z}c_{1}(L_{\chi_{\omega_{\alpha}}})$,    
    
\end{center}
thus $K_{\mathbb{C}{\rm{P}}^{1}}^{-1} = L_{\chi_{\omega_{\alpha}}}^{\otimes 2}$. If we denote by $L_{\chi_{\omega_{\alpha}}}^{\otimes \ell} = \mathscr{O}(\ell)$, $\forall \ell \in \mathbb{Z}$, we obtain $K_{\mathbb{C}{\rm{P}}^{1}} = \mathscr{O}(-2)$ and ${\text{Pic}}(\mathbb{C}{\rm{P}}^{1})$ generated by $\mathscr{O}(1) = L_{\chi_{\omega_{\alpha}}}$. Furthermore, in this case we have the Fano index given by $I(\mathbb{C}{\rm{P}}^{1}) = 2$, which implies that

\begin{center}

$K_{\mathbb{C}{\rm{P}}^{1}}^{\otimes \frac{1}{2}} = \mathscr{O}(-1).$
    
\end{center}
The computation above is an interesting constructive exercise to understand how the approach by elements of the Lie theory, especially representation theory, can be useful to describe geometric structures.
\end{example}

\begin{example}
\label{examplePn}
Let us briefly describe the generalization of the previous example for $X_{P} = \mathbb{C}{\rm{P}}^{n}$. At first, we recall some basic data related to the Lie algebra $\mathfrak{sl}(n+1,\mathbb{C})$. By fixing the Cartan subalgebra $\mathfrak{h} \subset \mathfrak{sl}(n+1,\mathbb{C})$ given by diagonal matrices whose the trace is equal to zero, we have the set of simple roots given by
$$\Sigma = \Big \{ \alpha_{l} = \epsilon_{l} - \epsilon_{l+1} \ \Big | \ l = 1, \ldots,n\Big\},$$
here $\epsilon_{l} \colon {\text{diag}}\{a_{1},\ldots,a_{n+1} \} \mapsto a_{l}$, $ \forall l = 1, \ldots,n+1$. Therefore, the set of positive roots is given by
$$\Pi^+ = \Big \{ \alpha_{ij} = \epsilon_{i} - \epsilon_{j} \ \Big | \ i<j  \Big\}. $$
In this example we consider $\Theta = \Sigma \backslash \{\alpha_{1}\}$ and $P = P_\Theta$. Now, we take the open set defined by the opposite big cell $U =  R_{u}(P_{\Theta})^{-}x_{0} \subset \mathbb{C}{\rm{P}}^{n}$, where $x_0=eP$ (trivial coset) and  
\begin{center}
 $ R_{u}(P_{\Theta})^{-} = \displaystyle \prod_{\alpha \in \Pi^{-} \backslash \langle \Theta \rangle^{-}}N_{\alpha}^{-}$, \ with \ $N_{\alpha}^{-} = \exp(\mathfrak{g}_{\alpha})$, $\forall \alpha \in \Pi^{-} \backslash \langle \Theta \rangle^{-}$.
 \end{center}
We remark that in this case the open set $U =  R_{u}(P_{\Theta})^{-}x_{0}$ is parameterized by
\begin{center}
$(z_{1},\ldots,z_{n}) \in \mathbb{C}^{n} \mapsto \begin{pmatrix}
1 & 0 &\cdots & 0 \\
z_{1} & 1  &\cdots & 0 \\                  
\ \vdots  & \vdots &\ddots  & \vdots  \\
z_{n} & 0 & \cdots &1 
 \end{pmatrix}x_{0} \in U =  R_{u}(P_{\Theta})^{-}x_{0}$.
\end{center}
Notice that the coordinate system above is induced directly from the exponential map $\exp \colon {\text{Lie}}(R_{u}(P)^{-}) \to R_{u}(P)^{-}$.

From this, we can take a local section $s_{U} \colon U \subset \mathbb{C}{\rm{P}}^{n} \to {\rm{SL}}(n+1,\mathbb{C})$, such that 
$$s_{U}(nx_{0}) = n \in {\rm{SL}}(n+1,\mathbb{C}).$$ 
Since $V(\omega_{\alpha_{1}}) = \mathbb{C}^{n+1}$, $v_{\omega_{\alpha_{1}}}^{+} = e_{1}$ and $ \langle \delta_{P_{\Sigma \backslash \{\alpha_{1}\}}},h_{\alpha_{1}}^{\vee} \rangle = n+1$, it follows from Equation \ref{localform} that over the opposite big cell $U =  R_{u}(P_{\Theta})^{-}x_{0} \subset \mathbb{C}{\rm{P}}^{n}$ we have the expression of $\omega_{\mathbb{C}{\rm{P}}^{n}}$ given by
$$\omega_{\mathbb{C}{\rm{P}}^{n}} = \displaystyle \frac{(n+1)}{2\pi} \sqrt{-1}\partial \overline{ \partial} \log \Big (1 + \sum_{l = 1}^{n}|z_{l}|^{2} \Big ).$$
Notice that in this case we have ${\text{Pic}}(\mathbb{C}{\rm{P}}^{n}) = \mathbb{Z}c_{1}(L_{\chi_{\omega_{\alpha_{1}}}})$, thus $K_{\mathbb{C}{\rm{P}}^{n}}^{-1} = L_{\chi_{\omega_{\alpha_{1}}}}^{\otimes (n+1)}$. If we denote by $L_{\chi_{\omega_{\alpha_{1}}}}^{\otimes \ell} = \mathscr{O}(\ell)$, $\forall \ell \in \mathbb{Z}$, we obtain $K_{\mathbb{C}{\rm{P}}^{n}} = \mathscr{O}(-n-1)$ and ${\text{Pic}}(\mathbb{C}{\rm{P}}^{n})$ generated by $\mathscr{O}(1) = L_{\chi_{\omega_{\alpha_{1}}}}$. Moreover, in this case we have the Fano index given by $I(\mathbb{C}{\rm{P}}^{n}) = n+1$, which implies that

\begin{center}

$K_{\mathbb{C}{\rm{P}}^{n}}^{\otimes \frac{1}{n+1}} = \mathscr{O}(-1).$

\end{center}

\end{example}

\begin{example}
\label{grassmanian}
Consider $G^{\mathbb{C}} = {\rm{SL}}(4,\mathbb{C})$, here we use the same choice of Cartan subalgebra and conventions for the simple root system as in the previous example. Since our simple root system is given by
$$\Sigma = \Big \{ \alpha_{1} = \epsilon_{1}-\epsilon_{2}, \alpha_{2} = \epsilon_{2} - \epsilon_{3}, \alpha_{3} = \epsilon_{3} - \epsilon_{4}\Big \},$$
by taking $\Theta = \Sigma \backslash \{\alpha_{2}\}$ we obtain  for $P = P_{\Theta}$ the flag manifold $X_{P} = {\rm{Gr}}(2,\mathbb{C}^{4})$ (Klein quadric). Notice that in this case we have ${\text{Pic}}({\rm{Gr}}(2,\mathbb{C}^{4})) = \mathbb{Z}c_{1}(L_{\chi_{\alpha_{2}}})$. Thus, from Proposition \ref{C8S8.2Sub8.2.3Eq8.2.35}, it follows that 
$$K_{{\rm{Gr}}(2,\mathbb{C}^{4})}^{-1} = L_{\chi_{\omega_{\alpha_{2}}}}^{\otimes \langle \delta_{P},h_{\alpha_{2}}^{\vee} \rangle}.$$
By considering our Lie-theoretical conventions, we have
$$\Pi^{+} \backslash \langle \Theta \rangle^{+} = \Big \{ \alpha_{2}, \alpha_{1}+\alpha_{2},\alpha_{2}+\alpha_{3},\alpha_{1} + \alpha_{2} + \alpha_{3}\Big \},$$
hence
$$\delta_{P} = \displaystyle \sum_{\alpha \in \Pi^{+} \backslash \langle \Theta \rangle^{+}} \alpha = 2\alpha_{1} + 4 \alpha_{2} + 2\alpha_{3}.$$
By means of the Cartan matrix of  $\mathfrak{sl}(4,\mathbb{C})$ we obtain: $ \langle \delta_{P},h_{\alpha_{2}}^{\vee} \rangle = 4 \implies K_{{\rm{Gr}}(2,\mathbb{C}^{4})}^{-1} = L_{\chi_{\omega_{\alpha_{2}}}}^{\otimes4}.$ 
In what follows, we will use the following notation: 
$$L_{\chi_{\omega_{\alpha_{2}}}}^{\otimes \ell} := \mathscr{O}_{\alpha_{2}}(\ell),$$
for every $\ell \in \mathbb{Z}$, therefore we have $K_{{\rm{Gr}}(2,\mathbb{C}^{4})} = \mathscr{O}_{\alpha_{2}}(-4)$. In order to compute the local expression of $\omega_{{\rm{Gr}}(2,\mathbb{C}^{4})} \in c_{1}(\mathscr{O}_{\alpha_{2}}(-4))$, we observe that in this case the quasi-potential $\varphi \colon  {\rm{SL}}(4,\mathbb{C}) \to \mathbb{R}$ is given by

\begin{center}

$\varphi(g) = \displaystyle \frac{\langle \delta_{P}, h_{\alpha_{2}}^{\vee}\rangle}{2\pi} \log \big (||g v_{\omega_{\alpha_{2}}}^{+}||^{2}\big ) = \displaystyle \frac{2}{\pi} \log \big (||g v_{\omega_{\alpha_{2}}}^{+}||^{2} \big )$,

\end{center}
where $V(\omega_{\alpha_{2}}) = \bigwedge^{2}(\mathbb{C}^{4})$ and $v_{\omega_{\alpha_{2}}}^{+} =  e_{1} \wedge e_{2}$. Thus, we fix the basis $\{e_{i} \wedge e_{j}\}_{i<j}$ for $V(\omega_{\alpha_{2}}) = \bigwedge^{2}(\mathbb{C}^{4})$. Similarly to the previous examples, we consider the open set defined by the opposite big cell $U = B^{-}x_{0} \subset {\rm{Gr}}(2,\mathbb{C}^{4})$. In this case we have the local coordinates $nx_{0} \in U$ given by

\begin{center}

$(z_{1},z_{2},z_{3},z_{4}) \in \mathbb{C}^{4} \mapsto\begin{pmatrix}
1 & 0 & 0 & 0\\
0 & 1 & 0 & 0 \\                  
z_{1}  & z_{3} & 1 & 0 \\
z_{2}  & z_{4} & 0 & 1
 \end{pmatrix} x_{0} \in U = B^{-}x_{0}.$

\end{center}
Notice that the coordinates above are obtained directly from the exponential map $\exp \colon {\text{Lie}}(R_{u}(P)^{-}) \to R_{u}(P)^{-}$. From this, by taking the local section $s_{U} \colon U \subset {\rm{Gr}}(2,\mathbb{C}^{4}) \to {\rm{SL}}(4,\mathbb{C})$, $s_{U}(nx_{0}) = n$, we obtain

\begin{center}

$\varphi(s_{U}(nx_{0})) = \displaystyle \frac{2}{\pi} \log \Bigg ( 1 + \sum_{k = 1}^{4}|z_{k}|^{2} + \bigg |\det \begin{pmatrix}
 z_{1} & z_{3} \\
 z_{2} & z_{4}
\end{pmatrix} \bigg |^{2} \Bigg)$,

\end{center}
and the following local expression for  $\omega_{{\rm{Gr}}(2,\mathbb{C}^{4})} \in c_{1}(\mathscr{O}_{\alpha_{2}}(-4))$
\begin{equation}
\label{C8S8.3Sub8.3.2Eq8.3.21}
\omega_{{\rm{Gr}}(2,\mathbb{C}^{4})} =  \displaystyle \frac{2 \sqrt{-1}}{\pi}  \partial \overline{\partial} \log \Bigg (1+ \sum_{k = 1}^{4}|z_{k}|^{2} + \bigg |\det \begin{pmatrix}
 z_{1} & z_{3} \\
 z_{2} & z_{4}
\end{pmatrix} \bigg |^{2} \Bigg).
\end{equation}
It is worthwhile to observe that in this case we have the Fano index of ${\rm{Gr}}(2,\mathbb{C}^{4})$ given by $I({\rm{Gr}}(2,\mathbb{C}^{4})) = 4$, thus we obtain

\begin{center}
    
$K_{{\rm{Gr}}(2,\mathbb{C}^{4})}^{\otimes \frac{1}{4}} = \mathscr{O}_{\alpha_{2}}(-1).$    
    
\end{center}

\end{example}

\begin{remark}
Notice that from Proposition \ref{C8S8.2Sub8.2.3Eq8.2.35} we have for a complex flag manifold $X_{P}$ its Fano index is given by 

\begin{center}   
$I(X_{P}) = {\text{gcd}} \Big (  \langle \delta_{P}, h_{\alpha}^{\vee} \rangle \ \Big | \ \alpha \in \Sigma \backslash \Theta \Big )$,    
\end{center}
here we suppose $P = P_{\Theta} \subset G^{\mathbb{C}}$, for some $\Theta \subset \Sigma$. Thus, $I(X_{P})$ can be completely determined by the Cartan matrix of $\mathfrak{g}^{\mathbb{C}}$.
\end{remark}

\subsection{Principal $S^{1}$-bundles over flag manifolds}
\label{subsec3.2}
As we have seen previously, given a complex manifold $X$ and a line bundle $L \to X$ with Hermitian structure $H$, we can define a circle bundle by taking the sphere bundle
\begin{center}
$Q(L) = \Big \{ u \in L \ \Big | \ \sqrt{H(u,u)} = 1 \Big\}$.
\end{center}
The action of ${\rm{U}}(1)$ on $Q(L)$ is defined by $u \cdot \theta = u \mathrm{e}^{2\pi \theta \sqrt{-1}}$, $\forall \theta \in {\rm{U}}(1)$ and $\forall u \in Q(L)$. Furthermore, a straightforward computation shows that $L = Q(L) \times_{{\rm{U}}(1)} \mathbb{C}$. Conversely, given a circle bundle ${\rm{U}}(1) \hookrightarrow Q \to X$, we can construct a line bundle $L(Q) \to X$ as an associated bundle such that 
\begin{center}
$L(Q) = Q \times_{{\rm{U}}(1)} \mathbb{C}$,
\end{center}
where the twisted product is taken with respect to the action
\begin{center}
$\theta \cdot (u,z) = (u \cdot \theta, \mathrm{e}^{-2\pi \theta \sqrt{-1}}z)$, 
\end{center}
$\forall \theta \in {\rm{U}}(1)$ and $\forall (u,z) \in Q \times \mathbb{C}$. If we denote the set of all isomorphism classes of circle bundles over $X$ by

\begin{center}
$\mathscr{P}(X,{\rm{U}}(1))$,    
\end{center}
the previous idea provides the correspondences: 
\begin{center}
${\text{Pic}}^{\infty}(X) \to \mathscr{P}(X,{\rm{U}}(1))$, $L \mapsto Q(L)$ \ \ and \ \ $\mathscr{P}(X,{\rm{U}}(1)) \to {\text{Pic}}^{\infty}(X)$, $Q \mapsto L(Q)$,     
\end{center}
where ${\text{Pic}}^{\infty}(X)$ denotes the smooth Picard group of $X$, i.e., the set of isomorphism classes of complex vector bundles of rank $1$. Furthermore, we have 
\begin{center}
$L(Q(L)) = L$, \ \ $[u,\xi] \mapsto \xi u$ \ \ and \ \ $Q(L(Q)) = Q$, \ \ $u \mapsto [u,1]$.
\end{center}
It will be important in this work to consider the following well-known results for which the details about the proofs can be found in \cite{TOROIDAL}, \cite[Theorem 2.1]{BLAIR}.

\begin{theorem}

The set $\mathscr{P}(X,{\rm{U}}(1))$ of isomorphism classes of all principal circle bundles over $X$ forms an additive group. The identity element is given by the trivial bundle.
\end{theorem}

\begin{remark}
\label{productcircle}
From the previous comments, it will be suitable to consider the following characterization for the group structure of  $\mathscr{P}(X,{\rm{U}}(1))$

\begin{center}
    
$Q_{1} + Q_{2} = Q(L(Q_{1}) \otimes L(Q_{2})),$    
    
\end{center}
for $Q_{1},Q_{2} \in \mathscr{P}(X,{\rm{U}}(1))$. 
\end{remark}

Given $Q \in \mathscr{P}(X,{\rm{U}}(1))$, we can consider its associated homotopy exact sequence:

\begin{center}
\begin{tikzcd}

\cdots \arrow[r] & \pi_{2}(Q) \arrow[r] & \pi_{2}(X) \arrow[r,"\Delta_{Q}"] & \pi_{1}({\rm{U}}(1)) \arrow[r] & \cdots, 

\end{tikzcd}
\end{center}
notice that, since $\pi_{1}({\rm{U}}(1)) \cong \mathbb{Z}$, it follows that $\Delta_{Q} \in {\rm{Hom}}(\pi_{2}(X),\mathbb{Z})$. From this, we have the following result. 

\begin{theorem}

Let $h \colon \pi_{2}(X) \to H_{2}(X,\mathbb{Z})$ be the natural homomorphism and $\ell$ an integer given by $\Delta_{Q}c = \ell b_{0}$, where $b_{0}$ is the generator of $\pi_{1}({\rm{U}}(1))$ and $\Delta_{Q}$ is the boundary operator of the exact homotopy sequence of a bundle $Q \in \mathscr{P}(X,{\rm{U}}(1))$. Then,

\begin{center}
    
$\Big \langle \mathrm{e}(Q),h(c) \Big \rangle = \displaystyle \int_{h(c)} \mathrm{e}(Q) = \ell,$    
    
\end{center}
where $\mathrm{e}(Q)$ denotes the Euler class of $Q \in \mathscr{P}(X,{\rm{U}}(1))$.
\end{theorem}

For our purpose it will be important to consider the following corollary. 

\begin{corollary}
If $X$ is simply connected, then $\mathscr{P}(X,{\rm{U}}(1))$ is isomorphic to ${\rm{Hom}}(\pi_{2}(X),\mathbb{Z})$. The isomorphism is given by $Q \mapsto \Delta_{Q}$, where $\Delta_{Q}$ is the boundary operator of the exact homotopy sequence of a bundle $Q \in \mathscr{P}(X,{\rm{U}}(1))$.
\end{corollary}

Now, let $X$ be a complex manifold. From Hurewicz's theorem, if $X$ is simply connected it follows that $h \colon \pi_{2}(X) \to H_{2}(X,\mathbb{Z})$ is an isomorphism, thus we obtain
\begin{equation}
\label{picinfty}
\mathscr{P}(X,{\rm{U}}(1)) \cong H^{2}(X,\mathbb{Z}) \cong {\text{Pic}}^{\infty}(X),
\end{equation}
where the first isomorphism is given by $\Delta_{Q} \mapsto \mathrm{e}(Q)$, $\forall Q \in \mathscr{P}(X,{\rm{U}}(1)) $, and the second isomorphism follows from the exponential exact sequence of sheaves

\begin{center}
\begin{tikzcd}

0 \arrow[r] & \underline{\mathbb{Z}} \arrow[r,"2\pi \sqrt{-1}"] & \underline{\mathbb{C}} \arrow[r,"\exp"] & \underline{\mathbb{C}}^{\times} \arrow[r] & 0, 

\end{tikzcd}
\end{center}
notice that $ {\text{Pic}}^{\infty}(X) \cong H^{1}(X,\underline{\mathbb{C}}^{\times})$, see for instance \cite[Chapter 2]{LOOP}.

The isomorphism \ref{picinfty} allows us to see that, when $X$ is simply connected, we have 

\begin{center}
    
$\mathrm{e}(Q) = c_{1}(L(Q))$ \ \ and \ \ $c_{1}(L) = \mathrm{e}(Q(L))$,    
    
\end{center}
$\forall Q \in \mathscr{P}(X,{\rm{U}}(1)) $, $\forall L \in {\text{Pic}}^{\infty}(X)$. 

\begin{remark}
It is worth pointing out that, in the setting above, if $X$ is not simply connected we can also obtain the isomorphism \ref{picinfty}. Actually, if we consider the natural exact sequence of sheaves
\begin{center}
\begin{tikzcd}

0 \arrow[r] & \underline{\mathbb{Z}} \arrow[r,"2\pi \sqrt{-1}"] & \underline{\sqrt{-1}\mathbb{R}} \arrow[r,"\exp"] & \underline{S}^{1} \arrow[r] & 0, 

\end{tikzcd}
\end{center}
the result follows from the associated cohomology sequence 
\begin{center}
\begin{tikzcd}

\cdots \arrow[r] & H^{1}(X, \underline{\sqrt{-1}\mathbb{R}}) \arrow[r] & H^{1}(X,\underline{S}^{1}) \arrow[r] & H^{2}(X,\underline{\mathbb{Z}}) \arrow[r] & H^{2}(X, \underline{\sqrt{-1}\mathbb{R}}) \arrow[r] &\cdots 

\end{tikzcd}
\end{center}
Notice that $\mathscr{P}(X,{\rm{U}}(1)) \cong H^{1}(X,\underline{S}^{1})$, see for instance \cite[Chapter 2, page 18]{BLAIR}.
\end{remark}
Therefore, from Proposition \ref{C8S8.2Sub8.2.3P8.2.6} and the last comments we have the following result. 

\begin{theorem}[ Kobayashi, \cite{TOROIDAL}]
\label{circlegroup}

Let $X_{P}$ be a complex flag manifold defined by a parabolic Lie subgroup $P = P_{\Theta} \subset G^{\mathbb{C}}$. Then, we have

\begin{center}
    
$\mathscr{P}(X_{P},{\rm{U}}(1)) = \displaystyle \bigoplus_{\alpha \in \Sigma \backslash \Theta}  \mathbb{Z}\mathrm{e}(Q(L_{\chi_{\omega_{\alpha}}})).$
    
\end{center}

\end{theorem}

\begin{remark}
It is worthwhile to point out that this last result which we presented above is stated slightly different in \cite{TOROIDAL}. We proceed in this way because our approach is concerned to describe connections and curvature of line bundles and principal circle bundles, thus we use characteristic classes to describe $\mathscr{P}(X_{P},{\rm{U}}(1))$.
\end{remark}

\begin{remark}
Notice that, particularly, we have $\mathscr{P}(X_{P},{\rm{U}}(1)) \cong {\text{Pic}}(X_{P})$.
\end{remark}
In what follows we shall use the following notation:
\begin{equation}
\label{notationsimple}    
Q(\mu) := Q(L_{\chi_{\mu}}),
\end{equation}
for every $\mu \in \Lambda_{\mathbb{Z}_{\geq 0}}^{\ast}$.\ We also will denote by $\pi_{Q(\mu)}  \colon Q(\mu) \to X_{P}$ the associated projection map.

Our next task will be to compute $\mathrm{e}(Q(\omega_{\alpha})) \in H^{2}(X_{P},\mathbb{Z})$, $\forall \alpha \in \Sigma \backslash \Theta$. In order to do this, it will be important to consider Proposition \ref{C8S8.2Sub8.2.3P8.2.7} and the fact that $\mathrm{e}(Q(\omega_{\alpha})) = c_{1}(L_{\chi_{\omega_{\alpha}}})$, $\forall \alpha \in \Sigma \backslash \Theta$.

Consider an open cover $X_{P} = \bigcup_{i \in I} U_{i}$ which trivializes both $P \hookrightarrow G^{\mathbb{C}} \to X_{P}$ and $L_{\chi_{\omega_{\alpha}}} \to X_{P}$, such that $\alpha \in \Sigma \backslash \Theta$, and take a collection of local sections $(s_{i})_{i \in I}$, such that $s_{i} \colon U_{i} \subset X_{P} \to G^{\mathbb{C}}$. As we have seen, associated to this data we can define $q_{i} \colon U_{i} \to \mathbb{R}_{+}$ by setting
\begin{center}

$q_{i} =  {\mathrm{e}}^{-2\pi \varphi_{\omega_{\alpha}} \circ s_{i}} = \displaystyle \frac{1}{||s_{i}v_{\omega_{\alpha}}^{+}||^{2}},$

\end{center}
and from these functions we obtain a Hermitian structure $H$ on $L_{\chi_{\omega_{\alpha}}}$ by taking on each trivialization $f_{i} \colon L_{\chi_{\omega_{\alpha}}} \to U_{i} \times \mathbb{C}$ a Hermitian metric defined by
\begin{center}
$H((x,v),(x,w)) = q_{i}(x) v\overline{w},$
\end{center}
for $(x,v),(x,w) \in L_{\chi_{\omega_{\alpha}}}|_{U_{i}} \cong U_{i} \times \mathbb{C}$. Hence, for the pair $(L_{\chi_{\omega_{\alpha}}},H)$ we have the associated principal circle bundle

\begin{center}
    
$Q(\omega_{\alpha}) = \Big \{ u \in L_{\chi_{\omega_{\alpha}}} \ \Big | \ \sqrt{H(u,u)} = 1 \Big\}.$    
    
\end{center}
In terms of cocycles the principal circle bundle $Q(\omega_{\alpha})$ is determined by 

\begin{center}
    
$t_{ij} \colon U_{i} \cap U_{j} \to {\rm{U}}(1)$, \ \ $t_{ij} = \frac{g_{ij}}{||g_{ij}||},$
    
\end{center}
where $g_{ij} = \chi_{\omega_{\alpha}}^{-1} \circ \psi_{ij}$, see the proof of Proposition \ref{C8S8.2Sub8.2.3P8.2.7}. Therefore, if we take a local chart $h_{i} \colon \pi_{Q(\omega_{\alpha})}^{-1}(U_{i}) \subset Q(\omega_{\alpha}) \to U_{i} \times {\rm{U}}(1)$, on the transition $U_{i} \cap U_{j} \neq \emptyset$ we obtain
\begin{equation}
\label{transition}
(h_{i} \circ h_{j}^{-1})(x,a_{j}) = (x,a_{j}t_{ij}(x)) = (x,a_{i}),    
\end{equation}
thus we have $a_{i} = a_{j}t_{ij}$, on $U_{i} \cap U_{j} \neq \emptyset$. If we set 
\begin{equation}
\mathcal{A}_{i} = - \displaystyle \frac{1}{2}\big ( \partial - \overline{\partial} \big ) \log||s_{i} v_{\omega_{\alpha}}^{+}||^{2},    
\end{equation}
$\forall i \in I$, we obtain the following result:

\begin{proposition}
\label{locconnection}
The collection of local $\mathfrak{u}(1)$-valued 1-forms defined by 
\begin{equation}
\label{connection}    
\eta_{i}' = \pi_{Q(\omega_{\alpha})}^{\ast}\mathcal{A}_{i} + \frac{da_{i}}{a_{i}},
\end{equation}
where 
\begin{center}
$\mathcal{A}_{i} = - \displaystyle \frac{1}{2}\big ( \partial - \overline{\partial} \big ) \log||s_{i} v_{\omega_{\alpha}}^{+}||^{2},$
\end{center}
$\forall i \in I$, provides a connection $\eta'_{\alpha}$ on $Q(\omega_{\alpha})$ which satisfies $\eta'_{\alpha} = \eta_{i}'$ on $Q(\omega_{\alpha})|_{U_{i}}$, and
\begin{equation}
\label{eqconnections}
\displaystyle \frac{\sqrt{-1}}{2\pi}d\eta'_{\alpha} = \pi_{Q(\omega_{\alpha})}^{\ast} \Omega_{\alpha}.
\end{equation}
\end{proposition}

\begin{remark}
In what follows we will denote by $\mathcal{A} = (\mathcal{A}_{i})_{i \in I}$ the collection of (gauge) potentials obtained by the result above. We also will denote by $d\mathcal{A} \in \Omega^{1,1}(X_{P})$ the globally defined $(1,1)$-form associated to $\mathcal{A}$.
\end{remark}

The description provided by Proposition \ref{locconnection} will be fundamental for our next step to describe the contact structure of homogeneous contact manifolds.

\subsection{Examples}
\label{subsec3.3}
Let us illustrate the previous results, especially Proposition \ref{locconnection}, by means of basic examples.

\begin{example}[Hopf bundle]
\label{HOPFBUNDLE}
Consider $G^{\mathbb{C}} = {\rm{SL}}(2,\mathbb{C})$ and $P = B \subset {\rm{SL}}(2,\mathbb{C})$ as in Example \ref{exampleP1}. As we have seen, in this case we have

\begin{center}
    
$X_{B} = \mathbb{C}{\rm{P}}^{1},$ \ \ and \ \ $\mathscr{P}(\mathbb{C}{\rm{P}}^{1},{\rm{U}}(1)) = \mathbb{Z}\mathrm{e}(Q(\omega_{\alpha}))$,
    
\end{center}
where $Q(\omega_{\alpha}) = Q(\mathscr{O}(1))$. Since $K_{\mathbb{C}{\rm{P}}^{1}}^{\otimes \frac{1}{2}} = \mathscr{O}(-1)$, it follows that $Q(-\omega_{\alpha}) = S^{3}$. By considering the opposite big cell $U =  N^{-}x_{0} \subset X_{B}$ and the local section $s_{U} \colon U \subset \mathbb{C}{\rm{P}}^{1} \to {\rm{SL}}(2,\mathbb{C})$ defined by $s_{U}(nx_{0}) = n$, \ \ $\forall n \in N^{-}$, we obtain from Proposition \ref{locconnection} the following local expression
\begin{center}
$\mathcal{A}_{U} = \displaystyle \frac{1}{2}\big ( \partial - \overline{\partial} \big ) \log||s_{U} v_{\omega_{\alpha}^{+}}||^{2},$      
\end{center}
on the opposite big cell $U \subset \mathbb{C}{\rm{P}}^{1}$. Thus, we have

\begin{center}
    
$\mathcal{A}_{U} =  \displaystyle  \frac{1}{2} \big ( \partial - \overline{\partial} \big ) \log \Bigg ( \Big |\Big |\begin{pmatrix}
 1 & 0 \\
 z & 1
\end{pmatrix} e_{1} \Big| \Big|^{2} \Bigg ) = \displaystyle  - \frac{1}{2}\frac{zd\overline{z} - \overline{z}dz}{(1+|z|^{2})}.$    
    
\end{center}
Hence, we have a principal ${\rm{U}}(1)$-connection on $Q(-\omega_{\alpha}) = S^{3}$ (locally) defined by

\begin{center}

$\eta'_{\alpha} = \displaystyle  - \frac{1}{2}\frac{zd\overline{z} - \overline{z}dz}{(1+|z|^{2})} + \frac{da_{U}}{a_{U}}.$    
    
\end{center}
Therefore, we have 

\begin{center}
$\mathrm{e}(S^{3}) = \displaystyle \bigg [\frac{\sqrt{-1}}{2\pi}d\mathcal{A} \bigg ] \in H^{2}( \mathbb{C}{\rm{P}}^{1},\mathbb{Z}).$
\end{center}
It is worth mentioning that from the ideas above, given $Q \in \mathscr{P}(\mathbb{C}{\rm{P}}^{1},{\rm{U}}(1))$, it follows that $Q = Q(-\ell \omega_{\alpha})$, for some $\ell \in \mathbb{Z}$, thus we have

\begin{center}
    
$Q = S^{3}/\mathbb{Z}_{\ell}$ \ \  and \ \ $\mathrm{e}(Q) = \displaystyle \bigg [\frac{\ell \sqrt{-1}}{2\pi}d\mathcal{A} \bigg ] \in H^{2}( \mathbb{C}{\rm{P}}^{1},\mathbb{Z}).$
    
\end{center}
Thus, we obtain the Euler class of the principal circle bundle defined by $ Q(-\ell \omega_{\alpha}) = S^{3}/\mathbb{Z}_{\ell}$ (Lens space).
\end{example}

\begin{example}[Complex Hopf fibrations]
\label{COMPLEXHOPF}
The previous example can be easily generalized. Let us briefly describe how it can be done.

Consider the basic data as in Example \ref{examplePn}, namely, the complex simple Lie group $G^{\mathbb{C}} = {\rm{SL}}(n+1,\mathbb{C})$ and the parabolic Lie subgroup $P = P_{\Sigma \backslash \{\alpha_{1}\}}$. As we have seen, in this case we have

\begin{center}
    
$X_{P_{\Sigma \backslash \{\alpha_{1}\}}} = \mathbb{C}{\rm{P}}^{n}$ \ \  and  \ \ $\mathscr{P}(\mathbb{C}{\rm{P}}^{n},{\rm{U}}(1)) = \mathbb{Z}\mathrm{e}(Q(\omega_{\alpha_{1}}))$,
    
\end{center}
where $Q(\omega_{\alpha_{1}}) = Q(\mathscr{O}(1))$. Since $K_{\mathbb{C}{\rm{P}}^{n}}^{\otimes \frac{1}{n+1}} = \mathscr{O}(-1)$, it follows that $Q(-\omega_{\alpha_{1}}) = S^{2n+1}$. From Proposition \ref{locconnection} and a similar computation as in the previous example, we have

\begin{center}
    
$\mathcal{A}_{U} = \displaystyle \frac{1}{2} \big ( \partial - \overline{\partial} \big )\log \Big (1 + \sum_{l = 1}^{n}|z_{l}|^{2} \Big ),$
    
\end{center}
on the opposite big cell $U \subset \mathbb{C}{\rm{P}}^{n}$. Hence, we have a principal ${\rm{U}}(1)$-connection on $Q(-\omega_{\alpha_{1}}) = S^{2n+1}$ (locally) defined by

\begin{center}

$\eta'_{\alpha_{1}} = \displaystyle  - \frac{1}{2} \sum_{l = 1}^{n}\frac{z_{l}d\overline{z}_{l} - \overline{z}_{l}dz_{l}}{\big (1 + \sum_{l = 1}^{n}|z_{l}|^{2} \big )} + \frac{da_{U}}{a_{U}}.$    
    
\end{center}
Therefore, we obtain  

\begin{center}
$\mathrm{e}(S^{2n+1}) = \displaystyle \bigg [\frac{\sqrt{-1}}{2\pi}d\mathcal{A} \bigg ] \in H^{2}( \mathbb{C}{\rm{P}}^{n},\mathbb{Z}).$
\end{center}
It is worth pointing out that, given $Q \in \mathscr{P}(\mathbb{C}{\rm{P}}^{n},{\rm{U}}(1))$, it follows that $Q = Q(-\ell \omega_{\alpha_{1}})$, for some $\ell \in \mathbb{Z}$, thus we have

\begin{center}
    
$Q = S^{2n+1}/\mathbb{Z}_{\ell}$ \ \  and \ \ $\mathrm{e}(Q) = \displaystyle \bigg [\frac{\ell \sqrt{-1}}{2\pi}d\mathcal{A} \bigg ] \in H^{2}( \mathbb{C}{\rm{P}}^{n},\mathbb{Z}).$
    
\end{center}
Hence, we obtain the Euler class of the principal circle bundle defined by the Lens space $ Q(-\ell \omega_{\alpha_{1}}) = S^{2n+1}/\mathbb{Z}_{\ell}$.
\end{example}

\begin{example}[Stiefel manifold]
\label{STIEFEL}
Now, consider $G^{\mathbb{C}} = {\rm{SL}}(4,\mathbb{C}),$ and $P = P_{\Sigma \backslash \{\alpha_{2}\}}$ as in Example \ref{grassmanian}. In this case we have

\begin{center}
$X_{P_{\Sigma \backslash \{\alpha_{2}\}}} = {\rm{Gr}}(2,\mathbb{C}^{4}),$ \ \ and \ \ $\mathscr{P}({\rm{Gr}}(2,\mathbb{C}^{4}),{\rm{U}}(1)) = \mathbb{Z}\mathrm{e}(Q(\omega_{\alpha_{2}})),$
\end{center}
where $Q(\omega_{\alpha_{2}}) = Q(\mathscr{O}_{\alpha_{2}}(1))$. Since $K_{{\rm{Gr}}(2,\mathbb{C}^{4})}^{\otimes \frac{1}{4}} = \mathscr{O}_{\alpha_{2}}(-1)$, it follows that $Q(-\omega_{\alpha_{2}}) = \mathscr{V}_{2}(\mathbb{R}^{6})$ (Stiefel manifold). From Proposition \ref{locconnection} and the computations of Example \ref{grassmanian} we obtain

\begin{center}
    
$\mathcal{A}_{U} = \displaystyle \frac{1}{2}\big ( \partial - \overline{\partial} \big ) \log \Big ( 1 + \sum_{k = 1}^{4}|z_{k}|^{2} + \bigg |\det \begin{pmatrix}
 z_{1} & z_{3} \\
 z_{2} & z_{4}
\end{pmatrix} \bigg |^{2} \Big),$
    
\end{center}
on the opposite big cell $U \subset {\rm{Gr}}(2,\mathbb{C}^{4})$. Hence, we have a principal ${\rm{U}}(1)$-connection on $Q(-\omega_{\alpha_{2}}) = \mathscr{V}_{2}(\mathbb{R}^{6})$ (locally) defined by

\begin{center}

$\eta'_{\alpha_{2}} = \displaystyle \frac{1}{2}\big ( \partial - \overline{\partial} \big ) \log \Big ( 1 + \sum_{k = 1}^{4}|z_{k}|^{2} + \bigg |\det \begin{pmatrix}
 z_{1} & z_{3} \\
 z_{2} & z_{4}
\end{pmatrix} \bigg |^{2} \Big) + \frac{da_{U}}{a_{U}}.$    
    
\end{center}
Thus, we obtain 

\begin{center}
$\mathrm{e}(\mathscr{V}_{2}(\mathbb{R}^{6})) = \displaystyle \bigg [\frac{\sqrt{-1}}{2\pi}d\mathcal{A} \bigg ] \in H^{2}(  {\rm{Gr}}(2,\mathbb{C}^{4}),\mathbb{Z}).$
\end{center}
Notice that, given $Q \in \mathscr{P}({\rm{Gr}}(2,\mathbb{C}^{4}),{\rm{U}}(1))$, it follows that $Q = Q(-\ell \omega_{\alpha_{2}})$, for some $\ell \in \mathbb{Z}$. Therefore, we have

\begin{center}
    
$Q =\mathscr{V}_{2}(\mathbb{R}^{6})/\mathbb{Z}_{\ell}$ \ \  and \ \ $\mathrm{e}(Q) = \displaystyle \bigg [\frac{\ell \sqrt{-1}}{2\pi}d\mathcal{A} \bigg ] \in H^{2}( {\rm{Gr}}(2,\mathbb{C}^{4}),\mathbb{Z}).$
    
\end{center}
Hence, we obtain the Euler class of the principal circle bundle defined by $ Q(-\ell \omega_{\alpha}) =\mathscr{V}_{2}(\mathbb{R}^{6})/\mathbb{Z}_{\ell}$.
\end{example}

Let us explain how the examples above fit inside of a more general setting. Let $G^{\mathbb{C}}$ be a complex simply connected simple Lie group, and consider $P \subset G^{\mathbb{C}}$ as being a parabolic Lie subgroup. If we suppose that $P = P_{\Sigma \backslash \{\alpha\}}$, i.e., $P$ is a maximal parabolic Lie subgroup, then we have 

\begin{center}
    
$\mathscr{P}(X_{ P_{\Sigma \backslash \{\alpha\}}},{\rm{U}}(1)) = \mathbb{Z}{\mathrm{e}}(Q(\omega_{\alpha})).$
    
\end{center}
In order to simplify the notation, let us denote $P_{\Sigma \backslash \{\alpha\}}$ by $P_{\omega_{\alpha}}$. A straightforward computation shows that 
\begin{equation}
\label{propertiesbasic}
I(X_{P_{\omega_{\alpha}}}) = \langle \delta_{P_{\omega_{\alpha}}},h_{\alpha}^{\vee} \rangle, \ \ {\text{and}} \ \ K_{X_{P_{\omega_{\alpha}}}}^{ \otimes \frac{1}{ \langle \delta_{P_{\omega_{\alpha}}},h_{\alpha}^{\vee} \rangle}} = L_{\chi_{\omega_{\alpha}}}^{-1},
\end{equation}
thus we have
\begin{equation}
\label{maxparaboliccase}
Q(K_{X_{P_{\omega_{\alpha}}}}^{ \otimes \frac{1}{ \langle \delta_{P_{\omega_{\alpha}}},h_{\alpha}^{\vee} \rangle}}) = Q(-\omega_{\alpha}).    
\end{equation}
Now, consider the following definition. 
\begin{definition}[\cite{MINUSCULE}, \cite{SARA}]
A fundamental weight $\omega_{\alpha}$ is called minuscule if it satisfies the condition

\begin{center}

$\langle \omega_{\alpha},h_{\beta}^{\vee} \rangle \in \{0,1\}, \ \forall \beta \in \Pi^{+}.$
    
\end{center}
A flag manifold $X_{P_{\omega_{\alpha}}}$ associated to a maximal parabolic Lie subgroup $P_{\omega_{\alpha}}$ is called minuscule flag manifold if $\omega_{\alpha}$ is a minuscule weight.
\end{definition}

\begin{remark}
\label{exampleminuscle}
The flag manifolds of the previous examples are particular cases of flag manifolds defined by maximal parabolic Lie subgroups. 
\end{remark}

\section{Basic model and proof of the main results}
\label{sec4}

In this section we shall prove Theorem \ref{Theo1} and Theorem \ref{Theo2}. In order to do so, we employ the results developed in the previous sections.

\subsection{Basic model}
\label{basiccase}
As mentioned above, in this section we will prove some of the main results of this work. In order to motivate the ideas involved in our proofs, let us start by recalling some basic facts.

As we have seen, given a compact homogeneous contact manifold $(M,\eta,G)$, we have that $M = Q(L),$ for some ample line bundle $L^{-1} \in {\text{Pic}}(X_{P})$, see Theorem \ref{BWhomo}. Further, under the assumption that $c_{1}(L^{-1})$ defines a K\"{a}hler-Einstein metric on $X_{P} = G^{\mathbb{C}}/P$, we have

\begin{center}
$L = K_{X_{P}}^{\otimes \frac{\ell}{I(X_{P})}}.$
\end{center}

The examples of compact homogeneous contact manifolds associated to flag manifolds defined by maximal parabolic Lie subgroups will be useful for us in the next subsections.  In what follows we will further explore these particular examples. As we have seen, from \ref{maxparaboliccase}, if $P = P_{\omega_{\alpha}}$ it follows that 

\begin{center}
    
$M = Q(- \ell \omega_{\alpha}) = Q(- \omega_{\alpha})/\mathbb{Z}_{\ell},$
    
\end{center}
for some $\ell \in \mathbb{Z}_{>0}$. Hence, from Proposition \ref{locconnection} we have a connection $\eta'_{\alpha}$ defined on $Q(- \ell \omega_{\alpha})$ by

\begin{center}
$\eta_{\alpha}' =  \displaystyle \frac{\ell}{2}\big ( \partial - \overline{\partial} \big ) \log||s_{i} v_{\omega_{\alpha}}^{+}||^{2}+ \frac{da_{i}}{a_{i}}.$
\end{center}
Thus, a contact structure on $M = Q(-\ell \omega_{\alpha})$ is obtained from $\eta = - \sqrt{-1} \eta'_{\alpha}$. If we consider $a_{i} = \mathrm{e}^{\sqrt{-1}\theta_{i}}$, where $\theta_{i}$ is real, and is defined up to an integral multiple of $2 \pi$, we have that 

\begin{center}

$\eta = \displaystyle - \frac{\ell \sqrt{-1}}{2}\big ( \partial - \overline{\partial} \big ) \log||s_{i} v_{\omega_{\alpha}}^{+}||^{2} + d\theta_{i},$    
    
\end{center}
it is not difficult to check that $d\eta = 2\pi \ell \pi^{\ast}\Omega_{\alpha}.$    
    
This particular case turns out to be the basic model for all the cases which we have described in the examples of the previous sections. As we will see, the ideas developed above are essentially the model for the general case of circle bundles over complex flag manifolds. In the next sections we will come back to this basic example in order to illustrate some constructions.

\subsection{Proof of main results}

In order to prove our main result we start with a fundamental theorem which gathers together some important features of circle bundles over complex flag manifolds.

\begin{theorem}
\label{fundamentaltheorem}
Let $X_{P} = G^{\mathbb{C}}/P$ be a complex flag manifold associated to some parabolic Lie subgroup $P = P_{\Theta}$ of a complex simple Lie group $G^{\mathbb{C}}$. Then, given a principal $S^{1}$-bundle $Q \in \mathscr{P}(X_{P},{\rm{U}}(1))$, we have that

\begin{enumerate}

\item $Q = \displaystyle \sum_{\alpha \in \Sigma \backslash \Theta} Q(\ell_{\alpha}\omega_{\alpha})$, such that $\ell_{\alpha} \in \mathbb{Z}$, $\forall \alpha \in \Sigma \backslash \Theta$.

\item The manifold defined by the total space $Q$ admits a normal almost contact structure $(\phi,\xi = \frac{\partial}{\partial \theta},\eta)$, such that
\begin{equation}
\label{localalmostcontact}
\eta = \displaystyle \frac{\sqrt{-1}}{2}\big ( \partial - \overline{\partial} \big ) \log \Big ( \prod_{\alpha \in \Sigma \backslash \Theta} ||s_{U}v_{\omega_{\alpha}}^{+}||^{2  \ell_{\alpha}}\Big) + d\theta_{U},
\end{equation}
where $s_{U} \colon U \subset X_{P} \to G^{\mathbb{C}}$ is a local section, and $v_{\omega_{\alpha}}^{+}$ is the highest weight vector with weight $\omega_{\alpha}$ associated to the fundamental irreducible $\mathfrak{g}^{\mathbb{C}}$-module $V(\omega_{\alpha})$, $\forall \alpha \in \Sigma \backslash \Theta$. Moreover, $\phi \in {\text{End}}(TQ)$ is completely determined by the horizontal lift of $\sqrt{-1}\eta$ and the canonical invariant complex structure $J_{0}$ of $X_{P}$. Furthermore, it satisfies $\pi_{\ast} \circ \phi = J_{0} \circ \pi_{\ast}$.

\item We have a Riemannian metric $g_{Q}$ on $Q$ such that 
\begin{equation}
g_{Q} = \pi^{\ast} \big (\omega_{X_{P}}({\rm{id}} \otimes J_{0}) \big ) + \eta \otimes \eta, \ \ {\text{and}} \ \ \mathscr{L}_{\xi}g_{Q} = 0, 
\end{equation}
where $\omega_{X_{P}}$ is an invariant K\"{a}hler form on $X_{P}$. 

\end{enumerate}

\end{theorem}

\begin{proof}
The proof for each fact above goes as follows. From Theorem \ref{circlegroup}, up to isomorphism, we can write $Q$ 
\begin{center}
$Q = \displaystyle \sum_{\alpha \in \Sigma \backslash \Theta} Q(\ell_{\alpha}\omega_{\alpha})$, 
\end{center}
such that $\ell_{\alpha} \in \mathbb{Z}$, $\forall \alpha \in \Sigma \backslash \Theta$. Thus, we have item $(1)$. Now, we can apply Proposition \ref{locconnection} and obtain a connection on $Q$ given by
\begin{center}
$\eta^{Q} = \displaystyle  \sum_{\alpha \in \Sigma \backslash \Theta} \Bigg \{- \frac{\ell_{\alpha}}{2}\big ( \partial - \overline{\partial} \big ) \log \Big ( ||s_{U}v_{\omega_{\alpha}}^{+}||^{2}\Big) + \frac{da_{U}^{\alpha}}{a_{U}^{\alpha}} \Bigg \}$,
\end{center}
note that $\mathrm{e}(Q) = \sum \ell_{\alpha}\big [\Omega_{\alpha} \big ]$. By rearranging the expression above, we have
\begin{center}
$\eta^{Q} = -\displaystyle \frac{1}{2}\big ( \partial - \overline{\partial} \big ) \log \Big (  \prod_{\alpha \in \Sigma \backslash \Theta} ||s_{U}v_{\omega_{\alpha}}^{+}||^{2\ell_{\alpha} }\Big) + \sqrt{-1}d\theta_{U}$.
\end{center}
Hence, from $\eta = - \sqrt{-1} \eta^{Q}$ we obtain the expression in Equation \ref{localalmostcontact}. Since

\begin{center}

$d\eta = - \displaystyle  \sum_{\alpha \in \Sigma \backslash \Theta} 2 \pi \ell_{\alpha}\pi^{\ast}\Omega_{\alpha}$,

\end{center}
it follows from \ref{fundclasses} that $d\eta$ is the pullback of a $(1,1)$-form on $X_{P}$. Therefore, from Theorem \ref{metrichatakeyma} and Theorem \ref{almostcircle} we obtain item $(2)$ and item $(3)$.
\end{proof}
\begin{remark}
\label{classification}
Notice that Theorem \ref{fundamentaltheorem} provides a concrete generalization for the result introduced in \cite{MORIMOTO} which states that every compact simply connected homogeneous contact manifold admits a normal almost contact structure. Moreover, we have that Theorem \ref{fundamentaltheorem} provides a classification for such structures. In fact, by considering the left action of the compact real form $G \subset G^{\mathbb{C}}$ on $X_{P}$, it follows that
\begin{center}
$H_{G}^{\bullet}(X_{P},\mathbb{R}) \cong H_{DR}^{\bullet}(X_{P},\mathbb{R}),$
\end{center}
see for instance \cite[Theorem 1.28]{Algmodels}. Therefore, given $Q \in \mathscr{P}(X_{P},{\rm{U}}(1))$, and a connection $1$-form $\sqrt{-1} \eta_{0} \in \Omega^{1}(Q,\sqrt{-1}\mathbb{R})$, we can suppose that $d\eta_{0} = \pi^{\ast} \omega_{0}$, such that $\omega_{0} \in \Omega^{2}(X_{P})^{G}$. It follows from the uniqueness of $\omega_{0}$ as $G$-invariant representative that 
\begin{center}
$\eta_{0} = \eta + \pi^{\ast}\lambda,$
\end{center}
where $\eta \in \Omega^{1}(Q)$ is given by \ref{localalmostcontact}, and $\lambda \in \Omega^{1}(X_{P})$ satisfies $d\lambda = 0$. Since $\pi_{1}(X_{P}) = \{0\}$, it follows that $\lambda = df$, for some $f \in C^{\infty}(X_{P})$. From this, we obtain a gauge transformation
\begin{center}
$g_{f} \colon Q \to Q$, \ \ \ $u \mapsto u \cdot \mathrm{e}^{\sqrt{-1}f(\pi(u))},$
\end{center}
which satisfies $\sqrt {-1}\eta_{0} = g_{f}^{\ast} (\sqrt{-1}\eta)$. Thus, up to gauge transformations, the normal almost contact strucrures provided by Theorem \ref{fundamentaltheorem} are unique in the homogeneous setting.
\end{remark}

Now, by following  \cite{MORIMOTO}, \cite{Manjarin}, and by using Theorem \ref{fundamentaltheorem}, we can prove the following theorem.

\begin{theorem}
\label{main2}
Let $X_{P_{i}}$ be a complex flag manifold associated to some parabolic Lie subgroup $P_{i} \subset G_{i}^{\mathbb{C}}$, such that $i = 1,2$. Then, given principal $S^{1}$-bundles $Q_{1} \in \mathscr{P}(X_{P_{1}},{\rm{U}}(1))$ and $Q_{2} \in \mathscr{P}(X_{P_{2}},{\rm{U}}(1))$, we have the following results:

\begin{enumerate}

\item There exists a $1$-parametric family of complex structures $J_{\tau} \in {\text{End}}(T(Q_{1} \times Q_{2}))$ determined by $\mathcal{J}_{1} \oplus \mathcal{J}_{2}$, and by a complex valued $1$-form $\Psi_{\tau} \in \Omega^{1}(Q_{1} \times Q_{2}) \otimes \mathbb{C}$, which satisfies $J_{\tau}(\Psi_{\tau}) = \sqrt{-1} \Psi_{\tau}$, defined by
$$\Psi_{\tau} = \displaystyle{\frac{\sqrt{-1}}{2{\text{Im}}(\tau)} \Bigg \{\overline{\tau} \Bigg [d^{c}\log \Bigg ( \frac{1}{\displaystyle{\prod_{\alpha \in \Sigma_{1} \backslash \Theta_{1}} ||s_{U_{1}}v_{\omega_{\alpha}}^{+}||^{\ell_{\alpha}}}}\Bigg ) + d\theta_{U_{1}} \Bigg ] + d^{c}\log \Bigg ( \frac{1}{\displaystyle{\prod_{\beta \in \Sigma_{2} \backslash \Theta_{2}} ||s_{U_{2}}w_{\omega_{\beta}}^{+}||^{\ell_{\beta}}}}\Bigg ) + d\theta_{U_{2}} \Bigg \}},$$
for some local section $s_{U_{i}} \colon U_{i} \subset X_{P_{i}} \to G_{i}^{\mathbb{C}}$, $i = 1,2$, such that $\tau \in \mathbb{C} \backslash \mathbb{R}$. 
\item Particularly, if $\tau = \sqrt{-1}$, we have Morimoto's complex structure  $J_{\sqrt{-1}} \in {\text{End}}(T(Q_{1} \times Q_{2}))$, i.e.,
\begin{equation}
J_{\sqrt{-1}}(X,Y) = \big (\phi_{1}(X) - \eta_{2}(Y)\xi_{1}, \phi_{2}(Y) + \eta_{1}(X)\xi_{2} \big ),
\end{equation}
for all $(X,Y) \in T(Q_{1} \times Q_{2})$, where $(\phi_{i},\xi_{i},\eta_{i})$ is a normal almost contact structure on $Q_{i}$, $i =1,2$, as in Theorem \ref{fundamentaltheorem}. Moreover, by considering the Riemannian metric $g_{Q_{1}} \times g_{Q_{2}}$, with $g_{Q_{i}}$ as in Theorem \ref{fundamentaltheorem}, we obtain a Hermitian non-K\"{a}hler structure $(J_{\sqrt{-1}},g_{Q_{1}} \times g_{Q_{2}})$ on $Q_{1} \times Q_{2}$, with fundamental form $\Omega$ given by 
\begin{equation}
\Omega = \pi_{1}^{\ast} \omega_{X_{P_{1}}} + \pi_{2}^{\ast} \omega_{X_{P_{2}}} + \eta_{1} \wedge \eta_{2},
\end{equation}
where $\pi_{i} \colon Q_{i} \to X_{P_{i}}$, and $\omega_{X_{P_{i}}}$ is an invariant K\"{a}hler metric on $X_{P_{i}}$, $i = 1,2$. Furthermore, regarding the complex structure $J_{\sqrt{-1}} \in {\text{End}}(T(Q_{1} \times Q_{2}))$ described above, we have that the natural projection map

\begin{center}

$\pi_{1} \times \pi_{2} \colon (Q_{1}\times Q_{2},J_{\sqrt{-1}}) \to (X_{P_{1}}\times X_{P_{2}}, J_{1} \times J_{2})$,

\end{center}
is holomorphic, where $J_{i}$ is the canonical invariant complex structure on $X_{P_{i}}$ induced from $G_{i}^{\mathbb{C}}$, for $i = 1,2$.
\end{enumerate}
\end{theorem}

\begin{proof}
The proof follows from an application of some results which we have described previously. We first note that, for each $Q_{i} \in \mathscr{P}(X_{P_{i}},{\rm{U}}(1))$, by applying the previous Theorem \ref{fundamentaltheorem}, we can assign a normal almost contact structure $(\phi_{i},\xi_{i},\eta_{i})$, $i = 1,2$. From these normal almost contact structures we obtain $\rm{CR}$-structures which in turn are normal almost contact structures $(T_{i},\mathscr{D}_{i},\eta_{i})$, $i =1,2$, in the sense of Definition \ref{normalCR}. 

Now, by applying Proposition \ref{Manjarincomplex}, we obtain a 1-parametric family of complex structures $J_{\tau} \in {\text{End}}(T(Q_{1} \times Q_{2}))$, $\tau \in \mathbb{C}\backslash \mathbb{R}$, defined from extensions of $\mathcal{J}_{1} \oplus \mathcal{J}_{2}$, which satisfy 
\begin{equation}
J_{\tau}(\Psi_{\tau}) = \sqrt{-1}\Psi_{\tau},
\end{equation}
where $\Psi_{\tau} \in \Omega^{1}(Q_{1} \times Q_{2}) \otimes \mathbb{C}$ is given by
\begin{equation}
\displaystyle \Psi_{\tau} = \frac{\sqrt{-1}}{2{\text{Im}}(\tau)} \big (\overline{\tau}\eta_{1} + \eta_{2} \big ),
\end{equation}
see Remark \ref{complexform} for more details about the comments above.

From Theorem \ref{main2} we can describe explicitly $\Psi_{\tau}$. In fact, if we denote $d^{c} := \sqrt{-1}(\overline{\partial} - \partial)$, we get from item (2) of Theorem \ref{fundamentaltheorem} that 
\begin{center}
$\displaystyle \eta_{1} = d^{c}\log \Bigg ( \frac{1}{\displaystyle{\prod_{\alpha \in \Sigma_{1} \backslash \Theta_{1}} ||s_{U_{1}}v_{\omega_{\alpha}}^{+}||^{\ell_{\alpha}}}}\Bigg ) + d\theta_{U_{1}}$, \ \ \ and \ \ \ $ \displaystyle \eta_{2} = d^{c}\log \Bigg ( \frac{1}{\displaystyle{\prod_{\beta \in \Sigma_{2} \backslash \Theta_{2}} ||s_{U_{2}}w_{\omega_{\beta}}^{+}||^{\ell_{\beta}}}}\Bigg ) + d\theta_{U_{2}},$
\end{center}
for some local section $s_{U_{i}} \colon U_{i} \subset X_{P_{i}} \to G_{i}^{\mathbb{C}}$, $i = 1,2$, here we denote, respectively, by $v_{\omega_{\alpha}}^{+}$, $\alpha \in \Sigma_{1}$, and $w_{\omega_{\beta}}^{+}$, $\beta \in \Sigma_{2}$, the highest-weight vectors associated to the fundamental irreducible representations of $G_{1}^{\mathbb{C}}$ and $G_{2}^{\mathbb{C}}$. From this, we conclude the proof of item (1).

For item (2), it is straightforward to verify that, if $\tau = \sqrt{-1}$, then
\begin{center}
$J_{\sqrt{-1}}(X,Y) = \big (\phi_{\mathcal{J}_{1}}(X) - \eta_{2}(Y)T_{1}, \phi_{\mathcal{J}_{2}}(Y) + \eta_{1}(X)T_{2} \big )$,
\end{center}
for all $(X,Y) \in T(Q_{1} \times Q_{2})$, where $(\phi_{i},\xi_{i} = T_{i},\eta_{i})$ is a normal almost contact structure on $Q_{i}$, $i =1,2$, see for instance Remark \ref{matrixcomplex}. Now, from item (3) of Theorem \ref{fundamentaltheorem}, we have 
\begin{center}

$g_{Q_{i}} = \pi_{i}^{\ast} \big (\omega_{X_{P_{i}}}({\rm{id}} \otimes J_{i}) \big ) + \eta_{i} \otimes \eta_{i}$, 

\end{center}
where $(\omega_{X_{P_{i}}},J_{i})$ is the invariant K\"{a}hler structure on $X_{P_{i}}$, $i = 1,2$, see Equation \ref{localform}. Therefore, by setting 

\begin{center}

$g((X,Y),(Z,W)) = g_{Q_{1}}(X,Z) + g_{Q_{2}}(Y,W),$

\end{center}
we obtain a Hermitian metric on $Q_{1} \times Q_{2}$. It is straightforward to verify that $\Omega = g(J \otimes {\rm{id}})$ is given by

\begin{center}

$\Omega = \pi_{1}^{\ast} \omega_{X_{P_{1}}} + \pi_{2}^{\ast} \omega_{X_{P_{2}}} + \eta_{1} \wedge \eta_{2}$.

\end{center}
Moreover, since $d\omega_{X_{P_{i}}} = 0$, we have $d\Omega = d\eta_{1} \wedge \eta_{2} - \eta_{1} \wedge d\eta_{2}$, which means that $\Omega$ is not closed. Thus, $(Q_{1} \times Q_{2}, \Omega, J)$ defines a Hermitian manifold which is non-K\"{a}hler. Now, since from item (2) of Theorem \ref{fundamentaltheorem} we have $(\pi_{i})_{\ast} \circ \phi_{i} = J_{i} \circ (\pi_{i})_{\ast}$, a straightforward computation shows that

\begin{center}

$(\pi_{1} \times \pi_{2})_{\ast} \circ J = (J_{1} \times J_{2}) \circ (\pi_{1} \times \pi_{2})_{\ast}.$

\end{center}
Thus, we have the desired result.
\end{proof}

\begin{remark}
Notice that in the setting above we have a natural induced principal $T^{2}$-bundle  

\begin{center}

$T^{2} \hookrightarrow Q_{1} \times Q_{2} \to X_{P_{1}} \times X_{P_{2}}$,

\end{center}
such that the action of $T^{2} = {\rm{U}}(1) \times {\rm{U}}(1)$ is the diagonal action on the product $Q_{1} \times Q_{2}$.
\end{remark}

\begin{remark}
As in Subsection \ref{subsec2.3}, consider the following principal ${\rm{U}}(1)$-bundles

\begin{center}

${\rm{U}}(1) \hookrightarrow Q(K_{X_{P_{1}}}^{\otimes \frac{\ell_{1}}{I(X_{P_{1}})}}) \to X_{P_{1}},$ \ \ and \ \ ${\rm{U}}(1) \hookrightarrow Q(K_{X_{P_{2}}}^{\otimes \frac{\ell_{2}}{I(X_{P_{2}})}}) \to X_{P_{2}},$

\end{center}
such that $\ell_{1},\ell_{2} \in \mathbb{Z}_{>0}$. A direct application of Theorem \ref{main2} provides that any product of Homogeneous contact manifolds as above can be endowed with a Hermitian non-K\"{a}hler structure. Moreover, since this last ideas also can be applied to any product of circle bundles associated to ample line bundles over flag manifolds, we recover the result \cite[p. 432]{MORIMOTO} for simply connected homogeneous contact manifolds, which implies the Calabi and Eckmann construction \cite{CALABIECKMANN} for product of two odd-dimensional spheres.

Another important fact to notice is the following. Since $H^{1}(X_{P_{1}},\mathbb{R}) = \{0\}$ and $c_{1}(K_{X_{P_{1}}}^{\otimes \frac{\ell_{1}}{I(X_{P_{1}})}}) \in H^{2}(X_{P_{1}},\mathbb{R})$ is non-zero, we have that  $Q(K_{X_{P_{1}}}^{\otimes \frac{\ell_{1}}{I(X_{P_{1}})}}) \times Q(K_{X_{P_{2}}}^{\otimes \frac{\ell_{2}}{I(X_{P_{2}})}})$ does not admit any symplectic structure, $ \forall \ell_{1},\ell_{2} \in \mathbb{Z}_{>0}$, see \cite[Theorem 2.13]{SANKARAN}. Thus, Theorem \ref{fundamentaltheorem} together with Theorem \ref{main2} allow us to explicitly describe Hermitian structures on a huge class of (compact) non-K\"{a}hler manifolds. These manifolds provide a concrete huge family of compact complex manifolds which are not algebraic.
\end{remark}

\subsection{Examples of Hermitian non-K\"{a}hler structures}
\label{sec5}
In this subsection we apply Theorem \ref{fundamentaltheorem} and Theorem \ref{main2} in concrete cases. We start with a basic case which covers an important class of flag manifolds, namely, the class of complex flag manifolds with Picard number one.

\begin{example}[Basic model] 
\label{examplebasic}
The first example which we will explore are given by principal $S^{1}$-bundles over flag manifolds defined by maximal parabolic Lie subgroups.

As in Subsection \ref{basiccase}, let $G^{\mathbb{C}}$ be a simply connected complex Lie group with simple Lie algebra, and let $X_{P_{\omega_{\alpha}}}$ be a complex flag manifold associated to some maximal parabolic Lie subgroup $P_{\omega_{\alpha}} \subset G^{\mathbb{C}}$. Since in this case we have $\mathscr{P}(X_{P_{\omega_{\alpha}}},{\rm{U}}(1)) = \mathbb{Z}{\mathrm{e}}(Q(\omega_{\alpha}))$, given $Q \in \mathscr{P}(X_{P_{\omega_{\alpha}}},{\rm{U}}(1))$, it follows that $Q = Q(\ell \omega_{\alpha})$. Moreover, from Theorem \ref{fundamentaltheorem} we have a normal almost contact structure $(\phi,\xi = \frac{\partial}{\partial \theta},\eta)$ on $Q$, such that
\begin{center}
$\eta = \displaystyle \frac{\ell}{2}d^{c} \log \Bigg ( \frac{1}{||s_{U}v_{\omega_{\alpha}}^{+}||^{2}}\Bigg) + d\theta_{U}$.
\end{center}
Therefore, given complex flag manifolds $X_{P_{1}}$ and $X_{P_{2}}$, such that $P_{1} = P_{\omega_{\alpha}} \subset G_{1}^{\mathbb{C}}$ and $P_{2} = P_{\widetilde{\omega}_{\beta}} \subset G_{2}^{\mathbb{C}}$, for every pair $(Q_{1},Q_{2})$, such that $Q_{i} \in \mathscr{P}(X_{P_{i}},{\rm{U}}(1))$, $i = 1,2$, from Theorem \ref{main2} we have a 1-parametric family of complex structures $J_{\tau} \in {\text{End}}(T(Q_{1} \times Q_{2}))$, $\tau \in \mathbb{C} \backslash \mathbb{R}$, determined by $J_{\tau}(\Psi_{\tau}) = \sqrt{-1}\Psi_{\tau}$, such that 
\begin{center}
$\displaystyle \Psi_{\tau} = \frac{\sqrt{-1}}{2{\text{Im}}(\tau)} \big (\overline{\tau}\eta_{1} + \eta_{2} \big ).$
\end{center}
Further, for the particular case $\tau = \sqrt{-1}$, we have the Hermitian non-K\"{a}hler structure $(J_{\sqrt{-1}},\Omega)$ on $Q_{1} \times Q_{2}$, such that

\begin{center}
$J_{\sqrt{-1}}(X,Y) = \big (\phi_{1}(X) - \eta_{2}(Y)\xi_{1}, \phi_{2}(Y) + \eta_{1}(X)\xi_{2}\big)$,
\end{center}
for all $(X,Y) \in T(Q_{1} \times Q_{2})$, where $(\phi_{i},\xi_{i} = \frac{\partial}{\partial \theta_{i}},\eta_{i})$, $i = 1,2$, is obtained from Theorem \ref{fundamentaltheorem}, such that
\begin{center}
$\eta_{1} = \displaystyle \frac{\ell_{1}}{2}d^{c} \log \Bigg ( \frac{1}{||s_{U_{1}}v_{\omega_{\alpha}}^{+}||^{2}}\Bigg) + d\theta_{U_{1}}$ \ \ and \ \ $\eta_{2} =  \displaystyle \frac{\ell_{2} }{2}d^{c} \log \Bigg ( \frac{1}{||s_{U_{2}}v_{\widetilde{\omega}_{\beta}}^{+}||^{2}}\Bigg) + d\theta_{U_{2}}$,
\end{center}
where $s_{U_{i}} \colon U_{i} \subset X_{P_{i}} \to G_{i}^{\mathbb{C}}$, $i = 1,2$, and the fundamental form $\Omega$ is given by

$\Omega = \displaystyle \frac{\sqrt{-1}}{2\pi} \Bigg \{\partial \overline{\partial}\log \Big ( ||s_{U_{1}}v_{\omega_{\alpha}}^{+}||^{2I(X_{P_{1}})}\Big) + \partial \overline{\partial} \log \Big ( ||s_{U_{2}}v_{\widetilde{\omega}_{\beta}}^{+}||^{2I(X_{P_{2}})}\Big)\Bigg\}$

$$ + \ \ \Bigg ( \displaystyle \frac{\ell_{1}}{2}d^{c} \log \Bigg ( \frac{1}{||s_{U_{1}}v_{\omega_{\alpha}}^{+}||^{2}}\Bigg) + d\theta_{U_{1}} \Bigg ) \wedge \Bigg (   \displaystyle \frac{\ell_{2} }{2}d^{c} \log \Bigg ( \frac{1}{||s_{U_{2}}v_{\widetilde{\omega}_{\beta}}^{+}||^{2}}\Bigg) + d\theta_{U_{2}}  \Bigg ),$$
where $I(X_{P_{1}}) = \langle \delta_{P_{\omega_{\alpha}}},h_{\alpha}^{\vee} \rangle$ and $I(X_{P_{2}}) = \langle \delta_{P_{\widetilde{\omega}_{\beta}}},\widetilde{h}_{\beta}^{\vee} \rangle$. Notice that in the expression above we consider the K\"{a}hler forms on $X_{P_{1}}$ and $X_{P_{2}}$ provided by the expression \ref{localform}. Hence, from Theorem \ref{main2} we have a Hermitian non-K\"{a}hler structure $(\Omega,J_{\sqrt{-1}})$ on $Q(\ell_{1} \omega_{\alpha}) \times Q(\ell_{2}\widetilde{\omega}_{\beta})$, $\forall \ell_{1},\ell_{2} \in \mathbb{Z}$, completely determined by elements of representation theory of $G_{1}^{\mathbb{C}}$ and $G_{2}^{\mathbb{C}}$. Notice also that in the setting above we have a natural induced principal $T^{2}$-bundle  

\begin{center}

$T^{2} \hookrightarrow Q_{1} \times Q_{2} \to X_{P_{1}} \times X_{P_{2}}$,

\end{center}
such that the action of $T^{2} = {\rm{U}}(1) \times {\rm{U}}(1)$ is the diagonal action on the product $Q_{1} \times Q_{2}$.

\end{example}

In the context of flag manifolds associated to maximal parabolic Lie subgroups $P_{\omega_{\alpha}} \subset G^{\mathbb{C}}$ we denote

\begin{center}

$\mathcal{Q}_{X_{P_{\omega_{\alpha}}}}(-\ell) = Q(K_{X_{P_{\omega_{\alpha}}}}^{\otimes \frac{\ell}{I(X_{P_{\omega_{\alpha}}})}}),$

\end{center}
for every $\ell \in \mathbb{Z}_{>0}$. In what follows, we provide some concrete application for the construction presented in Example \ref{examplebasic}.

\begin{example}[Herminitan structure on $\mathscr{V}_{2}(\mathbb{R}^{6}) \times S^{3}$]
\label{Grassmansphere}
In order to describe the associated 1-parametric family of complex structures and the Hermitian non-K\"{a}hler structure on $\mathscr{V}_{2}(\mathbb{R}^{6}) \times S^{3}$, provided by Theorem \ref{main2}, we notice that from Theorem \ref{fundamentaltheorem} we have contact structures associated to 

\begin{center}

${\rm{U}}(1) \hookrightarrow \mathscr{V}_{2}(\mathbb{R}^{6}) \to {\rm{Gr}}(2,\mathbb{C}^{4})$ \ \ and \ \ ${\rm{U}}(1) \hookrightarrow S^{3} \to \mathbb{C}{\rm{P}}^{1},$

\end{center}
respectively, given by 

\begin{center}

$\eta_{1} = \displaystyle \frac{1}{2}d^{c}\log \bigg ( 1 + \sum_{k = 1}^{4}|z_{k}|^{2} + \bigg |\det \begin{pmatrix}
 z_{1} & z_{3} \\
 z_{2} & z_{4}
\end{pmatrix} \bigg |^{2} \bigg) + d\theta_{U_{1}},$ \ \ and \ \ $\eta_{2} =  \displaystyle \frac{\overline{z}dz - zd\overline{z}}{2\sqrt{-1}(1 + |z|^{2})} + d\theta_{U_{2}},$

\end{center}
note that $\mathcal{Q}_{ {\rm{Gr}}(2,\mathbb{C}^{4})}(-1) = \mathscr{V}_{2}(\mathbb{R}^{6})$, and $\mathcal{Q}_{ \mathbb{C}{\rm{P}}^{1}}(-1) = S^{3}$. Therefore, from the normal almost contact structures $(\phi_{i},\xi_{i} = \frac{\partial}{\partial \theta_{i}},\eta_{i})$, $i = 1,2$, by means of Theorem \ref{main2} we can equip $\mathscr{V}_{2}(\mathbb{R}^{6}) \times S^{3}$ with a 1-parametric family of complex structures $J_{\tau} \in {\text{End}}(T(\mathscr{V}_{2}(\mathbb{R}^{6}) \times S^{3}))$, $\tau \in \mathbb{C} \backslash \mathbb{R}$, determined by $J_{\tau}(\Psi_{\tau}) = \sqrt{-1}\Psi_{\tau}$, such that 
\begin{center}
$\displaystyle \Psi_{\tau} = \frac{\sqrt{-1}}{2{\text{Im}}(\tau)} \Bigg \{\overline{\tau} \Bigg [ \frac{1}{2}d^{c}\log \bigg ( 1 + \sum_{k = 1}^{4}|z_{k}|^{2} + \bigg |\det \begin{pmatrix}
 z_{1} & z_{3} \\
 z_{2} & z_{4}
\end{pmatrix} \bigg |^{2} \bigg) + d\theta_{U_{1}}\Bigg ] + \displaystyle \frac{\overline{z}dz - zd\overline{z}}{2\sqrt{-1}(1 + |z|^{2})} + d\theta_{U_{2}} \Bigg \}.$
\end{center}
Moreover, for the particular case $\tau = \sqrt{-1}$, we have a Hermitian non-K\"{a}hler structure $(J_{\sqrt{-1}},\Omega)$, such that
\begin{center}
$J_{\sqrt{-1}}(X,Y) = \big (\phi_{1}(X) - \eta_{2}(Y)\xi_{1}, \phi_{2}(Y) + \eta_{1}(X)\xi_{2} \big )$,
\end{center}
for all $(X,Y) \in T(\mathscr{V}_{2}(\mathbb{R}^{6}) \times S^{3})$, and 

$\Omega = \displaystyle \frac{\sqrt{-1}}{\pi} \Bigg \{\displaystyle 2\partial \overline{\partial} \log \bigg (1+ \sum_{k = 1}^{4}|z_{k}|^{2} + \bigg |\det \begin{pmatrix}
 z_{1} & z_{3} \\
 z_{2} & z_{4}
\end{pmatrix} \bigg |^{2} \bigg) +   \frac{dz\wedge d\overline{z}}{(1 + |z|^{2})^{2}}\Bigg\}$
$$ \ \ + \ \ \ \Bigg (\displaystyle\frac{1}{2}d^{c}\log \bigg ( 1 + \sum_{k = 1}^{4}|z_{k}|^{2} + \bigg |\det \begin{pmatrix}
 z_{1} & z_{3} \\
 z_{2} & z_{4}
\end{pmatrix} \bigg |^{2} \bigg) + d\theta_{U_{1}} \Bigg ) \wedge \Bigg ( \displaystyle \frac{\overline{z}dz - zd\overline{z}}{2\sqrt{-1}(1 + |z|^{2})} + d\theta_{U_{2}} \Bigg ).$$
Hence, we have a compact simply connected Hermitian non-K\"{a}hler manifold defined by $(\mathscr{V}_{2}(\mathbb{R}^{6}) \times S^{3}, \Omega,J_{\sqrt{-1}})$.

The ideas above can be naturally used to describe 1-parametric families of complex structures and Hermitian non-K\"{a}hler structures also on

\begin{center}

$S^{3} \times S^{3}$, \ \ $S^{3} \times \mathbb{R}P^{3}$, \ \  $\mathscr{V}_{2}(\mathbb{R}^{6}) \times \mathscr{V}_{2}(\mathbb{R}^{6})$, \ \ $\mathscr{V}_{2}(\mathbb{R}^{6}) \times \mathbb{R}P^{3}$ \ \ and \ \ $\mathbb{R}P^{3} \times \mathbb{R}P^{3}$,

\end{center}
recall that $\mathbb{R}P^{3} = S^{3}/\mathbb{Z}_{2}$, see Example \ref{HOPFBUNDLE}.
\end{example}

\begin{example}
\label{Grassmanian}
The example above can be naturally generalized to principal ${\rm{U}}(1)$-bundles of the form
\begin{center}

${\rm{U}}(1) \hookrightarrow \mathcal{Q}_{ {\rm{Gr}}(k,\mathbb{C}^{n+1})}(-\ell)  \to {\rm{Gr}}(k,\mathbb{C}^{n+1})$. 

\end{center}
In fact, consider $G^{\mathbb{C}} = {\rm{SL}}(n+1,\mathbb{C})$, by fixing the Cartan subalgebra $\mathfrak{h} \subset \mathfrak{sl}(n+1,\mathbb{C})$ given by diagonal matrices whose the trace is equal to zero, as in Example \ref{examplePn} we have the set of simple roots given by
$$\Sigma = \Big \{ \alpha_{l} = \epsilon_{l} - \epsilon_{l+1} \ \Big | \ l = 1, \ldots,n\Big\},$$
here $\epsilon_{l} \colon {\text{diag}}\{a_{1},\ldots,a_{n+1} \} \mapsto a_{l}$, $ \forall l = 1, \ldots,n+1$. In this case we consider $\Theta = \Sigma \backslash \{\alpha_{k}\}$ and $P = P_{\omega_{\alpha_{k}}}$, thus we obtain

\begin{center}

${\rm{SL}}(n+1,\mathbb{C})/P_{\omega_{\alpha_{k}}} = {\rm{Gr}}(k,\mathbb{C}^{n+1}).$

\end{center}
A straightforward computation shows that $I({\rm{Gr}}(k,\mathbb{C}^{n+1})) = n+1,$ and $\mathscr{P}({\rm{Gr}}(k,\mathbb{C}^{n+1}),{\rm{U}}(1)) = \mathbb{Z}\mathrm{e}(Q(\omega_{\alpha_{k}})).$ Hence, from Example \ref{examplebasic}, it follows that  

\begin{center}

$\mathcal{Q}_{ {\rm{Gr}}(k,\mathbb{C}^{n+1})}(-1) = Q(-\omega_{\alpha_{k}}).$

\end{center}
Now, since $V(\omega_{\alpha_{k}}) = \bigwedge^{k}(\mathbb{C}^{n+1})$ and $v_{\omega_{\alpha_{k}}}^{+} = e_{1} \wedge \ldots \wedge e_{k}$, by taking the coordinate neighborhood $U =  R_{u}(P_{\omega_{\alpha_{k}}})^{-}x_{0} \subset {\rm{Gr}}(k,\mathbb{C}^{n+1})$, such that 

\begin{center}

$Z \in \mathbb{C}^{(n+1-k)k} \mapsto n(Z)x_{0} = \begin{pmatrix}
 \ 1_{k} & 0_{k,n+1-k} \\
 Z & 1_{n+1-k}
\end{pmatrix}x_{0},$

\end{center}
here we have identified $\mathbb{C}^{(n+1-k)k} \cong {\rm{M}}_{n+1-k,k}(\mathbb{C})$, we can take the local section $s_{U} \colon U \subset {\rm{Gr}}(k,\mathbb{C}^{n+1})\to {\rm{SL}}(n+1,\mathbb{C})$ defined by

\begin{center}

$s_{U}(n(Z)x_{0}) = n(Z) = \begin{pmatrix}
 \ 1_{k} & 0_{k,n+1-k} \\
 Z & 1_{n+1-k}
\end{pmatrix}.$

\end{center}
From this, we obtain the gauge potential 

\begin{center}

$A_{U} = (n+1)\partial \log \Bigg (\sum_{I} \bigg | \det_{I} \begin{pmatrix}
 \ 1_{k} \\
 Z 
\end{pmatrix} \bigg |^{2} \Bigg ),$

\end{center}
where the sum above is taken over all $k \times k$ submatrices whose the lines are labeled by $I = \{i_{1} < \ldots < i_{k}\} \subset \{1, \ldots, n+1\}$, here we consider the canonical basis for $\bigwedge^{k}(\mathbb{C}^{n+1})$ in the computations. From the potential above, we obtain the K\"{a}hler form on ${\rm{Gr}}(k,\mathbb{C}^{n+1})$ given by

\begin{center}

$\omega_{{\rm{Gr}}(k,\mathbb{C}^{n+1})} = \displaystyle \frac{(n+1)}{2\pi \sqrt{-1}}\overline{\partial} \partial\log \textstyle{ \Bigg (\sum_{I} \bigg | \det_{I} \begin{pmatrix}
 \ 1_{k} \\
 Z 
\end{pmatrix} \bigg |^{2} \Bigg )}.$

\end{center}
Now, from Theorem \ref{fundamentaltheorem}, we have an almost contact structure $(\phi,\xi = \frac{\partial}{\partial \theta},\eta)$ on $\mathcal{Q}_{ {\rm{Gr}}(k,\mathbb{C}^{n+1})}(-\ell) = Q(-\ell\omega_{\alpha_{k}})$ , such that
\begin{center}

$\eta = \displaystyle \frac{\ell}{2}d^{c} \log \Bigg (\sum_{I} \bigg | \det_{I} \begin{pmatrix}
 \ 1_{k} \\
 Z 
\end{pmatrix} \bigg |^{2} \Bigg ) + d\theta_{U}$.

\end{center} 
From the ideas above and Theorem \ref{main2}, given ${\rm{Gr}}(k,\mathbb{C}^{n+1})$ and ${\rm{Gr}}(r,\mathbb{C}^{m+1})$, we can equip the product

\begin{center}

$\mathcal{Q}_{ {\rm{Gr}}(k,\mathbb{C}^{n+1})}(-\ell_{1}) \times \mathcal{Q}_{{\rm{Gr}}(r,\mathbb{C}^{m+1})}(-\ell_{2})$,

\end{center}
with a 1-parametric family of complex structures $J_{\tau} \in {\text{End}}(T(\mathcal{Q}_{ {\rm{Gr}}(k,\mathbb{C}^{n+1})}(-\ell_{1}) \times \mathcal{Q}_{{\rm{Gr}}(r,\mathbb{C}^{m+1})}(-\ell_{2})))$, $\tau \in \mathbb{C} \backslash \mathbb{R}$, determined by $J_{\tau}(\Psi_{\tau}) = \sqrt{-1}\Psi_{\tau}$, such that 
\begin{center}
$\displaystyle \Psi_{\tau} = \frac{\sqrt{-1}}{2{\text{Im}}(\tau)} \Bigg \{\overline{\tau}\Bigg [\frac{\ell_{1}}{2}d^{c}\log \Bigg (\sum_{I_{1}} \bigg | \det_{I_{1}} \begin{pmatrix}
 \ 1_{k} \\
 Z_{1} 
\end{pmatrix} \bigg |^{2} \Bigg ) + d\theta_{U_{1}} \Bigg] + \frac{\ell_{2}}{2}d^{c}\log \Bigg (\sum_{I_{2}} \bigg | \det_{I_{2}} \begin{pmatrix}
 \ 1_{r} \\
 Z_{2} 
\end{pmatrix} \bigg |^{2} \Bigg ) + d\theta_{U_{2}} \Bigg \}.$
\end{center}
Moreover, for the particular case $\tau = \sqrt{-1}$, we have the Hermitian non-K\"{a}hler structure $(J_{\sqrt{-1}},\Omega)$ on $\mathcal{Q}_{ {\rm{Gr}}(k,\mathbb{C}^{n+1})}(-\ell_{1}) \times \mathcal{Q}_{{\rm{Gr}}(r,\mathbb{C}^{m+1})}(-\ell_{2})$, such that 

\begin{center}
$J_{\sqrt{-1}}(X,Y) = \big (\phi_{1}(X) - \eta_{2}(Y)\xi_{1}, \phi_{2}(Y) + \eta_{1}(X)\xi_{2} \big)$,
\end{center}
for all $(X,Y) \in T(\mathcal{Q}_{ {\rm{Gr}}(k,\mathbb{C}^{n+1})}(-\ell_{1}) \times \mathcal{Q}_{{\rm{Gr}}(r,\mathbb{C}^{m+1})}(-\ell_{2}))$, and

$\Omega = \displaystyle \frac{\sqrt{-1}}{2\pi} \Bigg \{(n+1)\partial \overline{\partial} \log \Bigg (\sum_{I_{1}} \bigg | \det_{I_{1}} \begin{pmatrix}
 \ 1_{k} \\
 Z_{1} 
\end{pmatrix} \bigg |^{2} \Bigg ) + (m+1)\partial \overline{\partial}  \log \Bigg (\sum_{I_{2}} \bigg | \det_{I_{2}} \begin{pmatrix}
 \ 1_{r} \\
 Z_{2} 
\end{pmatrix} \bigg |^{2} \Bigg )\Bigg\}$

$$ + \Bigg [\frac{\ell_{1}}{2}d^{c}\log \Bigg (\sum_{I_{1}} \bigg | \det_{I_{1}} \begin{pmatrix}
 \ 1_{k} \\
 Z_{1} 
\end{pmatrix} \bigg |^{2} \Bigg ) + d\theta_{U_{1}} \Bigg ] \wedge \Bigg [ \frac{\ell_{2}}{2}d^{c}\log \Bigg (\sum_{I_{2}} \bigg | \det_{I_{2}} \begin{pmatrix}
 \ 1_{r} \\
 Z_{2} 
\end{pmatrix} \bigg |^{2} \Bigg ) + d\theta_{U_{2}}  \Bigg ].$$
Thus, we have a compact Hermitian non-K\"{a}hler manifold defined by

\begin{center}

$(\mathcal{Q}_{ {\rm{Gr}}(k,\mathbb{C}^{n+1})}(-\ell_{1}) \times \mathcal{Q}_{{\rm{Gr}}(r,\mathbb{C}^{m+1})}(-\ell_{2}), \Omega, J_{\sqrt{-1}})$,

\end{center}
$\forall \ell_{1},\ell_{2} \in \mathbb{Z}$. It is strightforward to check that as a particular case of the construction above we obtain a 1-parametric family of complex structures and a Hermitian non-K\"{a}hler structure on the product of Lens spaces
\begin{center}
$S^{1} \times S^{1} \hookrightarrow \mathbb{L}_{(n,\ell_{1})} \times\mathbb{L}_{(m,\ell_{2})} \to  \mathbb{C}P^{n} \times \mathbb{C}P^{m},$
\end{center}
where $\mathbb{L}_{(n,\ell_{1})} = S^{2n+1}/ \mathbb{Z}_{\ell_{1}}$, and $\mathbb{L}_{(m,\ell_{2})} = S^{2m+1}/ \mathbb{Z}_{\ell_{2}}$. Moreover, particularly, we also have a 1-parametric family of complex structures and a Hermitian non-K\"{a}hler structure on the product of odd-dimensional spheres $S^{2n+1} \times S^{2m+1}$.
\end{example}

Now, we will illustrate how our approach can be used in the setting of maximal flag manifolds. 

\begin{example}
\label{ALOFFWALLACH}
 Consider $G^{\mathbb{C}} = {\rm{SL}}(3,\mathbb{C})$, and $P_{\emptyset} = B \subset {\rm{SL}}(3,\mathbb{C})$ (Borel subgroup). In this case we have the Wallach flag manifold given by the quotient space

\begin{center}

$W_{6} = {\rm{SL}}(3,\mathbb{C})/B = {\rm{SU}}(3)/T^{2}.$

\end{center}
Since the simple root system in this case is given by $\Sigma = \{\alpha_{1},\alpha_{2}\}$, it follows that 

\begin{center}

$\Pi^{+} = \big \{ \alpha_{1},\alpha_{2}, \alpha_{1} + \alpha_{2}\big \}$.

\end{center}
Thus, we obtain $\delta_{B} = 2\alpha_{1} + 2\alpha_{2}$. A straightforward computation shows that $\delta_{B} = 2\omega_{\alpha_{1}} + 2\omega_{\alpha_{2}}$, and $I(W_{6}) = 2$. Moreover, we have the following characterization 

\begin{center}
$\mathscr{P}(W_{6},{\rm{U}}(1)) = \mathbb{Z}\mathrm{e}(Q(\omega_{\alpha_{1}})) \oplus \mathbb{Z}\mathrm{e}(Q(\omega_{\alpha_{2}})).$
\end{center}
Hence, we have $\mathcal{Q}_{W_{6}}(-1) = Q(K_{W_{6}}^{\otimes \frac{1}{I(W_{6})}})$, such that 

\begin{center}

$\mathcal{Q}_{W_{6}}(-1) = Q(-\omega_{\alpha_{1}}) + Q(-\omega_{\alpha_{2}}) = {\rm{SU}}(3)/{\rm{U}}(1) = X_{1,1},$

\end{center}
notice that the manifold $X_{1,1}$ is an example of Aloff-Wallach space \cite{ALOFFWALLACH}. 

In order to compute the contact structure for $X_{1,1}$ as in Theorem \ref{fundamentaltheorem}, we notice that 

\begin{center}

$V(\omega_{\alpha_{1}}) = \mathbb{C}^{3}$ \ \  and \ \ $V(\omega_{\alpha_{2}}) = \bigwedge^{2}(\mathbb{C}^{3}),$

\end{center}
where $v_{\omega_{\alpha_{1}}}^{+} = e_{1}$, and $v_{\omega_{\alpha_{2}}}^{+} = e_{1} \wedge e_{2}$. From this, we consider the opposite big cell $U_{1} = R_{u}(B)^{-}x_{0} \subset W_{6}$, such that 

\begin{center}

$U_{1} = \Bigg \{ \begin{pmatrix}
1 & 0 & 0 \\
w_{1} & 1 & 0 \\                  
w_{2}  & w_{3} & 1
 \end{pmatrix}x_{0} \ \Bigg | \ w_{1},w_{2},w_{3} \in \mathbb{C} \Bigg \}$.

\end{center}
By taking the local section $s_{U_{1}} \colon U_{1} \subset W_{6} \to {\rm{SL}}(3,\mathbb{C})$, such that $s_{U_{1}}(nx_{0}) = n$, a straightforward computation shows that the contact form on $X_{1,1}$ provided by Theorem \ref{fundamentaltheorem} is given by 

\begin{center}

$\eta_{1} = \displaystyle \frac{1}{2} d^{c} \log \Bigg [ \bigg ( 1 + \displaystyle \sum_{i = 1}^{2}|w_{i}|^{2} \bigg ) \bigg (1 + |w_{3}|^{2} + \bigg | \det \begin{pmatrix}
w_{1} & 1  \\                  
w_{2}  & w_{3} 
 \end{pmatrix} \bigg |^{2} \bigg ) \Bigg ] + d\theta_{U_{1}},$

\end{center}
here we consider the canonical basis for $\mathbb{C}^{3}$ and $\bigwedge^{2}(\mathbb{C}^{3})$. Therefore, we have a normal almost contact structure  $(\phi_{1},\xi_{1} = \frac{\partial}{\partial \theta_{1}},\eta_{1})$ on $X_{1,1}$.

We notice that from \ref{localform} we have a  ${\rm{SU}}(3)$-invariant K\"{a}hler-Einstein metric on $W_{6}$ given by 
\begin{center}
$\omega_{W_{6}} = \displaystyle \frac{\sqrt{-1}}{\pi} \partial \overline{\partial}\log \Bigg [ \bigg ( 1 + \sum_{i = 1}^{2}|w_{i}|^{2} \bigg ) \bigg (1 + |w_{3}|^{2} + \bigg | \det \begin{pmatrix}
w_{1} & 1  \\                  
w_{2}  & w_{3} 
 \end{pmatrix} \bigg |^{2} \bigg ) \Bigg ],$

\end{center} 
so it follows that $\frac{d\eta_{1}}{\pi} = \omega_{W_{6}}$.

Now, we can combine the construction above with the previous constructions. Actually, consider

\begin{center}

${\rm{U}}(1) \hookrightarrow \mathscr{V}_{2}(\mathbb{R}^{6}) \to {\rm{Gr}}(2,\mathbb{C}^{4}),$ \ \ and \ \ ${\rm{U}}(1) \hookrightarrow S^{3} \to \mathbb{C}{\rm{P}}^{1},$

\end{center}
and the contact structures given, respectively, by 

\begin{enumerate}

\item $\eta_{2} = \displaystyle \frac{1}{2}d^{c} \log \bigg ( 1 + \sum_{k = 1}^{4}|z_{k}|^{2} + \bigg |\det \begin{pmatrix}
 z_{1} & z_{3} \\
 z_{2} & z_{4}
\end{pmatrix} \bigg |^{2} \bigg) + d\theta_{U_{2}},$  

\item $\eta_{3} =  \displaystyle \frac{\overline{z}dz - zd\overline{z}}{2\sqrt{-1}(1 + |z|^{2})} + d\theta_{U_{3}},$

\end{enumerate}
recall that $\mathcal{Q}_{ {\rm{Gr}}(2,\mathbb{C}^{4})}(-1) = \mathscr{V}_{2}(\mathbb{R}^{6})$, and $\mathcal{Q}_{ \mathbb{C}{\rm{P}}^{1}}(-1) = S^{3}$. If we consider the normal almost contact structures $(\phi_{i},\xi_{i} = \frac{\partial}{\partial \theta_{i}},\eta_{i})$, $i = 1,2,3$, obtained from Theorem \ref{fundamentaltheorem} on the manifolds described above, we can apply Theorem \ref{main2} and obtain a 1-parametric family of complex structures and Hermitian non-K\"{a}hler structures on the total space of the principal $T^{2}$-bundles

\begin{center}

$T^{2} \hookrightarrow X_{1,1} \times \mathscr{V}_{2}(\mathbb{R}^{6}) \to W_{6} \times {\rm{Gr}}(2,\mathbb{C}^{4}),
$ \ \ and \ \ $T^{2} \hookrightarrow X_{1,1} \times S^{3} \to W_{6} \times \mathbb{C}{\rm{P}}^{1}.$ 

\end{center}
In the first case, we have $J_{\tau} \in {\text{End}}(T(X_{1,1} \times \mathscr{V}_{2}(\mathbb{R}^{6})))$, $\tau \in \mathbb{C} \backslash \mathbb{R}$, determined by $J_{\tau}(\Psi_{\tau}) = \sqrt{-1}\Psi_{\tau}$, such that 

$\displaystyle \Psi_{\tau} = \frac{\sqrt{-1}}{2{\text{Im}}(\tau)} \Bigg \{\overline{\tau} \Bigg [\frac{1}{2} d^{c} \log \Bigg [ \bigg ( 1 + \displaystyle \sum_{i = 1}^{2}|w_{i}|^{2} \bigg ) \bigg (1 + |w_{3}|^{2} + \bigg | \det \begin{pmatrix}
w_{1} & 1  \\                  
w_{2}  & w_{3} 
 \end{pmatrix} \bigg |^{2} \bigg ) \Bigg ] + d\theta_{U_{1}} \Bigg] $
 
 $$ \displaystyle + \ \ \frac{1}{2}d^{c}\log \bigg ( 1 + \sum_{k = 1}^{4}|z_{k}|^{2} + \bigg |\det \begin{pmatrix}
 z_{1} & z_{3} \\
 z_{2} & z_{4}
\end{pmatrix} \bigg |^{2} \bigg) + d\theta_{U_{2}} \Bigg \}.$$
As before, for $\tau = \sqrt{-1}$ we have a Hermitian non-K\"{a}hler structure $(X_{1,1} \times \mathscr{V}_{2}(\mathbb{R}^{6}),\Omega_{1},J_{\sqrt{-1}})$, where
\begin{center}
$J_{\sqrt{-1}}(Z,Y) = \big (\phi_{1}(Z) - \eta_{2}(Y)\xi_{1}, \phi_{2}(Y) + \eta_{1}(Z)\xi_{2} \big )$,
\end{center}
for all $(Z,Y) \in T(X_{1,1} \times \mathscr{V}_{2}(\mathbb{R}^{6}))$, and

\begin{center}

$\Omega_{1} = \pi_{1}^{\ast} \omega_{W_{6}} + \pi_{2}^{\ast} \omega_{{\rm{Gr}}(2,\mathbb{C}^{4})} + \eta_{1} \wedge \eta_{2},$

\end{center}
here we omitted the local expression of $\Omega_{1}$, since it is quite extensive. In the second case, we have $J_{\tau} \in {\text{End}}(T(X_{1,1} \times S^{3}))$, $\tau \in \mathbb{C} \backslash \mathbb{R}$, determined by $J_{\tau}(\Psi_{\tau}) = \sqrt{-1}\Psi_{\tau}$, such that  
$$
\displaystyle \Psi_{\tau} = \frac{\sqrt{-1}}{2{\text{Im}}(\tau)} \Bigg \{\overline{\tau} \Bigg [\frac{1}{2} d^{c} \log \Bigg [ \bigg ( 1 + \displaystyle \sum_{i = 1}^{2}|w_{i}|^{2} \bigg ) \bigg (1 + |w_{3}|^{2} + \bigg | \det \begin{pmatrix}
w_{1} & 1  \\                  
w_{2}  & w_{3} 
 \end{pmatrix} \bigg |^{2} \bigg ) \Bigg ] + d\theta_{U_{1}} \Bigg] + \displaystyle \frac{\overline{z}dz - zd\overline{z}}{2\sqrt{-1}(1 + |z|^{2})} + d\theta_{U_{3}}\Bigg \}.
$$
Now, for the particular case $\tau = \sqrt{-1}$, we have $(X_{1,1} \times S^{3},\Omega_{2},J_{\sqrt{-1}})$, where
\begin{center}
$J_{\sqrt{-1}}(Z,Y) = \big (\phi_{1}(Z) - \eta_{3}(Y)\xi_{1}, \phi_{3}(Y) + \eta_{1}(Z)\xi_{3} \big)$,
\end{center}
for all $(Z,Y) \in T(X_{1,1} \times S^{3})$, and

$\Omega_{2} = \displaystyle \frac{\sqrt{-1}}{\pi} \Bigg \{\displaystyle  \partial \overline{\partial}\log \Bigg [ \bigg ( 1 + \sum_{i = 1}^{2}|w_{i}|^{2} \bigg ) \bigg (1 + |w_{3}|^{2} + \bigg | \det \begin{pmatrix}
w_{1} & 1  \\                  
w_{2}  & w_{3} 
 \end{pmatrix} \bigg |^{2} \bigg ) \Bigg ] + \frac{dz\wedge d\overline{z}}{(1 + |z|^{2})^{2}}\Bigg\}$

$$ + \ \ \Bigg \{\displaystyle \frac{1}{2} d^{c} \log \Bigg [ \bigg ( 1 + \displaystyle \sum_{i = 1}^{2}|w_{i}|^{2} \bigg ) \bigg (1 + |w_{3}|^{2} + \bigg | \det \begin{pmatrix}
w_{1} & 1  \\                  
w_{2}  & w_{3} 
 \end{pmatrix} \bigg |^{2} \bigg ) \Bigg ] + d\theta_{U_{1}} \Bigg \} \wedge \Bigg ( \displaystyle \frac{\overline{z}dz - zd\overline{z}}{2\sqrt{-1}(1 + |z|^{2})} + d\theta_{U_{3}} \Bigg ).$$
Note that the same idea explored above also can be used to obtain 1-parametric family of complex structures, and Hermitian non-K\"{a}hler structures, on $X_{1,1} \times X_{1,1}$ and $\mathbb{R}P^{3} \times X_{1,1}$.

\end{example}

Our next example illustrate how our results can be applied in the setting of complex flag manifolds with Picard number greater than one to construct examples of complex manifolds from Cartesian products of normal almost contact manifolds which are not contact manifolds.

\begin{example}
\label{Picardnotone}
As in the previous example, consider the Wallach flag manifold

\begin{center}

$W_{6} = {\rm{SL}}(3,\mathbb{C})/B = {\rm{SU}}(3)/T^{2}.$

\end{center}
As we have seen, in this case we have $\mathscr{P}(W_{6},{\rm{U}}(1)) = \mathbb{Z}\mathrm{e}(Q(\omega_{\alpha_{1}})) \oplus \mathbb{Z}\mathrm{e}(Q(\omega_{\alpha_{2}})).$ Thus, we can apply Theorem \ref{fundamentaltheorem} in order to obtain a normal almost contact structure $(\phi_{p,q},\xi_{p,q},\eta_{p,q})$ defined on

\begin{center}
$ \mathscr{Q}_{p,q} = Q(p \omega_{\alpha_{1}}) + Q(q\omega_{\alpha_{2}}),$
\end{center}
for every $p,q \in \mathbb{Z}$, such that 
\begin{center}

$\eta_{p,q} = \displaystyle \frac{\sqrt{-1}}{2}\big ( \partial - \overline{\partial} \big ) \log \Bigg [ \bigg ( 1 + \displaystyle \sum_{i = 1}^{2}|w_{i}|^{2} \bigg )^{p} \bigg (1 + |w_{3}|^{2} + \bigg | \det \begin{pmatrix}
w_{1} & 1  \\                  
w_{2}  & w_{3} 
 \end{pmatrix} \bigg |^{2} \bigg )^{q} \Bigg ] + d\theta_{U}.$ 

\end{center}
Therefore, we can apply Theorem \ref{main2} for $X_{P_{i}} = W_{6}$, $i = 1,2$, and then we have $J_{\tau} \in {\text{End}}(T(\mathscr{Q}_{p,q} \times \mathscr{Q}_{r,s}))$, $\tau \in \mathbb{C} \backslash \mathbb{R}$, determined by $J_{\tau}(\Psi_{\tau}) = \sqrt{-1}\Psi_{\tau}$, such that 

\begin{center}
$\displaystyle \Psi_{\tau} = \frac{\sqrt{-1}}{2{\text{Im}}(\tau)} \big (\overline{\tau}\eta_{p,q} + \eta_{r,s} \big ).$
\end{center}
Also, for $\tau = \sqrt{-1}$ we have a Hermitian non-K\"{a}hler structure $(\Omega_{pqrs},J_{\sqrt{-1}})$ on $\mathscr{Q}_{p,q} \times \mathscr{Q}_{r,s}$, where
\begin{center}
$J_{\sqrt{-1}}(Z,Y) = \big (\phi_{p,q}(Z) - \eta_{r,s}(Y)\xi_{p,q}, \phi_{r,s}(Y) + \eta_{p,q}(Z)\xi_{r,s} \big )$,
\end{center}
for all $(Z,Y) \in T(\mathscr{Q}_{p,q} \times \mathscr{Q}_{r,s})$, and

\begin{center}

$\Omega_{pqrs} = \pi_{1}^{\ast} \omega_{W_{6}} + \pi_{2}^{\ast} \omega_{W_{6}} + \eta_{p,q} \wedge \eta_{r,s},$

\end{center}
$\forall p,q,r,s \in \mathbb{Z}$. From this, we obtain examples of complex manifolds given by Cartesian products of normal almost contact manifolds which are not contact manifolds.

\begin{table}[H]
\label{table2}
\caption{ Table with examples which illustrate the case on which the almost contact manifolds involved in the Cartesian products of normal almost contact manifolds are not contact manifolds. We consider $p,q,r,s \in \mathbb{Z}$ to define the products.}

\begin{tabular}{|c|c|c|}
\hline
Normal         &         &                    \\  
almost contact & $Q(p\omega_{\alpha_{1}})$ & $Q(q\omega_{\alpha_{2}})$          \\
manifolds      &         &                     \\         
\hline
               &         &                    \\  
$Q(r\omega_{\alpha_{1}})$       & $Q(r\omega_{\alpha_{1}}) \times Q(p\omega_{\alpha_{1}})$ & $Q(r\omega_{\alpha_{1}}) \times Q(q\omega_{\alpha_{2}})$ \\
        &         &                  \\  
\hline
       &         &                                      \\  
$Q(s\omega_{\alpha_{2}})$& $Q(s\omega_{\alpha_{2}}) \times Q(p\omega_{\alpha_{1}})$ & $Q(s\omega_{\alpha_{2}}) \times Q(q\omega_{\alpha_{2}})$ \\
       &         &                                      \\  
\hline
\end{tabular}

\end{table}
\end{example}

\section{Applications in Hermitian geometry with torsion}
\label{section5}
In this section we shall explore some applications of our main results in the study of Hermitian geometry with torsion in principal torus bundles over complex flag manifolds.
\subsection{Calabi-Yau connections with torsion on Vaisman manifolds} In this subsection we further explore some applications of our main results. The goal is to provide constructive methods to describe certain structures associated to Hermitian manifolds defined by principal torus bundles over complex flag manifolds. We start by recalling some basic facts related to Hermitian manifolds.
\begin{remark}
\label{complexconv}
Let $(M,J)$ be a complex manifold. In this paper we consider $J(\xi) = \xi \circ J$, instead of $J(\xi) = -\xi \circ J$ as in \cite{GRANTCHAROV}, $\forall \xi \in T^{\ast}M$. Therefore, we should have some change of sign in certain expressions.
\end{remark}

Given a Hermitian manifold $(M,g,J)$, we have an associated $2$-form $\Omega \in \Omega^{2}(M)$, called fundamental $2$-form, such that
\begin{center}
$\Omega(X,Y) = g(JX,Y),$
\end{center}
$\forall X,Y \in TM$. From this, by considering the Levi-Civita connection $\nabla^{LC}$ associated to $g$, we can use the fundamental $2$-form to define a $1$-parametric family of connections $\nabla^{t} \colon \Gamma(TM) \times \Gamma(TM) \to \Gamma(TM)$, where $t$ is a free parameter, such that
\begin{equation}
g(\nabla^{t}_{X}(Y),Z) = g(\nabla^{LC}_{X}(Y),Z) + \frac{t-1}{4}(d^{c}\Omega)(X,Y,Z) + \frac{t+1}{4}(d^{c}\Omega)(X,JY,JZ), 
\end{equation}
$\forall X,Y,Z \in TM$, the connections $\nabla^{t}$ are called canonical connections of $(M,g,J)$, e.g. \cite{Gauduchon1}. Among these connections we have the Chern connection $\nabla^{C} = \nabla^{1}$ and the Bismut connection, or KT connection, given by $\nabla^{B} = \nabla^{-1}$. The Bismut connection is the unique connection which satisfies 
\begin{enumerate}
\item $\nabla^{B}g = 0$,
\item $\nabla^{B}J = 0$,
\item $g(T_{\nabla^{B}}(X,Y),Z) = d\Omega(JX,JY,JZ) = (Jd\Omega)(X,Y,Z),$
\end{enumerate}
$\forall X,Y,Z \in TM$, where $T_{\nabla^{B}}$ is the torsion of $\nabla^{B}$. 
\begin{remark}
Notice that from the characterization above of $\nabla^{B}$, it follows that 
\begin{center}
$g(\nabla^{B}_{X}(Y),Z )= g(\nabla^{LC}_{X}(Y),Z) + \displaystyle \frac{1}{2}T_{B}(X,Y,Z),$
\end{center}
$\forall X,Y,Z \in TM$, here we consider $T_{B} = -d^{c}\Omega = Jd\Omega$. Moreover, we have $\nabla^{B} = \nabla^{LC} + \frac{1}{2}T_{\nabla^{B}}$.
\end{remark}

In what follows we will focus our attention in the study of Bismut connections.
\begin{definition}
A KT structure on a smooth manifold $M^{2n}$ is a triple $(g, J, \nabla^{B})$, such that $(g,J)$ is a Hermitian structure and $\nabla^{B}$ is the corresponding Bismut connection.
\end{definition}
For the sake of simplicity, in some cases, we will denote by $(M,T_{B})$ a manifold endowed with a KT structure, where $T_{B} \in \Omega^{3}(M)$ is defined as before. 

Since $\nabla^{B}$ is a Hermitian connection, it follows that ${\text{Hol}}^{0}(\nabla^{B}) \subset {\rm{U}}(n)$, where ${\text{Hol}}^{0}(\nabla^{B})$ denotes the restricted holonomy group of the Bismut connection. 

\begin{definition}A Hermitian manifold $(M,g,J)$ is said to be Calabi-Yau with torsion, shortly CYT, if the restricted holonomy group of the Bismut connection ${\text{Hol}}^{0}(\nabla^{B})$ is contained in ${\rm{SU}}(n)$.
\end{definition}

\begin{remark}
\label{KTdef}
Another important class of KT structures are the strong KT structures (SKT). A KT structure $(g,J,\nabla^{B})$ is said to be SKT if $dT_{B} = 0$. The class of strong KT structures have been recently studied due to their applications in physics and generalized K\"{a}hler geometry, see \cite{Finosurvey} and references therein.
\end{remark}

Let us illustrate the ideas discussed so far by means of a simple example which incorporates some results covered previously.

\begin{example}
\label{KTMorimoto}
Let $X_{P_{i}}$ be a complex flag manifold associated to some parabolic Lie subgroup $P_{i} \subset G_{i}^{\mathbb{C}}$, such that $i = 1,2$. By taking $S^{1}$-bundles $Q_{1} \in \mathscr{P}(X_{P_{1}},{\rm{U}}(1))$ and $Q_{2} \in \mathscr{P}(X_{P_{2}},{\rm{U}}(1))$, it follows from Theorem \ref{main2} that the Hermitian non-K\"{a}hler structure $(J_{\sqrt{-1}},g_{Q_{1}} \times g_{Q_{2}})$ defines a KT structure $(g_{Q_{1}} \times g_{Q_{2}}, J_{\sqrt{-1}},\nabla^{B})$ on $Q_{1} \times Q_{2}$, such that 
\begin{center}
$T_{B} = J_{\sqrt{-1}} d\big (\eta_{1} \wedge \eta_{2} \big )$,
\end{center}
where 
\begin{center}
$\displaystyle \eta_{1} = d^{c}\log \Bigg ( \frac{1}{\displaystyle{\prod_{\alpha \in \Sigma_{1} \backslash \Theta_{1}} ||s_{U_{1}}v_{\omega_{\alpha}}^{+}||^{\ell_{\alpha}}}}\Bigg ) + d\theta_{U_{1}}$, \ \ \ and \ \ \ $ \displaystyle \eta_{2} = d^{c}\log \Bigg ( \frac{1}{\displaystyle{\prod_{\beta \in \Sigma_{2} \backslash \Theta_{2}} ||s_{U_{2}}w_{\omega_{\beta}}^{+}||^{\ell_{\beta}}}}\Bigg ) + d\theta_{U_{2}},$
\end{center}
recall that $J_{\sqrt{-1}}(X,Y) = \big (\phi_{1}(X) - \eta_{2}(Y)\xi_{1}, \phi_{2}(Y) + \eta_{1}(X)\xi_{2} \big )$, $\forall X,Y \in T(Q_{1} \times Q_{2})$. From this, we have a manifold with explicit KT structure given by $(Q_{1} \times Q_{2},T_{B})$.
\end{example}

As we have seen in Example \ref{KTMorimoto}, from Theorem \ref{main2} we can describe explicitly the KT structure induced by the Morimoto's Hermitian structure on the torus bundle
\begin{center}
$\pi_{1} \times \pi_{2} \colon (Q_{1}\times Q_{2},J_{\sqrt{-1}}) \to (X_{P_{1}}\times X_{P_{2}}, J_{1} \times J_{2})$.
\end{center}
In \cite{GRANTCHAROV} an another approach to construct Hermitian structures on total spaces of a principal torus bundles by using $(1,0)$-connection with $(1,1)$-curvature is provided. Let us briefly describe this construction. Let $M$ be the total space of a principal $T^{2n}$-bundle over a Hermitian manifold $B$ with characteristic classes of $(1,1)$-type. By choosing a connection 
\begin{center}
$\Phi = {\text{diag}}\big \{\sqrt{-1}\eta_{1}, \ldots, \sqrt{-1}\eta_{2n} \big \} \in \Omega^{1}(M;{\text{Lie}}(T^{2n})),$
\end{center}
we have $d\eta_{j} = \pi^{\ast}\psi_{j}$, such that $\psi_{j} \in \Omega^{1,1}(B)$, $\forall j = 1,\ldots,2n$. From this, we can construct a complex structure $\mathscr{J}$ on $M$ by using the horizontal lift of the base complex structure on $\ker(\Phi)$ (horizontal space), since the vertical space is identified with the tangent space of an even-dimensional torus, we can set $\mathscr{J}(\eta_{2k-1}) = -\eta_{2k}$, $1 \leq k \leq n$. Thus, we have a well-defined almost complex structure $\mathscr{J} \in {\text{End}}(TM)$. It is straightforward to verify that $\mathscr{J}$ is in fact integrable, see \cite[Lemma 1]{GRANTCHAROV}. Notice that our convention \ref{complexconv} reflects a slight change of sign in the definition of $\mathscr{J} \in {\text{End}}(TM)$, cf. \cite{GRANTCHAROV}.

By considering a Hermitian metric $g_{B}$ on the base manifold $B$, we can use the connection $\Phi$ described above to define a Hermitian metric on $(M,\mathscr{J})$. In fact, we can set
\begin{equation}
\label{metrictotal}
g_{M} = \pi^{\ast}g_{B} + \sum_{i = 1}^{n} \big (\eta_{2i-1} \otimes \eta_{2i-1} + \eta_{2i} \otimes \eta_{2i} \big) .
\end{equation}
Since $\mathscr{J}(\eta_{2k-1}) = -\eta_{2k}$, it follows that the fundamental 2-form $\Omega_{M} = g_{M}(\mathscr{J} \otimes {\rm{id}})$ is given by
\begin{equation}
\Omega_{M} = \pi^{\ast} \omega_{B} + \sum_{i = 1}^{n} \eta_{2i-1} \wedge \eta_{2i}
\end{equation}
where $\omega_{B}$ is a the fundamental 2-form of $B$, here we consider $a \wedge b = a \otimes b - b \otimes a$, $\forall a,b \in \Omega^{1}(M)$.

\begin{remark}
It is worth pointing out that the KT structure provided by Morimoto's Hermitian structure in Example \ref{KTMorimoto} is a particular case of the KT structure described above. Actually, if we consider $B = X_{P_{1}} \times X_{P_{2}}$ and $M = Q_{1} \times Q_{2}$, such that $Q_{i} \in \mathscr{P}(X_{P_{i}},{\rm{U}}(1))$, $i = 1,2$, then we have $J_{\sqrt{-1}}(\xi_{1}) = \xi_{2}$, and
\begin{center}
$(\pi_{1} \times \pi_{2})_{\ast} \circ J_{\sqrt{-1}} = (\pi_{1} \times \pi_{2})_{\ast} \circ (\phi_{1} \times \phi_{2}) = (J_{1}\times J_{2}) \circ (\pi_{1} \times \pi_{2})_{\ast},$
\end{center}
see Theorem \ref{fundamentaltheorem} and Theorem \ref{main2}. Also, it is straightforward to verify that fundamental 2-form $\Omega_{Q_{1} \times Q_{2}} = g_{Q_{1} \times Q_{2}}(J_{\sqrt{-1}} \otimes {\rm{id}})$ satisfies
\begin{center}
$\Omega_{Q_{1} \times Q_{2}} = (\pi_{1} \times \pi_{2})^{\ast}(\omega_{X_{P_{1}}} \times \omega_{X_{P_{1}}}) + \eta_{1} \wedge \eta_{2},$
\end{center}
where $\omega_{X_{P_{i}}}$ is an invariant K\"{a}hler metric on $X_{P_{i}}$, $i = 1,2$.
\end{remark}

\begin{remark}
Notice that if the base manifold $B$ is compact, K\"{a}hler and the curvature $\Phi$ is integral, then by the Lefschetz theorem on $(1,1)$-classes, any such bundle can be obtained as the unitary frame bundle associated to a Whitney sum of holomorphic line bundles. Thus, in this last case we have $M = {\rm{U}}(E)$ such that
\begin{center}
$E = L_{1} \oplus \cdots \oplus L_{2n},$
\end{center}
where $L_{i} \in {\text{Pic}}(B)$, and $c_{1}(L_{i}) = -\big[\frac{\psi_{i}}{2\pi}\big]$, $1 \leq i \leq 2n$.
\end{remark}
Since we are interested in applications of the results of \cite{GRANTCHAROV} in principal torus bundle over complex flag manifolds, in what follows we assume that in the fibration $T^{2n} \hookrightarrow M \to B$ the base manifold $B$ is a compact Hermitian manifold.

Recall that, given a K\"{a}hler manifold $(B,\omega_{B})$, a differential form $\psi \in \Omega^{\bullet}(B)$ is called primitive if $\Lambda_{\omega_{B}}(\psi) = 0$, where $\Lambda_{\omega_{B}}$ is the dual of the Lefschetz operator $L_{\omega_{B}} \colon \bigwedge^{\bullet}(B) \to \bigwedge^{\bullet+2}(B)$, $L_{\omega_{B}}(\psi) = \psi \wedge \omega_{B}$, e.g. \cite{DANIEL}. By following \cite{GRANTCHAROV}, we have the following results.

\begin{proposition}[\cite{GRANTCHAROV}]
\label{CYTEinstein}
Suppose that $B$ is compact real $2m$-dimensional K\"{a}hler-Einstein manifold with positive scalar curvature. Let its scalar
curvature be normalized to be $2m^{2}$. Suppose that $M$ is an even-dimensional toric bundle with curvature $d\Phi = {\text{diag}}\{\sqrt{-1}d\eta_{1},\ldots,\sqrt{-1}d\eta_{2n}\}$, such that $d\eta_{1} = \pi^{\ast}\omega_{B}$, and for all $2 \leq k \leq 2n$, $d\eta_{k} = \pi^{\ast}\psi_{k}$, with $\psi_{k}$ being primitive; then $M$ admits a CYT structure.
\end{proposition}

\begin{remark}
\label{CYTcondition}
The proof of the last proposition follows from the following facts:

\begin{enumerate}

\item If $(B,\omega_{B})$ is compact and K\"{a}hler, then
\begin{equation}
\displaystyle {\text{Ric}}^{\nabla^{B}}(\Omega_{M}) = \pi^{\ast} \Big ( {\text{Ric}}^{\nabla}(\omega_{B}) - \sum_{k = 1}^{2n}\Lambda_{\omega_{B}}(\psi_{k})\psi_{k}\Big ),
\end{equation}
where $\nabla^{B}$ is the Bismut connection associated to the metric \ref{metrictotal}, and $\nabla$ is the Chern connection associated to $\omega_{B}$.
\item If $(B,\omega_{B})$ is compact and K\"{a}hler-Einstein with $c_{1}(B) > 0$, then by considering ${\text{Ric}}^{\nabla}(\omega_{B}) = m\omega_{B}$, it follows that 
\begin{equation}
\displaystyle {\text{Ric}}^{\nabla^{B}}(\Omega_{M}) = 0 \Longleftrightarrow \omega_{B} = \frac{1}{m} \sum_{k = 1}^{2n}\Lambda_{\omega_{B}}(\psi_{k})\psi_{k},
\end{equation}
see \cite[Proposition 5]{GRANTCHAROV} for more details.
\end{enumerate}

\end{remark}

\begin{corollary}[\cite{GRANTCHAROV}] 
\label{CYTcanonical}
Let $Q$ be the principal ${\rm{U}}(1)$-bundle of the maximum root of the canonical bundle of a compact K\"{a}hler-Einstein manifold $B$ with positive scalar curvature. Then $M = Q \times {\rm{U}}(1)$ admits a CYT structure.
\end{corollary}

\begin{remark}
It is worthwhile to point out that the statement of lemma above is slightly different from the statement in \cite{GRANTCHAROV}. In fact, keeping our conventions, we consider ${\text{Lie}}({\rm{U}}(1)) = \sqrt{-1}\mathbb{R}$ instead of ${\text{Lie}}({\rm{U}}(1)) = \mathbb{R}$, so we have ${\rm{e}}(Q) = -\frac{1}{I(B)}c_{1}(B)$, where $I(B)$ denotes the Fano index of $B$, cf. \cite{GRANTCHAROV}. Therefore, we need to consider the canonical bundle instead of the anti-canonical bundle.
\end{remark}

From Theorem \ref{fundamentaltheorem} and Corollary \ref{CYTcanonical}, we have the following result.

\begin{theorem}
\label{CYThomo}
Let $X_{P}$ be a complex flag manifold, with real dimension $2m$, associated to some parabolic Lie subgroup $P \subset G^{\mathbb{C}}$, and let $I(X_{P})$ be its Fano index. Then the manifold $M = Q(L) \times {\rm{U}}(1)$, such that $L = K_{X_{P}}^{\otimes \frac{\ell}{I(X_{P})}}$, $\ell > 0$, admits a CYT structure $(g_{M},\mathscr{J},\nabla^{B})$ whose fundamental form $\Omega_{M} = g_{M}(\mathscr{J} \otimes {\rm{id}})$ is given by
\begin{equation}
\label{CYTstructure}
\Omega_{M} = \frac{m\ell}{I(X_{P})} d\eta + \eta \wedge d\sigma,
\end{equation}
such that $\sqrt{-1} d\sigma \in \Omega^{1}({\rm{U}}(1);\sqrt{-1}\mathbb{R})$ is the Maurer-Cartan form, and (locally)
\begin{equation}
\eta =  \frac{1}{2I(X_{P})}d^{c}\log \Big (\big | \big |s_{U}v_{\delta_{P}}^{+} \big| \big |^{2} \Big ) + d\theta_{U},
\end{equation}
for some local section $s_{U} \colon U \subset X_{P} \to G^{\mathbb{C}}$, where $v_{\delta_{P}}^{+}$ denotes the highest weight vector of weight $\delta_{P}$ associated to the irreducible $\mathfrak{g}^{\mathbb{C}}$-module $V(\delta_{P})$.
\end{theorem}

\begin{proof}
The proof goes as follows. From Proposition \ref{C8S8.2Sub8.2.3Eq8.2.35} and Theorem \ref{fundamentaltheorem}, it follows that 
\begin{center}

$Q(K_{X_{P}}^{\otimes \frac{\ell}{I(X_{P})}}) = \displaystyle \sum_{\alpha \in \Sigma \backslash \Theta}Q \big ( \textstyle{- \ell \frac{ \langle \delta_{P},h_{\alpha}^{\vee} \rangle}{I(X_{P})}} \omega_{\alpha} \big).$
    
\end{center}
Therefore, from Proposition \ref{connection} we have a connection one-form on $Q(K_{X_{P}}^{\otimes \frac{\ell}{I(X_{P})}})$ defined by
\begin{equation}
\label{localconnection}
\eta' = \displaystyle \sum_{\alpha \in \Sigma \backslash \Theta}  \frac{ \ell \langle \delta_{P},h_{\alpha}^{\vee} \rangle}{2 I(X_{P})} ( \partial - \overline{\partial} \big ) \log \big (||s_{U} v_{\omega_{\alpha}}^{+}||^{2} \big )+ a_{U}^{-1}da_{U},    
\end{equation}
thus the associated contact structure is given by $\eta = - \sqrt{-1} \eta'$. If we consider $a_{U} = \mathrm{e}^{\sqrt{-1}\theta_{U}}$, where $\theta_{U}$ is real and is defined up to an integral multiple of $2 \pi$, by rearranging the expression above we obtain
\begin{equation}
\label{contactcanonical}
\eta = \displaystyle - \frac{\ell \sqrt{-1}}{2 I(X_{P})}\big ( \partial - \overline{\partial} \big )\log \Big ( \prod_{\alpha \in \Sigma \backslash \Theta} ||s_{U}v_{\omega_{\alpha}}^{+}||^{2  \langle \delta_{P},h_{\alpha}^{\vee} \rangle}\Big) + d\theta_{U}.    
\end{equation}
Now, we recall some basic facts about representation theory of simple Lie algebras \cite[p. 186]{PARABOLICTHEORY}:

\begin{enumerate}

    \item $V(\delta_{P}) \subset \bigotimes_{\alpha \in \Sigma \backslash \Theta} V(\omega_{\alpha})^{ \otimes \langle \delta_{P},h_{\alpha}^{\vee} \rangle}$; 
    
    \item $v_{\delta_{P}}^{+} = \bigotimes_{\alpha \in \Sigma \backslash \Theta}v_{\omega_{\alpha}}^{+ \otimes \langle \delta_{P},h_{\alpha}^{\vee} \rangle}$, where $v_{\omega_{\alpha}}^{+} \in V(\omega_{\alpha})$ is the highest weight vector of highest weight $\omega_{\alpha}$, $\forall \alpha \in \Sigma \backslash \Theta$.
    
\end{enumerate}
From these two facts, we can take a $G$-invariant inner product on $V(\delta_{P})$ induced from a $G$-invariant inner product $\langle \cdot, \cdot \rangle_{\alpha}$ on each factor $V(\omega_{\alpha})$, $\forall \alpha \in \Sigma \backslash \Theta$, such that
\begin{equation}
\label{innerinv}
\Big \langle  \bigotimes_{\alpha \in \Sigma \backslash \Theta} \big(v_{1}^{(\alpha)} \otimes \cdots \otimes v_{\langle \delta_{P},h_{\alpha}^{\vee} \rangle}^{(\alpha)}\big),  \bigotimes_{\alpha \in \Sigma \backslash \Theta} \big(w_{1}^{(\alpha)} \otimes \cdots \otimes w_{\langle \delta_{P},h_{\alpha}^{\vee} \rangle}^{(\alpha)}\big) \Big \rangle = \prod_{ \alpha \in \Sigma \backslash \Theta}\big \langle v_{1}^{(\alpha)},w_{1}^{(\alpha)} \big \rangle_{\alpha}\cdots \big \langle v_{\langle \delta_{P},h_{\alpha}^{\vee} \rangle}^{(\alpha)},w_{\langle \delta_{P},h_{\alpha}^{\vee} \rangle}^{(\alpha)} \big \rangle_{\alpha},
\end{equation}
Hence, by rearranging the expression \ref{contactcanonical}, we obtain obtain from the norm induced by the inner product above the following expression
\begin{equation}
\label{contactroot}
\eta =  \displaystyle \frac{\ell}{2I(X_{P})}d^{c}\log \Big (\big | \big |s_{U}v_{\delta_{P}}^{+} \big| \big |^{2} \Big ) + d\theta_{U}.
\end{equation}
From this, given $M = Q(K_{X_{P}}^{\otimes \frac{\ell}{I(X_{P})}}) \times {\rm{U}}(1)$ we have
\begin{center}
$TM = TQ(K_{X_{P}}^{\otimes \frac{\ell}{I(X_{P})}}) \oplus T{\rm{U}}(1),$
\end{center}
thus we can define a connection $\Phi \in \Omega^{1}(M;{\text{Lie}}(T^{2}))$ such that
\begin{center}
$\Phi = \begin{pmatrix}
 \sqrt{-1}\eta & 0 \\
 0 & \sqrt{-1}d\sigma
\end{pmatrix}.$
\end{center}
Now, by considering the complex structure $\mathscr{J} \in {\text{End}}(M)$ defined previously, see also \cite[Lemma 1]{GRANTCHAROV}, we can fix an invariant  K\"{a}hler-Einstein metric $\omega_{0} \in \Omega^{1,1}(X_{P})$, such that ${\text{Ric}}^{\nabla}(\omega_{0}) = \lambda \omega_{0}$, $\lambda > 0$, and consider the Hermitian metric on $M = Q(K_{X_{P}}^{\otimes \frac{\ell}{I(X_{P})}}) \times {\rm{U}}(1)$ defined by $g_{M} = \Omega_{M}(\rm{id}\otimes \mathscr{J})$, such that 
\begin{center}
$\Omega_{M} = \pi^{\ast}\omega_{0} + \eta \wedge d\sigma$.
\end{center}
Therefore, from Proposition \ref{CYTEinstein} and Remark \ref{CYTcondition}, in order to obtain a CYT structure we need to solve the equation
\begin{center}
 ${\text{Ric}}^{\nabla}(\omega_{0}) - \Lambda_{\omega_{0}}(\psi)\psi = 0,$
\end{center}
where $\psi \in \Omega^{1,1}(X_{P})$ satisfies $d\eta = \pi^{\ast}\psi$, notice that from \ref{contactroot} we have
\begin{equation}
\label{curvform}
\psi = \displaystyle \frac{\ell\sqrt{-1}}{I(X_{P})}\partial \overline{\partial} \log \Big (\big | \big |s_{U}v_{\delta_{P}}^{+} \big| \big |^{2} \Big ),
\end{equation}
thus $\frac{\psi}{2\pi} = \frac{\ell}{I(X_{P})}c_{1}(X_{P})$, see Equation \ref{Cherncanonical} and Equation \ref{localform}. Since ${\text{Ric}}^{\nabla}(\omega_{0}) = 2\pi c_{1}(X_{P})$, we obtain
\begin{equation}
\lambda \omega_{0} - \frac{\ell \Lambda_{\omega_{0}}(\psi)}{I(X_{P})}\lambda \omega_{0} = 0 \Longleftrightarrow 1 - \frac{\ell \Lambda_{\omega_{0}}(\psi)}{I(X_{P})} = 0.
\end{equation}
Thus, since $\psi = \frac{\ell \lambda}{I(X_{P})}\omega_{0} \Longrightarrow \Lambda_{\omega_{0}}(\psi) =  \frac{\ell \lambda m}{I(X_{P})}$, we can solve the equation above on the right side for the parameter $\lambda$. We obtain
\begin{equation}
\displaystyle \lambda = \frac{I(X_{P})^{2}}{\ell^{2}m},
\end{equation}
and the base metric $\omega_{0}$ is given by
\begin{equation}
\label{inducedbasic}
\omega_{0} = \frac{m\ell^{2}\sqrt{-1}}{I(X_{P})^{2}} \partial \overline{\partial} \log \Big (\big | \big |s_{U}v_{\delta_{P}}^{+} \big| \big |^{2} \Big ), 
\end{equation}
notice that $\psi = \frac{I(X_{P})}{m \ell} \omega_{0}$. Therefore, we have 
\begin{center}
$\Omega_{M} = \displaystyle \frac{m\ell}{I(X_{P})} d\eta + \eta \wedge d\sigma$,
\end{center}
such that $\eta$ is given by expression \ref{contactroot}, which concludes the proof.
\end{proof}

\begin{remark}
\label{complexHopf}
Notice that in the setting of Theorem \ref{CYThomo}, we have the following decomposition for the tangent space of $M = Q(L) \times {\rm{U}}(1)$
\begin{center}
$TM = \ker(\eta) \oplus \big \langle \xi = \frac{\partial }{\partial \theta} \big \rangle \oplus \big \langle \frac{\partial}{\partial \sigma} \big\rangle.$
\end{center}
Thus, by considering the associated normal almost contact structure $(\phi,\xi,\eta)$ and the induced ${\rm{CR}}$-structure on $Q(L)$, we can describe the complex structure $\mathscr{J}\in {\text{End}}(TM)$, as being
\begin{center}
$ \mathscr{J} =  \begin{pmatrix}
\ \mathcal{J}_{\phi} & 0 &  \ \ 0 \\
0 & 0 & -1 \\                  
0  & 1 &  \ \ 0
 \end{pmatrix},$
\end{center}
where ${\mathcal{J}}_{\phi} \colon \mathscr{D} \to \mathscr{D}$, such that $\mathscr{D} = \ker(\eta)$, see for instance \ref{complexCR}. It is worthwhile to observe that $\mathscr{J}(\eta) = -d\sigma$. Besides, notice that in terms of the structure tensors  $(\phi,\xi,\eta)$ we have
\begin{equation}
\mathscr{J}(Y) = \displaystyle \phi(Y) + \eta(Y)\frac{\partial}{\partial \sigma} \ \ {\text{and}} \ \  \displaystyle \mathscr{J}\Big (\frac{\partial}{\partial \sigma} \Big ) = - \xi,
\end{equation}
$\forall Y \in TM$ tangent to $Q(L)$, recall that ${\mathcal{J}}_{\phi} = \phi|_{\mathscr{D}}$, see \ref{complexCR}.
\end{remark}

\begin{remark}
Although we denote by $\sqrt{-1} d\sigma \in \Omega^{1}({\rm{U}}(1);\sqrt{-1}\mathbb{R})$ the Maurer-Cartan form on $S^{1}$, it is worthwhile to point out that the 1-form $d\sigma \in \Omega^{1}(S^{1})$ is not exact. In fact, we have $d\sigma$ obtained from the restriction of the well-known closed 1-form $\Omega_{0} \in  \Omega^{1}(\mathbb{C}^{\times})$ defined by
\begin{equation}
\Omega_{0} = \displaystyle {\text{Im}}\Bigg (\frac{dw}{w} \Bigg) = \frac{udv - vdu}{u^{2} + v^{2}},
\end{equation}
here we consider $S^{1} \subset \mathbb{C}^{\times}$, and coordinates $w = u + \sqrt{-1}v \in \mathbb{C}^{\times}$.
\end{remark}

\begin{definition}
A Hermitian manifold $(M,g,J)$ is called locally conformally K\"{a}hler (L.C.K.) if it satisfies one of the following equivalent conditions: 

\begin{enumerate}

\item There exists an open cover $\mathscr{U}$ of $M$ and a family of smooth functions $\{f_{U}\}_{U \in \mathscr{U}}$, $f_{U} \colon U \to \mathbb{R}$, such that each local metric
\begin{equation}
g_{U} = {\mathrm{e}}^{-f_{U}}g|_{U},
\end{equation}
is K\"{a}hlerian, $\forall U \in \mathscr{U}$.

\item There exists a globally defined closed $1$-form $\theta \in \Omega^{1}(M)$ such that 
\begin{equation}
d\Omega = \theta \wedge \Omega.
\end{equation}
\end{enumerate}
\end{definition}

\begin{remark}
\label{LocalLee}
Notice that the two conditions in the definition above tells us that
\begin{center}
$\theta |_{U} = df_{U}$,
\end{center}
$\forall U \in \mathscr{U}$, see \cite{Dragomir}. The closed $1$-form $\theta \in \Omega^{1}(M)$ which satisfies the second condition of the definition above is called the Lee form of a L.C.K. manifold $(M,g,J)$. It is worth mentioning that, if the Lee form $\theta$ of a L.C.K. manifold $(M,g,J)$ is exact, i.e., $\theta = df$, such that $f \in C^{\infty}(M)$, then have that $(M,{\rm{e}}^{-f}g,J)$ is K\"{a}hler. In what follows, unless otherwise stated, we will assume that $\theta$ is not exact, and $\theta \not \equiv 0$.
\end{remark}

An important subclass of L.C.K. manifolds is defined by the parallelism of the Lee form with respect to the Levi-Civita connection of $g$.

\begin{definition}
A L.C.K. manifold $(M,g,J)$ is called a Vaisman manifold if $\nabla \theta \equiv 0$, where $\nabla$ is the Levi-Civita connection of $g$.
\end{definition}

 We observe that, under the hypotheses of Theorem \ref{CYThomo}, it follows that $Q(L)$ is a Sasakian manifold, see for instance \cite{CONTACTCORREA} for an explicit description of the associated Sasaki structure. Hence, if we consider its Riemannian cone $ {\mathscr{C}}(Q(L)) = Q(L) \times \mathbb{R}^{+}$, by taking a real number $q \in \mathbb{R}$, $q > 0$, and considering an equivalence relation $\sim_{q}$ on ${\mathscr{C}}(Q(L))$ generated by $(x,t) \sim (x,qt)$, the quotient ${\mathscr{C}}(Q(L))/ \sim_{q}$ defines a Vaisman manifold \cite{Structurevaisman}, \cite{Imersionvaisman}, or generalized Hopf manifold \cite{Vaisman}, with the Gauduchon metric \cite{Gauduchon} provided by an isomorphism ${\mathscr{C}}(Q(L))/ \sim_{q} \cong Q(L) \times S^{1}$. Being more precise, we have the following result.
 
\begin{corollary}
\label{CYTVaisman}
The manifold $M = Q(L) \times {\rm{U}}(1)$, such that $L = K_{X_{P}}^{\otimes \frac{\ell}{I(X_{P})}}$, $\ell > 0$, with the Hermitian structure $(g_{M},\mathscr{J})$ provided in Theorem \ref{CYThomo}, is also a Vaisman manifold. 
\end{corollary}

\begin{proof}
The result follows from the following ideas. By considering the fundamental $2$-form $\Omega_{M}$ associated to $(g_{M},\mathscr{J})$, we obtain
\begin{equation}
d\Omega_{M} = d\eta \wedge d\sigma = \Big (- \frac{I(X_{P})}{m\ell}d\sigma \Big ) \wedge \Omega_{M} \Longrightarrow \theta = - \frac{I(X_{P})}{m\ell}d\sigma.
\end{equation}
Thus, we have $d\Omega_{M} = \theta \wedge \Omega_{M} $, with $d\theta = 0$. Now, by considering the Levi-Civita connection $\nabla$ associated to the metric $g_{M} = \Omega_{M}({\rm{id}} \otimes \mathscr{J})$, it follows from Koszul's formula that 
\begin{equation}
2g_{M}(\nabla_{X}A,Y) = d\theta(X,Y) + (\mathscr{L}_{A}g_{M})(X,Y),
\end{equation}
$\forall X,Y \in \Gamma(TM)$, where $\theta = g_{M}(\cdot,A)$. Since $d\theta = 0$ and $A = -\frac{m\ell}{I(X_{P})}\frac{\partial}{\partial \sigma}$, we have
\begin{equation}
(\mathscr{L}_{A}g_{M}) = \big(\mathscr{L}_{A}d\sigma \big) \otimes d\sigma + d\sigma \otimes \big (\mathscr{L}_{A}d\sigma \big) = 0.
\end{equation}
Hence, we have $\nabla \theta \equiv 0  \Longrightarrow (M,\mathscr{J},g_{M})$ is Vaisman.
\end{proof}

In what follows, we provide some concrete examples which illustrate the result of Theorem \ref{CYThomo}. 
\begin{remark}
Given a complex flag manifold $X_{P}$, associated to some parabolic Lie subgroup $P \subset G^{\mathbb{C}}$, for the sake of simplicity, in the next examples we shall consider the invariant K\"{a}hler-Einstein metric $\rho_{0} \in \Omega^{1,1}(X_{P})$ such that
\begin{equation}
\label{riccinorm}
\rho_{0} = \sqrt{-1}\partial \overline{\partial} \log \Big (\big | \big |s_{U}v_{\delta_{P}}^{+} \big| \big |^{2} \Big ).
\end{equation}
The K\"{a}hler-Einstein metric above satisfies ${\text{Ric}}^{\nabla}(\rho_{0}) = \rho_{0}$, see Equation \ref{Cherncanonical} and Equation \ref{localform}. Notice that $\omega_{0} = \frac{m\ell^{2}}{I(X_{P})^{2}}\rho_{0}$, where $\omega_{0}$ is defined by \ref{inducedbasic}.
\end{remark}
\begin{example}[Basic Model] 
\label{CYTbasic}
Let us describe how the result provided by Theorem \ref{CYThomo} can be applied in the basic model \ref{basiccase}. Since this basic model allows us to describe explicitly CYT structures for a huge class of examples, our treatment for this partricular case will be more detailed .

Let $X_{P_{\omega_{\alpha}}}$ be a complex flag manifold, with real dimension $2m$, defined by a maximal parabolic Lie subgroup $P_{\omega_{\alpha}} \subset G^{\mathbb{C}}$. In this case we have
\begin{center}
$I(X_{P_{\omega_{\alpha}}}) = \langle \delta_{P_{\omega_{\alpha}}},h_{\alpha}^{\vee} \rangle$ \ \ and \ \ $K_{X_{P_{\omega_{\alpha}}}}^{ \otimes \frac{1}{ \langle \delta_{P_{\omega_{\alpha}}},h_{\alpha}^{\vee} \rangle}} = L_{\chi_{\omega_{\alpha}}}^{-1}$.
\end{center}
Thus, by considering the K\"{a}hler-Einstein metric
\begin{center}
$\rho_{0} = \langle \delta_{P_{\omega_{\alpha}}},h_{\alpha}^{\vee} \rangle \sqrt{-1}\partial \overline{\partial} \log \Big (\big | \big |s_{U}v_{\omega_{\alpha}}^{+} \big| \big |^{2} \Big )$,
\end{center}
given $L = K_{X_{P_{\omega_{\alpha}}}}^{ \otimes \frac{\ell}{ \langle \delta_{P_{\omega_{\alpha}}},h_{\alpha}^{\vee} \rangle}}$, for some $\ell > 0$, we have a connection $\sqrt{-1}\eta \in \Omega^{1}(Q(L);\sqrt{-1}\mathbb{R})$ such that 
\begin{center}
$\eta =  \displaystyle \frac{\ell}{2}d^{c}\log \Big (\big | \big | s_{U}v_{\omega_{\alpha}}^{+} \big| \big |^{2} \Big ) + d\theta_{U},$
\end{center}
notice that, as in the proof of Theorem \ref{CYThomo}, we have
\begin{center}
$\displaystyle d\eta = \frac{\ell}{\langle \delta_{P_{\omega_{\alpha}}},h_{\alpha}^{\vee} \rangle} \pi^{\ast} \rho_{0}$ \ \ and \ \ $\displaystyle \omega_{0} = \frac{m\ell^{2}}{\langle \delta_{P_{\omega_{\alpha}}},h_{\alpha}^{\vee} \rangle^{2}}\rho_{0}.$
\end{center}
Hence, we obtain $\psi = \frac{\ell}{\langle \delta_{P_{\omega_{\alpha}}},h_{\alpha}^{\vee} \rangle} \rho_{0} \in \Omega^{1,1}(X_{P_{\omega_{\alpha}}})$, such that $d\eta = \pi^{\ast} \psi$, and also satisfying $\psi = \frac{\langle \delta_{P_{\omega_{\alpha}}},h_{\alpha}^{\vee} \rangle}{m \ell} \omega_{0}$, which in turn implies that
\begin{center}
${\text{Ric}}^{\nabla}(\omega_{0}) - \Lambda_{\omega_{0}}(\psi)\psi = 0.$
\end{center}
In fact, we have
\begin{center}
$\displaystyle {\text{Ric}}^{\nabla}(\omega_{0}) = {\text{Ric}}^{\nabla}\Big ( \frac{m\ell^{2}}{\langle \delta_{P_{\omega_{\alpha}}},h_{\alpha}^{\vee} \rangle^{2}}\rho_{0}\Big) = \frac{\langle \delta_{P_{\omega_{\alpha}}},h_{\alpha}^{\vee} \rangle^{2}}{m\ell^{2}} \omega_{0},$
\end{center}
and
\begin{center}
$\displaystyle \Lambda_{\omega_{0}}(\psi)\psi = \Lambda_{\omega_{0}}\Big (\frac{\langle \delta_{P_{\omega_{\alpha}}},h_{\alpha}^{\vee} \rangle}{m \ell} \omega_{0} \Big) \frac{\langle \delta_{P_{\omega_{\alpha}}},h_{\alpha}^{\vee} \rangle}{m \ell} \omega_{0} = \frac{\langle \delta_{P_{\omega_{\alpha}}},h_{\alpha}^{\vee} \rangle^{2}}{m\ell^{2}} \omega_{0}.$
\end{center}
Notice that in the last equation above we have used that $\Lambda_{\omega_{0}}(\omega_{0}) = m$. Therefore, it follows from Theorem \ref{CYThomo} and the computations above that
\begin{center}
$\displaystyle \Omega_{M} = \frac{m\ell^{2}}{\langle \delta_{P_{\omega_{\alpha}}},h_{\alpha}^{\vee} \rangle}\sqrt{-1}\partial \overline{\partial} \log \Big (\big | \big |s_{U}v_{\omega_{\alpha}}^{+} \big| \big |^{2} \Big ) + \Bigg (\displaystyle \frac{\ell}{2}d^{c}\log \Big (\big | \big | s_{U}v_{\omega_{\alpha}}^{+} \big| \big |^{2} \Big ) + d\theta_{U} \Bigg ) \wedge d\sigma,$
\end{center}
defines a CYT structure $(g_{M},\mathscr{J},\nabla^{B})$ on $M = Q(L)\times {\rm{U}}(1)$, such that $L = K_{X_{P_{\omega_{\alpha}}}}^{ \otimes \frac{\ell}{ \langle \delta_{P_{\omega_{\alpha}}},h_{\alpha}^{\vee} \rangle}}$, for some $\ell > 0$, where $g_{M} = \Omega_{M}(\rm{id}\otimes \mathscr{J})$, $\mathscr{J}$ is the complex structure described previously, see also Remark \ref{complexHopf}, and $\nabla^{B}$ is the associated Bismut connection.
\end{example}

\begin{example}[Hopf surface] 
\label{Hopfsurface}
A particular case of the previous example is given by the well-known Hopf surface $S^{3}\times S^{1}$. Consider $G^{\mathbb{C}} = {\rm{SL}}(2,\mathbb{C})$ and $P = B \subset {\rm{SL}}(2,\mathbb{C})$ as in Example \ref{exampleP1}. As we have seen, in this case we have

\begin{center}
    
$X_{B} = \mathbb{C}{\rm{P}}^{1}$ \ \ and \ \ $\mathscr{P}(\mathbb{C}{\rm{P}}^{1},{\rm{U}}(1)) = \mathbb{Z}\mathrm{e}(Q(\omega_{\alpha}))$,
    
\end{center}
where $Q(\omega_{\alpha}) = Q(\mathscr{O}(1))$. Since $I(\mathbb{C}{\rm{P}}^{1}) = 2$ and $K_{\mathbb{C}{\rm{P}}^{1}} = \mathscr{O}(-2)$, it follows that  $K_{\mathbb{C}{\rm{P}}^{1}}^{\otimes \frac{1}{2}} = \mathscr{O}(-1)$, thus

\begin{center}    
$Q(K_{\mathbb{C}{\rm{P}}^{1}}^{\otimes \frac{1}{2}}) = S^{3}$.
\end{center}
By following Example \ref{exampleP1}, we obtain 
\begin{center}
$\rho_{0} = \displaystyle  \langle \delta_{B},h_{\alpha}^{\vee} \rangle \sqrt{-1}\partial \overline{\partial} \log \Bigg ( \Big |\Big |\begin{pmatrix}
 1 & 0 \\
 z & 1
\end{pmatrix} e_{1} \Big| \Big|^{2} \Bigg ) = 2\sqrt{-1}\frac{dz \wedge d\overline{z}}{(1+|z|^{2})^{2}}.$
\end{center}
Moreover, in this case we have a connection $\sqrt{-1}\eta \in \Omega^{1}(S^{3};\sqrt{-1}\mathbb{R})$ such that 
\begin{center}
$\eta =  \displaystyle \frac{\sqrt{-1}}{2}\frac{zd\overline{z} - \overline{z}dz}{(1 + |z|^{2})} + d\theta_{U}.$
\end{center}
Therefore, similarly to the previous example, we obtain
\begin{center}
$\displaystyle \psi = \frac{\sqrt{-1}dz \wedge d\overline{z}}{(1+|z|^{2})^{2}}$ \ \ and \ \ $ \displaystyle \omega_{0} = \frac{\sqrt{-1}}{2}\frac{dz \wedge d\overline{z}}{(1+|z|^{2})^{2}}.$
\end{center}
such that $d\eta = \pi^{\ast}\psi$ and $\psi = 2\omega_{0}$. It is straightforward to verify that ${\text{Ric}}^{\nabla}(\omega_{0}) - \Lambda_{\omega_{0}}(\psi)\psi = 0.$ Hence, we obtain from Theorem \ref{CYThomo} that 
\begin{center}
$\displaystyle \Omega_{S^{3}\times S^{1}} = \frac{\sqrt{-1}}{2}\frac{dz \wedge d\overline{z}}{(1+|z|^{2})^{2}} + \Bigg ( \displaystyle \frac{\sqrt{-1}}{2}\frac{zd\overline{z} - \overline{z}dz}{(1 + |z|^{2})} + d\theta_{U}\Bigg ) \wedge d\sigma,$
\end{center}
defines a CYT structure $(g_{S^{3}\times S^{1}},\mathscr{J},\nabla^{B})$ on $S^{3}\times S^{1}$, where $g_{S^{3}\times S^{1}} = \Omega_{S^{3}\times S^{1}}(\rm{id}\otimes \mathscr{J})$, $\mathscr{J}$ is the complex structure described in Remark \ref{complexHopf} and $\nabla^{B}$ is the associated Bismut connection.
\end{example}

\begin{remark}
Notice that by a similar argument, and by following Example \ref{CYTbasic}, we can explicitly describe a CYT structure on $\mathbb{R}P^{3} \times S^{1}$, as a particular case of CYT structure on $(S^{3}/\mathbb{Z}_{\ell}) \times S^{1}$, $\ell >0$. 
\end{remark}

\begin{example}[Hopf manifold] We also have as a particular case of Example \ref{CYTbasic} the Hopf manifold $S^{2n+1}\times S^{1}$. As in Example \ref{examplePn}, consider $G^{\mathbb{C}} = {\rm{SL}}(n+1,\mathbb{C})$ and $P_{\omega_{\alpha_{1}}} \subset {\rm{SL}}(n+1,\mathbb{C})$. From this, we have 
\begin{center}
    
$X_{P_{\omega_{\alpha_{1}}}} = \mathbb{C}{\rm{P}}^{n},$ \ \ and \ \ $\mathscr{P}(\mathbb{C}{\rm{P}}^{n},{\rm{U}}(1)) = \mathbb{Z}\mathrm{e}(Q(\omega_{\alpha_{1}}))$,
    
\end{center}
where $Q(\omega_{\alpha_{1}}) = Q(\mathscr{O}(1))$. Since $I(\mathbb{C}{\rm{P}}^{n}) = n+1$ and $K_{\mathbb{C}{\rm{P}}^{n}} = \mathscr{O}(-n-1)$, it follows that  $K_{\mathbb{C}{\rm{P}}^{n}}^{\otimes \frac{1}{n+1}} = \mathscr{O}(-1)$, thus

\begin{center}    
$Q(K_{\mathbb{C}{\rm{P}}^{n}}^{\otimes \frac{1}{n+1}}) = S^{2n+1}$.
\end{center}
By following Example \ref{examplePn} and Example \ref{CYTbasic}, we obtain 
\begin{center}
$\rho_{0} =  \displaystyle (n+1)\sqrt{-1}\partial \overline{ \partial} \log \Big (1 + \sum_{l = 1}^{n}|z_{l}|^{2} \Big ),$
\end{center}
cf. \cite[p. 97 ]{Moroianu}. Moreover, in this case we have a connection $\sqrt{-1}\eta \in \Omega^{1}(S^{2n+1};\sqrt{-1}\mathbb{R})$ such that 
\begin{center}
$\eta = \displaystyle \frac{\sqrt{-1}}{2} \sum_{l = 1}^{n}\frac{z_{l}d\overline{z}_{l} - \overline{z}_{l}dz_{l}}{\big (1 + \sum_{l = 1}^{n}|z_{l}|^{2} \big )} + d\theta_{U},$
\end{center}
see \ref{localconnection} and Example \ref{COMPLEXHOPF}. Thus, similarly to the previous example, we obtain
\begin{center}
$\displaystyle \psi =\sqrt{-1}\partial \overline{ \partial} \log \Big (1 + \sum_{l = 1}^{n}|z_{l}|^{2} \Big )$ \ \ and \ \ $ \displaystyle \omega_{0} = \frac{n}{n+1}\sqrt{-1}\partial \overline{ \partial} \log \Big (1 + \sum_{l = 1}^{n}|z_{l}|^{2} \Big ),$
\end{center}
such that $d\eta = \pi^{\ast}\psi$, and $\psi = \frac{n+1}{n}\omega_{0}$. It is straightforward to verify that ${\text{Ric}}^{\nabla}(\omega_{0}) - \Lambda_{\omega_{0}}(\psi)\psi = 0.$
Hence, we have from Theorem \ref{CYThomo} that 
\begin{center}
$\displaystyle \Omega_{S^{2n+1}\times S^{1}} =  \frac{n}{n+1}\sqrt{-1}\partial \overline{ \partial} \log \Big (1 + \sum_{l = 1}^{n}|z_{l}|^{2} \Big ) + \Bigg ( \displaystyle\displaystyle \frac{\sqrt{-1}}{2} \sum_{l = 1}^{n}\frac{z_{l}d\overline{z}_{l} - \overline{z}_{l}dz_{l}}{\big (1 + \sum_{l = 1}^{n}|z_{l}|^{2} \big )} + d\theta_{U}\Bigg ) \wedge d\sigma,$
\end{center}
defines a CYT structure $(g_{S^{2n+1}\times S^{1}},\mathscr{J},\nabla^{B})$ on $S^{2n+1}\times S^{1}$, where $g_{S^{2n+1}\times S^{1}} = \Omega_{S^{2n+1}\times S^{1}}(\rm{id}\otimes \mathscr{J})$, $\mathscr{J}$ is the complex structure described in Remark \ref{complexHopf}, and $\nabla^{B}$ is the associated Bismut connection.
\end{example}

\begin{remark}
Similarly to Example \ref{Hopfsurface}, the ideas used in the example above can be applied to describe a CYT structure on $\mathbb{L}_{(n,\ell)} \times S^{1}$, where $\mathbb{L}_{(n,\ell)} = S^{2n+1}/ \mathbb{Z}_{\ell}$, $\forall \ell >0$, see Example \ref{COMPLEXHOPF}.
\end{remark}

\begin{example}[$\mathscr{V}_{2}(\mathbb{R}^{6}) \times S^{1}$] 
Consider $G^{\mathbb{C}} = {\rm{SL}}(4,\mathbb{C})$ and $P_{\omega_{\alpha_{2}}} \subset {\rm{SL}}(4,\mathbb{C})$ as in Example \ref{grassmanian}. In this case we have  $X_{P} = {\rm{Gr}}(2,\mathbb{C}^{4})$ (Klein quadric). Also, notice that
\begin{center}
    
${\text{Pic}}({\rm{Gr}}(2,\mathbb{C}^{4})) = \mathbb{Z}c_{1}(L_{\chi_{\alpha_{2}}})$.
    
\end{center}
Since $I({\rm{Gr}}(2,\mathbb{C}^{4})) = 4$ and $K_{{\rm{Gr}}(2,\mathbb{C}^{4})}^{\otimes \frac{1}{4}} = L_{\chi_{\alpha_{2}}}$, it follows that
\begin{center}    
$Q(K_{{\rm{Gr}}(2,\mathbb{C}^{4})}^{\otimes \frac{1}{4}}) = \mathscr{V}_{2}(\mathbb{R}^{6})$,
\end{center}
see Example \ref{STIEFEL}. Thus, from Equation \ref{C8S8.3Sub8.3.2Eq8.3.21} we obtain
\begin{center}
$\rho_{0} =  \displaystyle 4 \sqrt{-1}\partial \overline{\partial} \log \Bigg (1+ \sum_{k = 1}^{4}|z_{k}|^{2} + \bigg |\det \begin{pmatrix}
 z_{1} & z_{3} \\
 z_{2} & z_{4}
\end{pmatrix} \bigg |^{2} \Bigg)$.
\end{center}
Moreover, we have a connection $\sqrt{-1}\eta \in \Omega^{1}(\mathscr{V}_{2}(\mathbb{R}^{6});\sqrt{-1}\mathbb{R})$ such that 
\begin{center}
$\displaystyle \eta = \frac{1}{2}d^{c}\log \Bigg (1+ \sum_{k = 1}^{4}|z_{k}|^{2} + \bigg |\det \begin{pmatrix}
 z_{1} & z_{3} \\
 z_{2} & z_{4}
\end{pmatrix} \bigg |^{2} \Bigg) + d\theta_{U}.$
\end{center}
Therefore, similarly to the previous example, we obtain
\begin{center}
$\displaystyle \psi = \sqrt{-1}\partial \overline{\partial} \log \Bigg (1+ \sum_{k = 1}^{4}|z_{k}|^{2} + \bigg |\det \begin{pmatrix}
 z_{1} & z_{3} \\
 z_{2} & z_{4}
\end{pmatrix} \bigg |^{2} \Bigg)$ \ \ and \ \ $ \displaystyle \omega_{0} = \sqrt{-1}\partial \overline{\partial} \log \Bigg (1+ \sum_{k = 1}^{4}|z_{k}|^{2} + \bigg |\det \begin{pmatrix}
 z_{1} & z_{3} \\
 z_{2} & z_{4}
\end{pmatrix} \bigg |^{2} \Bigg),$
\end{center}
such that $d\eta = \pi^{\ast}\psi$, notice that $\psi = \omega_{0}$. Hence, we can verify that ${\text{Ric}}^{\nabla}(\omega_{0}) - \Lambda_{\omega_{0}}(\psi)\psi = 0.$ Thus, we obtain from Theorem \ref{CYThomo} that 

$$\Omega_{\mathscr{V}_{2}(\mathbb{R}^{6}) \times S^{1}} =  \displaystyle \sqrt{-1}\partial \overline{\partial} \log \Bigg (1+ \sum_{k = 1}^{4}|z_{k}|^{2} + \bigg |\det \begin{pmatrix}
 z_{1} & z_{3} \\
 z_{2} & z_{4}
\end{pmatrix} \bigg |^{2} \Bigg) \ + \ \Bigg ( \displaystyle\frac{1}{2}d^{c}\log \Bigg (1+ \sum_{k = 1}^{4}|z_{k}|^{2} + \bigg |\det \begin{pmatrix}
 z_{1} & z_{3} \\
 z_{2} & z_{4}
\end{pmatrix} \bigg |^{2} \Bigg) + d\theta_{U}\Bigg ) \wedge d\sigma,$$
defines a CYT structure $(g_{\mathscr{V}_{2}(\mathbb{R}^{6}) \times S^{1}},\mathscr{J},\nabla^{B})$ on $\mathscr{V}_{2}(\mathbb{R}^{6}) \times S^{1}$, where $g_{\mathscr{V}_{2}(\mathbb{R}^{6}) \times S^{1}} = \Omega_{\mathscr{V}_{2}(\mathbb{R}^{6}) \times S^{1}}(\rm{id}\otimes \mathscr{J})$, $\mathscr{J}$ is the complex structure described in Remark \ref{complexHopf}, and $\nabla^{B}$ is the associated Bismut connection.
\end{example}

\begin{remark}
As in the previous cases, the ideas applied in the last example also can be naturally generalized in order to obtain a CYT structure on $Q(K_{{\rm{Gr}}(k,\mathbb{C}^{n})}^{\otimes \frac{\ell}{I({\rm{Gr}}(k,\mathbb{C}^{n}))}}) \times S^{1}$, $\forall \ell > 0$, see Example \ref{Grassmanian}.
\end{remark}

Our next example illustrate by means of a concrete case how we can apply Theorem \ref{CYThomo} in order to obtain a CYT structure on certain principal $T^{2}$-bundles over full flag manifolds.

\begin{example}[$X_{1,1} \times S^{1}$] 
 Consider $G^{\mathbb{C}} = {\rm{SL}}(3,\mathbb{C})$, and $P_{\emptyset} = B \subset {\rm{SL}}(3,\mathbb{C})$ (Borel subgroup). As we have seen in Example \ref{ALOFFWALLACH}, in this case we have the Wallach flag manifold $X_{B} = W_{6}$ given by the quotient space

\begin{center}

$W_{6} = {\rm{SL}}(3,\mathbb{C})/B = {\rm{SU}}(3)/T^{2}.$

\end{center}
Moreover, we have the characterization $\mathscr{P}(W_{6},{\rm{U}}(1)) = \mathbb{Z}\mathrm{e}(Q(\omega_{\alpha_{1}})) \oplus \mathbb{Z}\mathrm{e}(Q(\omega_{\alpha_{2}})).$
A straightforward computation shows that $\delta_{B} = 2\omega_{\alpha_{1}} + 2\omega_{\alpha_{2}}$, and $I(W_{6}) = 2$. 
Hence, we obtain  

\begin{center}

$Q(K_{W_{6}}^{\otimes \frac{1}{I(W_{6})}}) = Q(-\omega_{\alpha_{1}}) + Q(-\omega_{\alpha_{2}}) = {\rm{SU}}(3)/{\rm{U}}(1) = X_{1,1},$

\end{center}
where $X_{1,1}$ is an example of Aloff-Wallach space, see Example \ref{ALOFFWALLACH}. By means of \ref{riccinorm}, we obtain 
\begin{center}
$\rho_{0} = \displaystyle 2\sqrt{-1}\partial \overline{\partial}\log \Bigg [ \bigg ( 1 + \sum_{i = 1}^{2}|w_{i}|^{2} \bigg ) \bigg (1 + |w_{3}|^{2} + \bigg | \det \begin{pmatrix}
w_{1} & 1  \\                  
w_{2}  & w_{3} 
 \end{pmatrix} \bigg |^{2} \bigg ) \Bigg ].$
\end{center}
Furthermore, in this case we have a connection $\sqrt{-1}\eta \in \Omega^{1}(X_{1,1};\sqrt{-1}\mathbb{R})$ such that 
\begin{center}
$$\eta = \displaystyle \frac{1}{2} d^{c} \log \Bigg [ \bigg ( 1 + \displaystyle \sum_{i = 1}^{2}|w_{i}|^{2} \bigg ) \bigg (1 + |w_{3}|^{2} + \bigg | \det \begin{pmatrix}
w_{1} & 1  \\                  
w_{2}  & w_{3} 
 \end{pmatrix} \bigg |^{2} \bigg ) \Bigg ] + d\theta_{U}.$$
\end{center}
Therefore, similarly to the previous examples, we obtain
\begin{center}
$\displaystyle \psi = \displaystyle \sqrt{-1}\partial \overline{\partial}\log \Bigg [ \bigg ( 1 + \sum_{i = 1}^{2}|w_{i}|^{2} \bigg ) \bigg (1 + |w_{3}|^{2} + \bigg | \det \begin{pmatrix}
w_{1} & 1  \\                  
w_{2}  & w_{3} 
 \end{pmatrix} \bigg |^{2} \bigg ) \Bigg ],$ 
\end{center} 
such that $d\eta = \pi^{\ast}\psi$, and 
\begin{center}
$ \displaystyle \omega_{0} = \frac{3}{2}\displaystyle \sqrt{-1}\partial \overline{\partial}\log \Bigg [ \bigg ( 1 + \sum_{i = 1}^{2}|w_{i}|^{2} \bigg ) \bigg (1 + |w_{3}|^{2} + \bigg | \det \begin{pmatrix}
w_{1} & 1  \\                  
w_{2}  & w_{3} 
 \end{pmatrix} \bigg |^{2} \bigg ) \Bigg ],$
\end{center}
notice that $\psi = \frac{2}{3}\omega_{0}$, where $m = 2$. It is straightforward to verify that ${\text{Ric}}^{\nabla}(\omega_{0}) - \Lambda_{\omega_{0}}(\psi)\psi = 0$. Thus, we obtain from Theorem \ref{CYThomo} that 

$\Omega_{X_{1,1} \times S^{1}} = \displaystyle \sqrt{-1}\partial \overline{\partial}\log \Bigg [ \bigg ( 1 + \sum_{i = 1}^{2}|w_{i}|^{2} \bigg ) \bigg (1 + |w_{3}|^{2} + \bigg | \det \begin{pmatrix}
w_{1} & 1  \\                  
w_{2}  & w_{3} 
 \end{pmatrix} \bigg |^{2} \bigg ) \Bigg ]$
 
$$+ \ \ \displaystyle \Bigg \{ \frac{1}{2} d^{c} \log \Bigg [ \bigg ( 1 + \displaystyle \sum_{i = 1}^{2}|w_{i}|^{2} \bigg ) \bigg (1 + |w_{3}|^{2} + \bigg | \det \begin{pmatrix}
w_{1} & 1  \\                  
w_{2}  & w_{3} 
 \end{pmatrix} \bigg |^{2} \bigg ) \Bigg ] + d\theta_{U} \Bigg \} \wedge d\sigma,$$  
defines a CYT structure $(g_{X_{1,1} \times S^{1}},\mathscr{J},\nabla^{B})$ on $X_{1,1} \times S^{1}$, where $g_{X_{1,1} \times S^{1}} = \Omega_{X_{1,1} \times S^{1}}(\rm{id}\otimes \mathscr{J})$, $\mathscr{J}$ is the complex structure described in Remark \ref{complexHopf} and $\nabla^{B}$ is the associated Bismut connection.
\end{example}

\begin{example}[Full flag manifolds $G/T$] A full flag manifold is defined as the homogeneous space given by $G/T$, where $G$ is a compact simple Lie group and $T \subset G$ is a maximal torus. By considering the root system $\Pi = \Pi^{+} \cup \Pi^{-}$ associated to the pair $(G,T)$ \cite{KNAPP}, from the complexification $G^{\mathbb{C}}$ of $G$ we have an identification
\begin{center}
$G/T \cong G^{\mathbb{C}}/B,$
\end{center}
where $B \subset G^{\mathbb{C}}$ is a Borel subgroup such that $B \cap G = T$. In this setting we have $\dim_{\mathbb{R}}(G/T) = 2|\Pi^{+}|$, and if we consider the principal circle bundle $Q(K_{G/T}) \to G/T$, we can apply Theorem \ref{CYThomo} and obtain an explicit CYT structure $(g_{M},\mathscr{J},\nabla^{B})$ on $M = Q(K_{G/T}) \times {\rm{U}}(1)$, such that 
\begin{equation}
{\text{Hol}}^{0}(\nabla^{B}) \subset {\rm{SU}} \big (|\Pi^{+}| + 1 \big ).
\end{equation}
An interesting feature of this particular example is that, if we denote by
\begin{equation}
\label{fullweight}
\displaystyle \varrho = \frac{1}{2} \sum_{\alpha \in \Pi^{+}}\alpha,
\end{equation}
it follows that $\delta_{B} = 2\varrho$, see \ref{fundweight}. Therefore, from Proposition \ref{C8S8.2Sub8.2.3Eq8.2.35} it follows that $K_{G/T} = L_{\chi_{2\varrho}}^{-1}$. Moreover, by using the notation \ref{notationsimple}, we can write $Q(K_{G/T}) = Q(-2\varrho)$, notice that in this case we have $I(G/T) = 2$ \cite[$\S$ 13.3]{Humphreys}. Also, we have that the CYT structure on $M = Q(-2\varrho) \times {\rm{U}}(1)$ is completely determined by 
\begin{center}
$\displaystyle \eta =  \frac{1}{4}d^{c}\log \Big (\big | \big |s_{U}v_{2\varrho}^{+} \big| \big |^{2} \Big ) + d\theta_{U},$
\end{center}
for some local section $s_{U} \colon U \subset G/T \to G^{\mathbb{C}}$, where $v_{2\varrho}^{+}$ denotes the highest weight vector of weight $2\varrho$ for the irreducible $\mathfrak{g}^{\mathbb{C}}$-module $V(2\varrho)$.
\end{example}

\subsection{Astheno-K\"{a}hler structures on products of compact homogeneous Sasaki manifolds} In this subsection we further explore the applications of our main results, i.e. Theorem \ref{fundamentaltheorem} and Theorem \ref{main2}. The main goal is providing an explicit description, in terms of elements of Lie theory, for astheno-K\"{a}hler structures by means of Tsukada's Hermitian structures on the product of two homogeneous Sasakian manifolds, e.g. \cite{Tsukada}, \cite{Matsuo1}, \cite{Matsuo2}. 

In what follows we keep the conventions and notations of the previous subsection.

\begin{definition}
\label{asthenodef}
A Hermitian manifold $M$ with complex dimension $m$ is said to be astheno-K\"{a}hler if its fundamental $2$-form $\Omega_{M}$ satisfies $dd^{c}\Omega_{M}^{m-2} = 0$, i.e., if 
\begin{equation}
\Omega_{M}^{m-2} = \underbrace{\Omega_{M} \wedge \ldots \wedge \Omega_{M}}_{(m-2){\text{-times}}}, 
\end{equation}
is pluriharmonic.
\end{definition}

In the setting of Theorem \ref{main2} we have $M = Q_{1} \times Q_{2}$, where $Q_{i} \in \mathscr{P}(X_{P_{i}},{\rm{U}}(1))$, $i = 1,2$, and we can also consider Morimoto's Hermitian structure $(\Omega_{M},J_{\sqrt{-1}})$ on $M$, see Example \ref{KTMorimoto}. Further, if we suppose that $\dim_{\mathbb{C}}(M) = 3$, the astheno-K\"{a}hler condition becomes
\begin{equation}
dd^{c}\Omega_{M} = 0.
\end{equation}
Thus, since $T_{B} = -d^{c}\Omega_{M} = J_{\sqrt{-1}}\Omega_{M}$, the astheno-K\"{a}hler condition is equivalent to the strong KT condition, see Remark \ref{KTdef}. From these last comments, it is straightforward to see that $S^{3}\times S^{3}$ endowed with Morimoto's Hermitian structure is an example of astheno-K\"{a}hler which is also strong $KT$, see \cite[Theorem 4.1]{Matsuo1}.
\begin{remark}
As we have seen in Theorem \ref{main2}, the complex structure $J_{\sqrt{-1}}$ is a particular case of a more general construction provided by \cite{Manjarin} and \cite{Tsukada}, see Proposition \ref{Manjarincomplex} and Remark \ref{Tsukadacomplex}. By following \cite{Tsukada}, \cite{Matsuo1}, \cite{Matsuo2}, we have similar results related to the construction of astheno-K\"{a}hler structures for more general Calabi-Eckmann manifolds of the form $S^{2n+1} \times S^{2m+1}$, $n,m >1$. 
\end{remark}
Let us briefly describe some general facts related to the ideas introduced in \cite{Tsukada}, and more recently applied in the setting of astheno-K\"{a}hler manifolds in \cite{Matsuo1}, \cite{Matsuo2}. 

Given an almost contact metric manifold $Q$ with structure tensors $(\phi, \xi,\eta,g)$, see \ref{contactmetric}, we say that $M$ is a \textit{contact metric} manifold if the structure tensors satisfy the compatibility conditions:
\begin{enumerate}

    \item $\eta \wedge (d\eta)^{n} \neq 0$, \  $\eta(\xi) = 1$,
    
    \item $\phi \circ \phi = - {\rm{id}} + \eta \otimes \xi$,
    
    \item $g(\phi \otimes \phi) = g - \eta \otimes \eta$,
    
    \item $d\eta = 2g(\phi \otimes {\rm{id}})$.
    
\end{enumerate}
Given a contact metric manifold $(Q,\phi, \xi,\eta,g)$, we say that $Q$ is a Sasaki manifold if and only if $\xi$ is a Killing vector field (i.e. $\mathscr{L}_{\xi}g = 0$) and the underlying almost contact structure $(\phi, \xi,\eta)$ satisfies
\begin{center}
$\big [ \phi,\phi \big ] + d\eta \otimes \xi = 0,$
\end{center}
i.e., if $(\phi, \xi,\eta)$ is a normal almost contact structure, see for instance \cite{BLAIR}.
\begin{remark}
An alternative characterization of Sasaki manifolds obtained directly from almost contact metric manifolds is the following: an almost contact metric manifold $(Q,\phi, \xi,\eta,g)$ is a Sasaki manifold if and only if 
\begin{equation}
(\nabla_{X}\phi)Y = g(X,Y)\xi - \eta(Y)X,
\end{equation}
$\forall X,Y \in TQ$, such that $\nabla$ is the associated Levi-Civita connection, e.g. \cite[Theorem 6.3]{BLAIR}.
\end{remark}
Given two Sasaki manifolds $Q_{1}$ and $Q_{2}$, with structure tensors $(\phi_{i}, \xi_{i},\eta_{i},g_{i})$, $i = 1,2$, we can consider the 1-parametric family of complex structures given by Tsukada's \cite{Tsukada} complex structures $J_{a,b} \in  {\text{End}}(T(Q_{1}\times Q_{2}))$, $a + \sqrt{-1}b \in \mathbb{C} \backslash \mathbb{R}$, by setting
\begin{equation}
\displaystyle J_{a,b} = \phi_{1} - \bigg ( \frac{a}{b}\eta_{1} + \frac{a^{2} + b^{2}}{b} \eta_{2}\bigg) \otimes \xi_{1} + \phi_{2} + \bigg ( \frac{1}{b}\eta_{1} + \frac{a}{b} \eta_{2}\bigg) \otimes \xi_{2}.
\end{equation}
Notice that, as we have seen in Remark \ref{Tsukadacomplex}, up to a suitable change, the complex structures described above are the same as in Proposition \ref{Manjarincomplex}. For this particular case, we can consider also the 1-parametric family of Hermitian metrics on the complex manifold $(Q_{1} \times Q_{2},J_{a,b})$ introduced by Tsukada \cite{Tsukada}, see also \cite{Matsuo2}, which is defined by
\begin{equation}
g_{a,b} = g_{1} + g_{2} + a\big(\eta_{1} \otimes \eta_{2} + \eta_{2} \otimes \eta_{1} \big) + \big(a^{2} + b^{2} - 1 \big)\eta_{2} \otimes \eta_{2}.
\end{equation}
From the compatibility conditions of the underlying contact metric structures  $(\phi_{i}, \xi_{i},\eta_{i},g_{i})$, $i = 1,2$, we have that the fundamental $2$-form $\Omega_{a,b} = g_{a,b}(J_{a,b}\otimes {\rm{id}})$ associated to the Hermitian manifold $(Q_{1} \times Q_{2},J_{a,b},g_{a,b})$ is given by
\begin{equation}
\Omega_{a,b} = \frac{1}{2}\Big (d\eta_{1} + d\eta_{2} \Big ) + b\eta_{1} \wedge \eta_{2}.
\end{equation}
\begin{remark}
It is worthwhile to observe that, in the setting above, if we consider the $1$-parametric family of complex structures $J_{\tau} \in {\text{End}}(T(M_{1}\times M_{2}))$, with $\tau = a +\sqrt{-1}b \in \mathbb{C} \backslash \mathbb{R}$, as in Proposition \ref{Manjarincomplex}, we can also obtain a 1-parametric family of Hermitian metrics $g_{\tau}$. In fact, the complex structures $J_{\tau}$ can be expressed in terms of structure tensors $(\phi,\xi,\eta)$ by
\begin{equation}
 J_{\tau} = \phi_{1} - \bigg ( \frac{a}{b}\eta_{1} + \frac{1}{b} \eta_{2}\bigg) \otimes \xi_{1} + \phi_{2} + \bigg (\frac{a^{2} + b^{2}}{b}\eta_{1} + \frac{a}{b} \eta_{2}\bigg) \otimes \xi_{2}.
\end{equation}
From this, we have associated to the complex structure above a Hermitian metric $g_{\tau}$ such that 
\begin{equation}
g_{\tau} = g_{1} + g_{2} + a \big(\eta_{1} \otimes \eta_{2} + \eta_{2} \otimes \eta_{1} \big) + \big (a^{2} + b^{2} - 1 \big)\eta_{1} \otimes \eta_{1}.
\end{equation}
It is straightforward to verify that $\Omega_{\tau} = g_{\tau}(J_{\tau} \otimes {\rm{id}})$ is given by
\begin{center}
$\Omega_{\tau} = \displaystyle \frac{1}{2}\Big (d\eta_{1} + d\eta_{2} \Big ) + b\eta_{1} \wedge \eta_{2}.$
\end{center}
Therefore, we obtain a precise correspondence between Tsukada's Hermitian structures and the Hermitian structures associated to the complex structured described in Proposition \ref{Manjarincomplex} (and Theorem \ref{main2}) on products of Sasaki manifolds.
\end{remark}
By considering the Hermitian manifold $(Q_{1}\times Q_{2},J_{a,b},\Omega_{a,b})$, $a + \sqrt{-1}b \in \mathbb{C} \backslash \mathbb{R}$, described above, from \cite{Matsuo2} we have the following result. 
\begin{theorem}
\label{asthenotheo}
Let $Q_{i}$ be a $(2m_{i} + 1)$-dimensional Sasaki manifold with structure tensors $(\phi_{i}, \xi_{i},\eta_{i},g_{i})$, $i = 1,2$, and $m_{1} + m_{2} + 1 > 3$. Then the Hermitian manifold $(Q_{1}\times Q_{2},J_{a,b},\Omega_{a,b})$, $a + \sqrt{-1}b \in \mathbb{C} \backslash \mathbb{R}$, is astheno-K\"{a}hler if and only if the real constants $a,b \in \mathbb{R}$ satisfy
\begin{equation}
m_{1}\big(m_{1} - 1\big) + 2am_{1}m_{2} + m_{2}\big(m_{2} - 1\big)\big(a^{2} + b^{2}\big) = 0.
\end{equation}
\end{theorem}

\begin{remark}
Even though we are following some different conventions as in Remark \ref{complexconv}, up to some changes of signs, the proof of the last theorem remains essentially the same as the proof presented in \cite[Theorem 4.1]{Matsuo2}.
\end{remark}

\begin{definition}
A Sasakian manifold $Q$ with structure tensors $(\phi,\xi,\eta,g)$ is said to be homogeneous if there is a connected Lie group $G$ acting transitively and effectively as a group of isometries on $Q$ preserving the Sasakian structure.
\end{definition}

\begin{remark}
In order to study the application of Theorem \ref{asthenotheo} in the homogeneous setting, we observe that for a homogeneous Sasaki manifold $M$ with structure tensors $(\phi,\xi,\eta,g)$, denoting by $G$ the connected Lie group acting on $Q$, since the Lie group preserves the structure tensors $(\phi,\xi,\eta,g)$, in particular it preserves the contact structure $\eta \in \Omega^{1}(Q)$. Therefore, the underlying contact manifold $(Q,\eta)$ is homogeneous. 
\end{remark}

As we have seen in the remark above, if we suppose that a compact connected Lie group acts on a compact Sasaki manifold $Q$ in such a way that the structure tensors $(\phi,\xi,\eta,g)$ are preserved, it follows from Theorem \ref{BWhomo} and Theorem \ref{contacthomo} that $Q = Q(L)$, where $Q(L)$ is a principal circle bundle over a complex flag manifold $X_{P}$, and $L^{-1} \in {\text{Pic}}(X_{P})$ is a very ample holomorphic line bundle, see Proposition \ref{CONTACTSIMPLY}.

By following \cite{CONTACTCORREA} and Theorem \ref{fundamentaltheorem}, given a compact homogeneous Sasaki manifold $Q = Q(L)$, its structure tensors $(\phi,\xi,\eta,g)$ can be described by means of the Cartan-Ehresmann connection $\sqrt{-1}\eta \in \Omega^{1}(Q(L);\sqrt{-1}\mathbb{R})$, such that 
\begin{equation}
\label{contactveryample}
\eta = \frac{1}{2}d^{c} \log \Big ( \prod_{\alpha \in \Sigma \backslash \Theta} ||s_{U}v_{\omega_{\alpha}}^{+}||^{2  \ell_{\alpha}}\Big) + d\theta_{U},
\end{equation}
where $P = P_{\Theta}$ and $\ell_{\alpha} \in \mathbb{Z}_{>0}$, $\forall \alpha \in \Sigma \backslash \Theta$. We notice that, for $L^{-1} \in {\text{Pic}}(X_{P})$ as above, we have
\begin{equation}
L^{-1} = \bigotimes_{\alpha \in \Sigma \backslash \Theta} L_{\chi_{\alpha}}^{\otimes \ell_{\alpha}},
\end{equation}
such that $\ell_{\alpha} \in \mathbb{Z}_{>0}$, $\forall \alpha \in \Sigma \backslash \Theta$. Thus, if we consider the weight $\lambda(L) \in  \Lambda_{\mathbb{Z}_{\geq 0}}^{\ast}$ defined by
\begin{equation}
\lambda(L) = \sum_{\alpha \in \Sigma \backslash \Theta} \ell_{\alpha} \omega_{\alpha},
\end{equation}
we can associate to $M = Q(L)$ an irreducible $\mathfrak{g}^{\mathbb{C}}$-module $V(\lambda(L))$ with highest weight vector $v_{\lambda(L)}^{+} \in V(\lambda(L))$. Since we have
\begin{equation}
\displaystyle V(\lambda(L)) \subset \bigotimes_{\alpha \in \Sigma \backslash \Theta}V(\omega_{\alpha})^{\otimes \ell_{\alpha}} \ \ {\text{and}} \ \ \displaystyle v_{\lambda(L)}^{+} = \bigotimes_{\alpha \in \Sigma \backslash \Theta}(v_{\omega_{\alpha}}^{+})^{\otimes \ell_{\alpha}},
\end{equation}
see for instance \cite[p. 186]{PARABOLICTHEORY}, we can take a $G$-invariant inner product on $V(\lambda(L))$ induced from a $G$-invariant inner product $\langle \cdot, \cdot \rangle_{\alpha}$ on each factor $V(\omega_{\alpha})$, $\forall \alpha \in \Sigma \backslash \Theta$, such that
\begin{equation}
\Big \langle  \bigotimes_{\alpha \in \Sigma \backslash \Theta} \big(v_{1}^{(\alpha)} \otimes \cdots \otimes v_{\ell_{\alpha}}^{(\alpha)}\big),  \bigotimes_{\alpha \in \Sigma \backslash \Theta} \big(w_{1}^{(\alpha)} \otimes \cdots \otimes w_{\ell_{\alpha}}^{(\alpha)}\big) \Big \rangle = \prod_{ \alpha \in \Sigma \backslash \Theta}\big \langle v_{1}^{(\alpha)},w_{1}^{(\alpha)} \big \rangle_{\alpha}\cdots \big \langle v_{\ell_{\alpha}}^{(\alpha)},w_{\ell_{\alpha}}^{(\alpha)} \big \rangle_{\alpha},
\end{equation}
By considering the norm $||\cdot||$ on $V(\lambda(L))$ induced from the inner product above, we can rewrite \ref{contactveryample} as 
\begin{equation}
\eta = \frac{1}{2}d^{c} \log \Big ( ||s_{U}v_{\lambda(L)}^{+}||^{2}\Big) + d\theta_{U},
\end{equation}
for some local section $s_{U} \colon U \subset X_{P} \to G^{\mathbb{C}}$.
\begin{remark}
It is worthwhile to observe that the ample line bundle $L^{-1} \in {\text{Pic}}(X_{P})$ associated to $M = Q(L)$ is in fact very ample, i.e., we have a projective embedding 
\begin{center}
$G^{\mathbb{C}}/P \hookrightarrow \mathbb{P}(V(\lambda(L))) = {\text{Proj}}\big (H^{0}(X_{P},L^{-1})^{\ast} \big),$
\end{center}
see for instance \cite[Page 193]{Flaginterplay}, \cite[Theorem 3.2.8]{PARABOLICTHEORY}, \cite{TAYLOR}.
\end{remark}
Now, from Theorem \ref{asthenotheo}, Theorem \ref{main2} and the ideas introduced in \cite{CONTACTCORREA}, we have the following result in the homogeneous setting.
\begin{theorem}
\label{asthenohomo}
Let $Q_{i} = Q(L_{i})$ be a compact homogeneous Sasaki manifold with structure tensors $(\phi_{i},\xi_{i},\eta_{i},g_{i})$, such that $L_{i}^{-1} \in {\text{Pic}}(X_{P_{i}})$ is an ample line bundle, for some $P_{i} = P_{\Theta_{i}} \subset G_{i}^{\mathbb{C}}$, $i = 1,2$. Then we have that:
\begin{enumerate}

\item The manifold $M = Q(L_{1}) \times Q(L_{2})$ admits a $1$-parametric family of Hermitian structures $(\Omega_{a,b},J_{a,b})$, $a + \sqrt{-1}b \in \mathbb{C} \backslash \mathbb{R}$, completely determined by principal $S^{1}$-connections $\sqrt{-1}\eta_{i} \in \Omega^{1}(Q(L_{i});\sqrt{-1}\mathbb{R})$, $i = 1,2$, such that 
\begin{equation}
\displaystyle \eta_{i} = \frac{1}{2}d^{c} \log \Big ( ||s_{U_{i}}v_{\lambda(L_{i})}^{+}||^{2}\Big) + d\theta_{U_{i}},
\end{equation}
for some local section $s_{U_{i}} \colon U_{i} \subset X_{P_{i}} \to G_{i}^{\mathbb{C}}$, where $v_{\lambda(L_{i})}^{+}$ is the highest weight vector associated to an irreducible $\mathfrak{g}^{\mathbb{C}}$-module $V(\lambda(L_{i}))$, $i = 1,2$; 

\item Moreover, the Hermitian structure $(\Omega_{a,b},J_{a,b})$ is astheno-K\"{a}hler if and only if the real constants $a,b \in \mathbb{R}$ satisfy 
\begin{equation}
\label{athenocond}
m_{\Theta_{1}}\big(m_{\Theta_{1}} - 1\big) + 2am_{\Theta_{1}}m_{\Theta_{2}} + m_{\Theta_{2}}\big(m_{\Theta_{2}} - 1\big)\big(a^{2} + b^{2}\big) = 0,
\end{equation}
where $m_{\Theta_{i}} =  |\Pi_{i}^{+} \backslash \langle \Theta_{i} \rangle^{+}|$, $i = 1,2$.
\end{enumerate}

\end{theorem}
\begin{proof} The proof is a consequence of some results covered previously. The proof of item (1) follows from Theorem \ref{fundamentaltheorem} and Theorem \ref{main2}. In fact, from Theorem \ref{fundamentaltheorem} we have an explicit description for the structure tensors $(\phi_{i},\xi_{i},\eta_{i},g_{i})$ on $Q_{i} = Q(L_{i})$ in terms of the Cartan-Ehresmann connection $\sqrt{-1}\eta_{i} \in \Omega^{1}(Q(L_{i});\sqrt{-1}\mathbb{R})$, $i = 1,2$, induced by
\begin{center}
$\displaystyle \eta_{i} = \frac{1}{2}d^{c} \log \Big ( ||s_{U_{i}}v_{\lambda(L_{i})}^{+}||^{2}\Big) + d\theta_{U_{i}}.$
\end{center}
Therefore, we obtain from Theorem \ref{main2} a concrete description for Tsukada's \cite{Tsukada} Hermitian structure $(\Omega_{a,b},J_{a,b})$, $a + \sqrt{-1}b \in \mathbb{C} \backslash \mathbb{R}$, on $Q(L_{1}) \times Q(L_{2})$. The proof of item (2) follows from Theorem \ref{asthenotheo}, and from the fact that for $P = P_{\Theta}$ we have ${\text{Lie}}(P_{\Theta}) = \mathfrak{p}_{\Theta} \subset \mathfrak{g}^{\mathbb{C}}$ given by
\begin{center}
$\mathfrak{p}_{\Theta} = \mathfrak{n}^{+} \oplus \mathfrak{h} \oplus \mathfrak{n}(\Theta)^{-},$ \ with \ $\mathfrak{n}(\Theta)^{-} = \displaystyle \sum_{\alpha \in \langle \Theta \rangle^{-}} \mathfrak{g}_{\alpha}$,
\end{center}
which implies that $\dim_{\mathbb{C}}(X_{P}) = |\big ( \Pi_{i}^{+} \backslash \langle \Theta_{i} \rangle^{+}|$. 
\end{proof}

\begin{remark}
We notice that the last theorem provides a constructive method to produce examples of astheno-K\"{a}hler manifolds in such a way that the underlying Hermitian structure can be completely determined by using elements of Lie theory such as painted Dynkin graphs \cite{Alekseevsky} and representation theory of simple Lie algebras. 
\end{remark}

\begin{remark}
\label{gauduchonatheno}
As we have mentioned in the introduction, from Theorem \ref{asthenohomo} we obtain a huge class of examples of compact non-K\"{a}hler Hermitian manifolds which can be used to illustrate the solution of Gauduchon's conjecture \cite{Gauduchon}. Being more precise, given an astheno-K\"{a}hler manifold $\big (M,\Omega_{a,b},J_{a,b} \big )$ as in Theorem \ref{asthenohomo}, and fixing the Gauduchon metric $\Omega_{0}$ (cf. \cite{Gauduchon}) in the conformal class of $\Omega_{a,b}$, for every closed real $(1,1)$-form $\psi \in c_{1}^{BC}(M)$, we have 
\begin{center}
${\text{Ric}}(\Omega_{a,b}) = \psi + \sqrt{-1}\partial \overline{\partial}F$,
\end{center}
for some $F \in C^{\infty}(M)$. From this, it follows from \cite[Corollary 1.2]{Tosatti} that there exists a unique constant $A$, and a unique Gauduchon metric $\Omega_{u} \in \Omega^{2}(M)$ satisfying
\begin{equation}
\Omega_{u}^{n-1} = \Omega_{0}^{n-1} + \sqrt{-1}\partial \overline{\partial} u \wedge \Omega_{a,b}^{n-2}
\end{equation}
for some smooth function $u$, solving the Calabi-Yau equation
\begin{equation}
\Omega_{u}^{n} = {\rm{e}}^{F + A}\Omega_{a,b}^{n}.
\end{equation}
Therefore, applying $-\sqrt{-1} \partial \overline{\partial}\log $ in the both sides of the equation above, it follows that 
\begin{equation}
{\text{Ric}}(\Omega_{u}) = {\text{Ric}}(\Omega_{a,b}) -  \sqrt{-1}\partial \overline{\partial}F = \psi.
\end{equation}
The computation above provides a brief illustration of how the result of Theorem \ref{asthenohomo} can be used as a source of examples in the setting of Gauduchon's conjecture. For more details about Gauduchon's conjecture, as well as its solution, see \cite{STW}, \cite{Tosatti}, \cite{Gauduchon}.
\end{remark}

\end{document}